\documentclass[twoside,reqno]{amsart}
\usepackage{amssymb,amsmath,amsfonts,amsthm}
\usepackage{graphicx,graphics,epsfig,psfrag,calrsfs,enumerate,array}
\usepackage[colorlinks=true]{hyperref}

\newtheorem{definition}{Definition}
\newtheorem{remark}{Remark}

\newtheorem{lemma}{Lemma}
\newtheorem{proposition}{Proposition}
\newtheorem{theorem}{Theorem}
\newtheorem{corollary}{Corollary}

\newcommand{\N}{{\mathbb N}}
\newcommand{\Z}{{\mathbb Z}}
\newcommand{\R}{{\mathbb R}}
\newcommand{\C}{{\mathbb C}}
\newcommand{\T}{{\mathbb T}}
\newcommand{\eps}{{\varepsilon}}
\newcommand{\op}{\text{\rm Op}}
\newcommand{\optilde}{\widetilde{\text{\rm Op}}}
\newcommand{\opeg}{{\text{\rm Op}^{\varepsilon,\gamma}}}
\newcommand{\opteg}{\optilde^{\varepsilon,\gamma}}
\newcommand{\Ng}{| \! | \! |}
\newcommand{\Nd}{| \! | \! |}

\title[Singular pseudodifferential calculus]{Singular pseudodifferential calculus\\ for wavetrains and pulses}

\author[jean-fran\c{c}ois coulombel, olivier gu\`es \& mark williams]{}


\email{jean-francois.coulombel@univ-nantes.fr}
\email{gues@cmi.univ-mrs.fr}
\email{williams@email.unc.edu}

\thanks{Research of J.-F. C. and O. G. was supported by the French Agence Nationale de la Recherche,
contract ANR-08-JCJC-0132-01. Research of M. W. was partially supported by NSF grants number DMS-0701201
and DMS-1001616.}

\begin{document}
\maketitle
\centerline{\scshape Jean-Fran\c{c}ois Coulombel}
\smallskip
{\footnotesize
 \centerline{CNRS, Laboratoire de Math\'ematiques Jean Leray (UMR CNRS 6629)}
 \centerline{Universit\'e de Nantes, 2 rue de la Houssini\`ere, BP 92208, 44322 Nantes Cedex 3, France}
}
\medskip

\centerline{\scshape Olivier Gu\`es}
\smallskip
{\footnotesize
 \centerline{Universit\'e de Provence, Laboratoire d'Analyse, Topologie et Probabilit\'es (UMR CNRS 6632)}
 \centerline{Technop\^ole Ch\^ateau-Gombert, 39 rue F. Joliot Curie, 13453 Marseille Cedex 13, France}
}
\medskip

\centerline{\scshape Mark Williams}
\smallskip
{\footnotesize
 \centerline{Department of Mathematics, University of North Carolina, Chapel Hill}
 \centerline{North Carolina 27599, USA}
}
\bigskip

\begin{abstract}
We generalize the analysis of \cite{williams3} and develop a singular pseudodifferential calculus.
The symbols that we consider do not satisfy the standard decay with respect to the frequency variables.
We thus adopt a strategy based on the Calder\'on-Vaillancourt Theorem. The remainders in the symbolic
calculus are bounded operators on $L^2$, whose norm is measured with respect to some small parameter.
Our main improvement with respect to \cite{williams3} consists in showing a regularization effect for
the remainders. Due to a nonstandard decay in the frequency variables, the regularization takes place
in a scale of anisotropic, and singular, Sobolev spaces. Our analysis allows to extend the results of
\cite{williams3} on the existence of highly oscillatory solutions to nonlinear hyperbolic problems by
dropping the compact support condition on the data. The results are also used in our companion work
\cite{cgw3} to justify nonlinear geometric optics with boundary amplification, which corresponds
to a more singular regime than the one considered in \cite{williams3}. The analysis is carried out
with either an additional real or periodic variable in order to cover problems for pulses or wavetrains
in geometric optics.
\end{abstract}

\tableofcontents

\section{Introduction}
\label{sect1}

Nonlinear geometric optics is devoted to the construction and asymptotic analysis of highly oscillatory 
solutions to some partial differential equations. In the context of nonlinear hyperbolic partial differential 
equations, one of the main issues is to prove existence of a solution to the highly oscillatory problem on 
a time interval that is independent of the (small) wavelength. Such uniform existence results cannot follow 
from a naive application of a standard existence result in some functional space, say a Sobolev space $H^s$, 
because the sequence of data does not remain in a fixed ball of $H^s$. A now classical procedure for proving 
uniform existence results is to work on singular problems with additional variables and to prove uniform energy 
estimates with respect to the singular parameter. This strategy was used in \cite{jmr} for the hyperbolic 
Cauchy problem and adapted in \cite{williams3} to hyperbolic initial boundary value problems. While energy 
estimates in \cite{jmr} relied on symmetry assumptions and integration by parts, those in \cite{williams3} are 
much more difficult to obtain and rely on a suitable singular pseudodifferential calculus.  The operators are 
pseudodifferential in the singular derivative $\partial_x +\beta \, \partial_\theta /\eps$. This calculus is 
well-adapted to boundary value problems that satisfy a maximal energy estimate, that is an $L^2$ estimate 
with no loss derivative. In particular, remainders in the calculus of \cite{williams3} are bounded operators on 
$L^2$ whose norm is controlled with respect to some parameter $\gamma$. This calculus is adapted to the 
situation studied in \cite{williams3} because such terms of order $0$ can be absorbed in the energy estimates 
by choosing $\gamma$ large enough.

In \cite{jfcog}, two of the authors have studied and justified geometric optics expansions with an amplification 
phenomenon for a certain class of {\it linear} hyperbolic boundary value problems. For linear problems, uniform 
existence is no source for concern. In the companion article \cite{cgw3}, we extend the result of \cite{jfcog} 
to {\it semilinear} problems. One major issue in \cite{cgw3} is to prove that the amplification phenomenon 
exhibited in \cite{jfcog} combined with the nonlinearity of the zero order term does not rule out existence 
of a solution on a fixed time interval. Our strategy in \cite{cgw3} is to study a singular problem for which 
we need to prove uniform estimates. As in \cite{jfcog}, the linearized problems in \cite{cgw3} satisfy a weak 
energy estimate with a loss of one tangential derivative.\footnote{In fact, the loss in \cite{cgw3} is a loss of 
one singular derivative $\partial_x +\beta \, \partial_\theta /\eps$ which implies a very bad control with respect 
to $\eps$.} Such estimates with a loss of derivative were originally proved in \cite{jfc} and are optimal, as proved 
in \cite{jfcog}. The amplification of oscillations is more or less equivalent to the loss of derivatives in the estimates. 
Compared with \cite{williams3}, we now need to control our remainders by showing that they are smoothing 
operators, otherwise there will be no way to absorb these errors in the energy estimates. Moreover, since the 
nonlinear problems of \cite{cgw3} are solved by a Nash-Moser procedure where we use smoothing operators, 
it is crucial to extend all the results on the singular calculus of \cite{williams3} by including the following features:
\begin{itemize}
 \item The symbols should not be assumed to be independent of the space variables outside of a compact set.
       Otherwise, we would face a lot of difficulties with the smoothing procedure in the Nash-Moser iteration.

 \item The remainders in the calculus of \cite{williams3} should be smoothing operators when they were merely
       bounded operators on $L^2$ with a small ($O(\gamma^{-1})$, $\gamma$ large) norm in \cite{williams3}. 
       Moreover, we desire more systematic and easily applicable criteria than in \cite{williams3} for determining 
       the mapping properties of remainders.
\end{itemize}

The techniques in \cite{williams3} heavily use the fact that all symbols are either Fourier multipliers or they are 
independent of the space variables outside of a compact set. One major goal here is to get rid of this assumption. 
Furthermore our symbols do not satisfy the standard isotropic decay in the frequency variables (basically there 
is one direction in frequency space in which there is no decay). We thus adopt a different strategy that is based 
on the Calder\'on-Vaillancourt Theorem and more specifically Hwang's proof of this Theorem \cite{hwang}. Our 
motivation for doing so is that we shall only rely on $L^\infty$ bounds for pseudodifferential symbols, while the 
classical proofs seem inapplicable when the frequency decay fails. The situation is even worse here because 
some "expected" results on adjoints or products of singular pseudodifferential operators seem not to hold. For 
instance, asymptotic expansions of symbols do not hold beyond the first term, and even the justification of the 
first term in the expansion depends on the order of the operators. Our final results are thus in some ways rather 
weak, but they seem to be more or less the best one can hope for in a singular calculus. Fortunately, the calculus 
is strong enough to be applicable to a variety of geometric optics problems for both wavetrains and pulses, 
including problems that display an amplification phenomenon.

We thus review the results of \cite{williams3} by improving them along the lines described above. For practical 
purposes, we have found it convenient to first prove general results on $L^2$-boundedness of pseudodifferential 
and oscillatory integral operators. This will be done in Section \ref{sect3} below. The calculus rules proved in 
Section \ref{sect5} are then more or less ``basic'' applications of the general results. We have also found it 
convenient to include in the same article, the results for both the whole space and the periodic framework. 
Results in the case of the whole space are gathered in Sections \ref{sect6}, \ref{sect7}, \ref{sect8} and will 
be used in a future work to deal with pulse-like solutions to hyperbolic boundary value problems, while the 
companion article \cite{cgw3} is devoted to wavetrains.
\newpage

\begin{center}
{\sc Part A: singular pseudodifferential calculus for wavetrains}
\end{center}

\section{Functional spaces}
\label{sect2}

In all this article, functions may be valued in $\C$, $\C^N$ or even in the space of square matrices
${\mathcal M}_N(\C)$ (or $\C^{N \times N}$). Products have to be understood in the sense of matrices when
the dimensions agree. If $M \in {\mathcal M}_N(\C)$, $M^*$ denotes the conjugate transpose of $M$. The
norm of a vector $x \in \C^N$ is $|x|:=(x^* \, x)^{1/2}$. If $x,y$ are two vectors in $\C^N$, we let $x \cdot y$
denote the quantity $\sum_j x_j \, y_j$, which coincides with the usual scalar product in $\R^N$ when $x$
and $y$ are real.

\subsection{Functional spaces on $\R^d$}

The Schwartz space ${\mathcal S}(\R^d)$ of ${\mathcal C}^\infty$ functions with fast decay at infinity is
equipped with the family of semi-norms:
\begin{equation*}
\forall \, J \in \N \, ,\quad \| u \|_{{\mathcal S}(\R^d),J} := \sup_{\alpha \in \N^d, |\alpha| \le J} \,
\sup_{x \in \R^d} \, (1+|x|^2)^{J/2} \, \big| \partial_x^\alpha u \, (x) \big| \, .
\end{equation*}
When equipped with this topology, ${\mathcal S}(\R^d)$ is a Fr\'echet space. We shall say that a sequence
$(u_k)_{k \in \Z}$ in ${\mathcal S}(\R^d)$ has fast decay if for all polynomial $P$, the sequence $(P(k)
\, u_k)_{k \in \Z}$ is bounded in ${\mathcal S}(\R^d)$.

The Fourier transform on ${\mathcal S}(\R^d)$ is defined by
\begin{equation*}
\forall \, f \in {\mathcal S}(\R^d) \, ,\quad \forall \, \xi \in \R^d \, ,\quad \widehat{f} (\xi) :=
\int_{\R^d} {\rm e}^{-i\, x \cdot \xi} \, f(x) \, {\rm d}x \, .
\end{equation*}
In particular, the Fourier transform is a continuous isomorphism on ${\mathcal S}(\R^d)$. It is extended
to the space of temperate distributions ${\mathcal S}'(\R^d)$ in the usual way.
\bigskip

For $s \in \R$, we let $H^s(\R^d)$ denote the Sobolev space
\begin{equation*}
H^s(\R^d) := \left\{ u \in {\mathcal S}'(\R^d) \, / \, (1+|\xi|^2)^{s/2} \, \widehat{u} \in L^2(\R^d)
\right\} \, .
\end{equation*}
It is equipped with the family of norms
\begin{equation*}
\forall \, \gamma \ge 1 \, ,\quad \forall \, u \in H^s(\R^d) \, ,\quad
\| u \|_{s,\gamma}^2 := \dfrac{1}{(2\, \pi)^d} \,
\int_{\R^d} (\gamma^2+|\xi|^2)^s \, \big| \widehat{u}(\xi) \big|^2 \, {\rm d}\xi \, .
\end{equation*}
The norm $\| \cdot \|_{0,\gamma}$ does not depend on $\gamma$ and coincides with the usual $L^2$-norm on
$\R^d$. We shall thus write $\| \cdot \|_0$ instead of $\| \cdot \|_{0,\gamma}$ for the $L^2$-norm on $\R^d$.
For simplicity, we also write $\| \cdot \|_s$ instead of $\| \cdot \|_{s,1}$ for the standard $H^s$-norm
(when the parameter $\gamma$ equals $1$).

\subsection{Functional spaces on $\R^d \times \T$}

We now extend the previous definitions to functions that depend in a periodic way on an additional variable
$\theta$. We shall in some sense ``interpolate'' between Fourier transform and Fourier series. Let us begin
with the definition of the Schwartz space. The Schwartz space ${\mathcal S}(\R^d \times \T)$ is the set of
${\mathcal C}^\infty$ functions $f$ on $\R^d \times \R$, that are $1$-periodic with respect to the last
variable, and with fast decay at infinity in the first variable, that is:
\begin{equation*}
\forall \, \alpha \, ,\beta \in \N^d \, ,\quad \forall \, j \in \N \, ,\quad
\Big( (x,\theta) \in \R^d \times \R \mapsto
x^\alpha \, \partial_x^\beta \, \partial_\theta^j \, f \, (x,\theta) \Big) \in L^\infty (\R^d \times \R) \, .
\end{equation*}
Using the perdiocity of $f$ with respect to its last argument $\theta$, one can replace equivalently
$L^\infty (\R^d \times \R)$ by $L^\infty (\R^d \times [0,1])$. The Schwartz space ${\mathcal S}(\R^d \times \T)$
is equipped with the family of semi-norms
\begin{equation*}
\forall \, J \in \N \, ,\quad \| f \|_{{\mathcal S}(\R^d \times \T),J} :=
\sup_{\substack{(\alpha,j) \in \N^d \times \N \\ |\alpha|+j \le J}} \,
\sup_{(x,\theta) \in \R^d \times [0,1]} \, (1+|x|^2)^{J/2} \,
\big| \partial_x^\alpha \, \partial_\theta^j \, f \, (x,\theta) \big| \, .
\end{equation*}
When equipped with this topology, ${\mathcal S}(\R^d \times \T)$ is a Fr\'echet space. We let ${\mathcal S}'
(\R^d \times \T)$ denote its topological dual, that is the set of continuous linear forms on ${\mathcal S}
(\R^d \times \T)$.

The ``Fourier transform'' on ${\mathcal S}(\R^d \times \T)$ is defined by considering Fourier series
in $\theta$ and Fourier transform in $x$. More precisely, we introduce the $k$-th Fourier coefficient:
\begin{equation*}
\forall \, f \in {\mathcal S}(\R^d \times \T) \, ,\quad \forall \, k \in \Z \, ,\quad
\forall \, x \in \R^d \, ,\quad
c_k(f) (x) := \int_0^1 {\rm e}^{-2\, i\, \pi \, k \, \theta} \, f(x,\theta) \, {\rm d}\theta \, .
\end{equation*}
For all integer $k$, the Fourier coefficient $c_k(f)$ belongs to the standard Schwartz space ${\mathcal S}
(\R^d)$. We can therefore define its Fourier transform $\widehat{c_k(f)}$. In all what follows, the sequence
$(\widehat{c_k(f)})_{k \in \Z}$ is called the Fourier transform of $f$. When we only consider Fourier series
in $\theta$, we use the notation $c_k$ to denote the $k$-th Fourier coefficient. When we only consider Fourier
transform with respect to the first variable $x \in \R^d$, we use the classical ``hat'' notation introduced
previously.

The reader can check that the Fourier transform $(\widehat{c_k(f)})_{k \in \Z}$ of a function $f \in {\mathcal S}
(\R^d \times \T)$ is a sequence in ${\mathcal S}(\R^d)$ with fast decay. The inverse Fourier transform is defined
through the formula:
\begin{equation*}
f(x,\theta) = \sum_{k \in \Z}
{\mathcal F}^{-1} (\widehat{c_k(f)}) (x) \, {\rm e}^{2\, i\, \pi \, k \, \theta} \, ,
\end{equation*}
where ${\mathcal F}^{-1}$ stands for the inverse Fourier transform in ${\mathcal S}(\R^d)$. To summarize, the
Fourier transform is an isomorphism between ${\mathcal S}(\R^d \times \T)$ and the sequences $(g_k)_{k \in \Z}$
in ${\mathcal S}(\R^d)$ with fast decay.

Let us now extend the Fourier transform to the set of temperate distributions ${\mathcal S}'(\R^d \times \T)$.
For $u \in {\mathcal S}'(\R^d \times \T)$, the Fourier coefficients $c_k(u) \in {\mathcal S}'(\R^d)$ are defined
by the formula
\begin{equation*}
\forall \, k \in \Z \, ,\quad \forall \, g \in {\mathcal S}(\R^d) \, ,\quad
\langle c_k(u),g \rangle_{{\mathcal S}'(\R^d),{\mathcal S}(\R^d)} :=
\Big\langle u,g(x) \, {\rm e}^{-2\, i \, \pi \, k \, \theta}
\Big\rangle_{{\mathcal S}'(\R^d \times \T),{\mathcal S}(\R^d \times \T)} \, .
\end{equation*}
It is straightforward to check that there exists a constant $C$ and an integer $J$ such that for all
$k \in \Z$, there holds the continuity estimate
\begin{equation*}
\forall \, g \in {\mathcal S}(\R^d) \, ,\quad
\big| \langle c_k(u),g \rangle_{{\mathcal S}'(\R^d),{\mathcal S}(\R^d)} \big| \le C \, (1+|k|)^J \,
\| g \|_{{\mathcal S}(\R^d),J} \, .
\end{equation*}
The Fourier transform of $u$ is the sequence $(\widehat{c_k(u)})_{k \in \Z}$ in ${\mathcal S}'(\R^d)$.
For an appropriate constant that is still denoted $C$ and a possibly larger integer that is still denoted
$J$, there holds the continuity estimate
\begin{equation}
\label{fourierS'}
\forall \, g \in {\mathcal S}(\R^d) \, ,\quad
\big| \langle \widehat{c_k(u)},g \rangle_{{\mathcal S}'(\R^d),{\mathcal S}(\R^d)} \big| \le C \,
(1+|k|)^J \, \| g \|_{{\mathcal S}(\R^d),J} \, .
\end{equation}
The continuity estimate \eqref{fourierS'} is uniform with respect to $k \in \Z$: the constant $C$ and the
integer $J$ are independent of $k$. Moreover, the Fourier transform
\begin{equation*}
u \in {\mathcal S}'(\R^d \times \T) \longmapsto (\widehat{c_k(u)})_{k \in \Z} \in {\mathcal S}'(\R^d)^\Z \, ,
\end{equation*}
is an isomorphism between ${\mathcal S}'(\R^d \times \T)$ and the sequences in ${\mathcal S}'(\R^d)$ that
satisfy a uniform estimate of the type \eqref{fourierS'}. The inverse Fourier transform is defined as follows:
for a given sequence $(u_k)_{k \in \Z}$ in ${\mathcal S}'(\R^d)$ satisfying a uniform continuity estimate with
respect to $k$, we define an element $v$ of ${\mathcal S}'(\R^d \times \T)$ by the formula
\begin{equation*}
\forall \, f \in {\mathcal S}(\R^d \times \T) \, ,\quad
\langle v,f \rangle_{{\mathcal S}'(\R^d \times \T),{\mathcal S}(\R^d \times \T)} := \sum_{k \in \Z}
\big\langle u_{-k},{\mathcal F}^{-1}(c_k(f)) \big\rangle_{{\mathcal S}'(\R^d),{\mathcal S}(\R^d)} \, ,
\end{equation*}
where ${\mathcal F}^{-1}$ denotes the inverse Fourier transform in ${\mathcal S}(\R^d)$. Indeed the
reader can check first that $v$ is well-defined, that it is a continuous linear form with respect to the
topology of ${\mathcal S}(\R^d \times \T)$ and that $\widehat{c_k(v)}$ equals $u_k$ for all $k \in \Z$.
\bigskip

For $s \in \R$, the Sobolev space $H^s(\R^d \times \T)$ is defined by
\begin{multline*}
H^s(\R^d \times \T) := \Big\{
u \in {\mathcal S}'(\R^d \times \T) \, / \, (c_k(u))_{k \in \Z} \in H^s(\R^d)^\Z \\
\text{\rm and} \quad \sum_{k \in \Z} \int_{\R^d} (1+k^2+|\xi|^2)^s \, \big| \widehat{c_k(u)}(\xi) \big|^2
\, {\rm d}\xi <+\infty \Big\} \, .
\end{multline*}
It is equipped with the family of norms
\begin{equation*}
\forall \, \gamma \ge 1 \, ,\quad \forall \, u \in H^s(\R^d \times \T) \, ,\quad
\| u \|_{s,\gamma}^2 := \dfrac{1}{(2\, \pi)^d} \,
\sum_{k \in \Z} \int_{\R^d} (\gamma^2+k^2+|\xi|^2)^s \, \big| \widehat{c_k(u)}(\xi) \big|^2 \, {\rm d}\xi \, .
\end{equation*}
The norm $\| \cdot \|_{0,\gamma}$ does not depend on $\gamma$ and coincides with the usual $L^2$-norm on
$\R^d \times \T$. We shall thus write $\| \cdot \|_0$ instead of $\| \cdot \|_{0,\gamma}$ for the $L^2$-norm
on $\R^d \times \T$. More precisely, if $f \in L^2(\R^d \times \T)$, then the Fourier coefficient
\begin{equation*}
c_k(f)(x) := \int_0^1 {\rm e}^{-2\, i\, \pi \, k \, \theta} \, f(x,\theta) \, {\rm d}\theta
\end{equation*}
is well-defined for almost every $x \in \R^d$, and $c_k(f)$ belongs to $L^2(\R^d)$ (use Cauchy-Schwarz
inequality). The Parseval-Bessel equality and Plancherel's Theorem give
\begin{equation*}
\int_{\R^d \times [0,1]} |f(x,\theta)|^2 \, {\rm d}x \, {\rm d}\theta
=\sum_{k \in \Z} \int_{\R^d} |c_k(f)(x)|^2 \, {\rm d}x =\| f \|_0^2 \, .
\end{equation*}
In what follows, we always identify the space $L^2(\R^d \times \T)$ and Fourier series in $\theta \in \R$
whose coefficients belong to $\ell_2(\Z;L^2(\R^d))$. For simplicity, we also write $\| \cdot \|_s$ instead
of $\| \cdot \|_{s,1}$ for the standard $H^s$-norm on $\R^d \times \T$.

\begin{remark}
\label{rem1}
Observe that our notation for the norm $\| \cdot \|_{s,\gamma}$ is consistent with the notation for functions
that are defined on $\R^d$. More precisely, if $u \in H^s(\R^d)$, then one can also consider $u$ as an element
of $H^s(\R^d \times \T)$ that does not depend on $\theta$, meaning that only the $0$-th harmonic in $\theta$
occurs ($c_0(u)=u$ and $c_k(u)=0$ if $k \neq 0$). The norms of $u$ in $H^s(\R^d)$ and $H^s(\R^d \times \T)$
coincide. This is the reason why we omit to write the underlying space $\R^d$ or $\R^d \times \T$ in the
definition of the norms $\| \cdot \|_{s,\gamma}$.
\end{remark}

We now introduce the ``singular'' Sobolev spaces that we shall widely use in this article. From now on, we
consider a vector $\beta \in \R^d \setminus \{ 0\}$ that is fixed once and for all. For $s \in \R$ and $\eps
\in \, ]0,1]$, the anisotropic Sobolev space $H^{s,\eps}(\R^d \times \T)$ is defined by
\begin{multline*}
H^{s,\eps}(\R^d \times \T) := \Big\{ u \in {\mathcal S}'(\R^d \times \T) \, / \,
\forall \, k \in \Z \, ,\quad \widehat{c_k(u)} \in L^2_{\rm loc}(\R^d) \\
\text{\rm and} \quad
\sum_{k \in \Z} \int_{\R^d} \left( 1+\left| \xi+\dfrac{2\, \pi \, k \, \beta}{\eps} \right|^2 \right)^s
\, \big| \widehat{c_k(u)}(\xi) \big|^2 \, {\rm d}\xi <+\infty \Big\} \, .
\end{multline*}
It is equipped with the family of norms
\begin{equation}
\label{defnormeepsgamma}
\forall \, \gamma \ge 1 \, ,\quad \forall \, u \in H^{s,\eps}(\R^d \times \T) \, ,\quad
\| u \|_{H^{s,\eps},\gamma}^2 := \dfrac{1}{(2\, \pi)^d} \, \sum_{k \in \Z}
\int_{\R^d} \left( \gamma^2 +\left| \xi+\dfrac{2\, \pi \, k \, \beta}{\eps} \right|^2 \right)^s
\, \big| \widehat{c_k(u)}(\xi) \big|^2 \, {\rm d}\xi \, .
\end{equation}
Let us observe that the definition of the space $H^{s,\eps}$ depends on $\eps$, and there is no obivous inclusion
$H^{s,\eps_1} \subset H^{s,\eps_2}$ if $\eps_1 \le \eps_2$ or $\eps_1 \ge \eps_2$. However, for a fixed $\eps>0$,
the norms $\| \cdot \|_{H^{s,\eps},\gamma_1}$ and $\| \cdot \|_{H^{s,\eps},\gamma_2}$ are equivalent. In particular,
\eqref{defnormeepsgamma} defines a norm on the space $H^{s,\eps} (\R^d \times \T)$ defined above. When $m$ is
an integer, the space $H^{m,\eps} (\R^d \times \T)$ coincides with the space of functions $u \in L^2 (\R^d \times \T)$
such that the derivatives, in the sense of distributions,
\begin{equation*}
\left( \partial_{x_1} +\dfrac{\beta_1}{\eps} \, \partial_\theta \right)^{\alpha_1} \dots
\left( \partial_{x_d} +\dfrac{\beta_d}{\eps} \, \partial_\theta \right)^{\alpha_d} \, u \, ,\quad
\alpha_1+\dots+\alpha_d \le m \, ,
\end{equation*}
belong to $L^2 (\R^d \times \T)$. In the definition of the norm $\| \cdot \|_{H^{m,\eps},\gamma}$, one power of
$\gamma$ counts as much as one derivative.

In what follows, we shall also make use of the spaces ${\mathcal C}^k_b(\R^d \times \T)$, $k \in \N$: these
are the spaces of continuous and bounded functions on $\R^d \times \R$ that are $1$-periodic with respect to
their last argument, whose derivatives up to the order $k$ exist, are continuous and bounded.

\section{The main $L^2$ continuity results}
\label{sect3}

Our goal is to develop in Section \ref{sect4} a singular symbolic calculus on $\R^d \times \T$. This Section
will give the basic results to achieve this goal. As in \cite{williams3}, the symbols that we shall consider
do not satisfy the standard decay estimates in the frequency variable. Consequently, it will be more difficult
to show that remainders in the symbolic calculus are smoothing operators. As a matter of fact, this property
will hold only in a framework of the anisotropic Sobolev spaces defined above. A more embarassing consequence
of this non-decay is that there seems to be little hope for developing a paradifferential version of the calculus
below. More precisely, in the paradifferential calculus theory (see e.g. \cite{metivier3}), symbols have a fixed,
say $W^{k,\infty}$, regularity in $x$. To cope with this small regularity, one introduces an isotropic frequency
cut-off in the space variable. The regularized symbol belongs to the class $S^m_{1,1}$ and satisfies a suitable
spectral condition, which yields continuity results for the associated pseudodifferential operator. This
strategy applies when symbols are ${\mathcal C}^\infty$ in $\xi$ with the standard decay property (each
derivative in $\xi$ yields one negative power of $|\xi|$). However, when this frequency decay does not hold or
when it holds only in an anisoptropic way, the smoothing procedure yields symbols in the class $S^m_{0,1}$ where
derivatives in $x$ are not balanced anylonger by derivatives in $\xi$. For such symbols, even with an appropriate
spectral condition, there seems to be very little hope for continuity results in Sobolev spaces.

The remarks above are the main reason why we base our approach on the Calder\'on-Vaillancourt Theorem
\cite{calderonvaillancourt}. More precisely, we shall prove continuity results on $L^2(\R^d \times \T)$
with symbols satisfying $L^\infty$ bounds (no decay in the frequency variable is needed, see Theorem
\ref{thm1} below). This is the same strategy as in \cite{williams3}. However we shall use more elaborate
tools in order to get some refined estimates. Our goal is to get rid of the compact support assumptions
in \cite{williams3}, and to lower the regularity required on the symbols whenever this is possible.
We refer the reader to \cite{cordes,coifmanmeyer,bourdaudmeyer,hwang} for some background on the
Calder\'on-Vaillancourt Theorem and some generalizations. Here we clarify how these results can be
adapted to a mixed situation where part of the space variables lie in $\R^d$ while the other space
variables lie in the torus $\T$. As far as we know, all previous versions were restricted to the
case of $\R^d$ or to the case of the torus. Our first continuity result is:

\begin{theorem}
\label{thm1}
Let $\sigma : \R^d \times \T \times \R^d \times \Z \rightarrow \C^{N \times N}$ be a continuous
function\footnote{Here and in all what follows, a function $\sigma$ that is defined on ${\mathcal O} \times \Z$
is said to be continuous if for all $k \in \Z$, $\sigma (\cdot,k)$ is continuous on ${\mathcal O}$. The set
${\mathcal O}$ will represent either $\R^d$, or $\R^d \times \T$ or analogous sets. We adopt the same convention
for differentiability properties.} that satisfies the property: for all $\alpha,\beta \in \{ 0,1\}^d$ and for all
$j \in \{ 0,1\}$, the derivative (in the sense of distributions) $\partial_x^\alpha \, \partial_\theta^j \,
\partial_\xi^\beta \, \sigma$ belongs to $L^\infty (\R^d \times \T \times \R^d \times \Z)$.

For $u \in {\mathcal S}(\R^d \times \T;\C^N)$, let us define
\begin{equation*}
\forall \, (x,\theta) \in \R^d \times \T \, ,\quad
\op (\sigma) \, u \, (x,\theta) := \dfrac{1}{(2\, \pi)^d} \, \sum_{k \in \Z} \int_{\R^d}
{\rm e}^{i \, x \cdot \xi} \, {\rm e}^{2 \, i \, \pi \, k \, \theta} \, \sigma (x,\theta,\xi,k) \,
\widehat{c_k(u)}(\xi) \, {\rm d}\xi \, .
\end{equation*}
Then $\op (\sigma)$ extends as a continuous operator on $L^2(\R^d \times \T;\C^N)$. More precisely, there
exists a numerical constant $C$, that only depends on $d$ and $N$, such that for all $u \in {\mathcal S}
(\R^d \times \T;\C^N)$, there holds
\begin{equation}
\label{continuiteL2-1}
\| \op (\sigma) \, u \|_0 \le C \, \Ng \sigma \Nd \, \| u \|_0 \, ,\, \quad
\text{\rm with } \Ng \sigma \Nd := \sup_{\alpha,\beta \in \{ 0,1\}^d} \, \sup_{j \in \{ 0,1\}} \,
\left\| \partial_x^\alpha \, \partial_\theta^j \, \partial_\xi^\beta \, \sigma
\right\|_{L^\infty (\R^d \times \T \times \R^d \times \Z)} \, .
\end{equation}
\end{theorem}

The proof of Theorem \ref{thm1} below is analogous to the proof of \cite[Theorem 2]{hwang}. We emphasize
that in the assumptions on the symbol $\sigma$, no finite difference with respect to the index $k \in \Z$
appears. This is in sharp contrast with for instance the paradifferential calculus on the torus developed
in \cite{delortszeftel}. The fact that we do not need to estimate finite differences in $k$ will be helpful
in Section \ref{sect4} when we consider singular pseudodifferential operators.

\begin{proof}[Proof of Theorem \ref{thm1}]
The proof of Theorem \ref{thm1} combines two ingredients. First, the main estimate \eqref{continuiteL2-1}
holds when $\sigma$ is smooth, say ${\mathcal C}^\infty$, with compact support in all variables. Second,
it is possible to approximate a symbol $\sigma$ satisfying the assumptions of Theorem \ref{thm1} by a
sequence $(\sigma_p)_{p \in \N}$ of smooth symbols with $\sup_p \Ng \sigma_p \Nd$ controled by $\Ng \sigma \Nd$.
The corresponding pseudodifferential operators $\op (\sigma_p)$ converge in a weak sense towards $\op (\sigma)$.

For smooth symbols with compact support, integration by parts and derivation under the integral show that
$\op(\sigma) \, u$ belongs to ${\mathcal S} (\R^d \times \T)$ if $u$ does. In particular, $\op(\sigma) \,
u$ belongs to $L^2 (\R^d \times \T)$. This integrability property is not so clear under the general
assumptions of Theorem \ref{thm1}.

Let us state more precisely our first point.

\begin{lemma}
\label{lem1}
Let $\sigma \in {\mathcal C}^\infty_0 (\R^d \times \T \times \R^d \times \Z; \C^{N \times N})$, that is:
\begin{itemize}
 \item[{\rm (i)}] $\sigma (\cdot,\cdot,\cdot,k) \equiv 0$ except for a finite number of integers $k$,
 \item[{\rm (ii)}] $\sigma (\cdot,\cdot,\cdot,k)$ is a ${\mathcal C}^\infty$ function on $\R^d \times \T
      \times \R^d$ for all $k \in \Z$, with compact support in its first and third variables.
\end{itemize}
Then for all $u \in {\mathcal S}(\R^d \times \T;\C^N)$, $\op (\sigma) \, u$ belongs to ${\mathcal S}
(\R^d \times \T;\C^N)$ and the estimate \eqref{continuiteL2-1} holds with a numerical constant $C$
that is independent of $\sigma$ and $u$.
\end{lemma}

\begin{proof}[Proof of Lemma \ref{lem1}]
We make the proof in the case $N=1$. When $\sigma$ takes its values in the space of matrices $\C^{N \times N}$,
the result applies for each component. Following \cite{hwang}, it is sufficient to prove an estimate of the form
\begin{equation}
\label{lem1-equation0}
\left| \int_{\R^d \times \T} \op(\sigma) \, u (x,\theta) \, v(x,\theta) \, {\rm d}x \, {\rm d}\theta \right|
\le C \, \Ng \sigma \Nd \, \| u \|_0 \, \| v \|_0 \, ,
\end{equation}
for all $u,v \in {\mathcal S}(\R^d \times \T;\C)$. We define a function $\varphi$ on $\R^d$ by the formula:
\begin{equation*}
\forall \, y = (y_1,\dots,y_d) \in \R^d \, ,\quad \varphi(y) := \prod_{j=1}^d (1+i \, y_j)^{-1} \, .
\end{equation*}
In particular, $\varphi$ belongs to $L^2(\R^d)$. Applying Fubini's Theorem, we have
\begin{align*}
I &:= \int_{\R^d \times \T} \op(\sigma) \, u (x,\theta) \, v(x,\theta) \, {\rm d}x \, {\rm d}\theta \\
&= \dfrac{1}{(2\, \pi)^d} \, \sum_{k \in \Z} \int_{\R^d \times \T \times \R^d \times \R^d}
{\rm e}^{i \, (x-y) \cdot \xi} \, {\rm e}^{2 \, i \, \pi \, k \, \theta} \, \sigma (x,\theta,\xi,k) \,
c_k(u)(y) \, v(x,\theta) \, {\rm d}x \, {\rm d}\theta \, {\rm d}y \, {\rm d}\xi \, .
\end{align*}
Starting from the relation
\begin{equation*}
{\rm e}^{i \, (x-y) \cdot \xi} =\varphi(x-y) \, \prod_{j=1}^d (1+\partial_{\xi_j}) \,
{\rm e}^{i \, (x-y) \cdot \xi} \, ,
\end{equation*}
and integrating by parts, we obtain
\begin{equation}
\label{lem1-equation1}
I = \dfrac{1}{(2\, \pi)^d} \, \sum_{k \in \Z} \int_{\R^d \times \T \times \R^d} {\rm e}^{i \, x \cdot \xi} \,
{\rm e}^{2 \, i \, \pi \, k \, \theta} \, \sigma_\sharp (x,\theta,\xi,k) \, U(x,\xi,k) \, v(x,\theta) \,
{\rm d}x \, {\rm d}\theta \, {\rm d}\xi \, ,
\end{equation}
where we have used the notation
\begin{equation*}
\sigma_\sharp := \prod_{j=1}^d (1-\partial_{\xi_j}) \, \sigma \, ,\quad
U(x,\xi,k) := \int_{\R^d} {\rm e}^{-i \, y \cdot \xi} \, \varphi(x-y) \, c_k(u)(y) \, {\rm d}y \, .
\end{equation*}

We use the expression
\begin{equation*}
v(x,\theta) = \dfrac{1}{(2\, \pi)^d} \, \sum_{\ell \in \Z} \int_{\R^d} {\rm e}^{i \, x \cdot \eta} \,
{\rm e}^{2 \, i \, \pi \, \ell \, \theta} \, \widehat{c_\ell(v)} (\eta) \, {\rm d}\eta
\end{equation*}
in \eqref{lem1-equation1} and apply Fubini's Theorem again. Then we use the relation
\begin{equation*}
{\rm e}^{i \, x \cdot (\xi+\eta)} \, {\rm e}^{2 \, i \, \pi \, (k+\ell) \, \theta}
= \dfrac{\varphi(\xi+\eta)}{1+2 \, i \, \pi \, (k+\ell)} \, \left\{ (1+\partial_\theta) \,
\prod_{j=1}^d (1+\partial_{x_j}) \right\} \,
{\rm e}^{i \, x \cdot (\xi+\eta)} \, {\rm e}^{2 \, i \, \pi \, (k+\ell) \, \theta} \, ,
\end{equation*}
and integrate by parts. These operations yield
\begin{multline}
\label{lem1-equation2}
I = \sum_{\alpha \in \{ 0,1\}^d, j \in \{0,1\}} \sum_{\alpha' \le \alpha} \star \\
\sum_{k\in \Z} \int_{\R^d \times \T \times \R^d} {\rm e}^{i \, x \cdot \xi} \,
{\rm e}^{2 \, i \, \pi \, k \, \theta} \,
\partial_x^{\alpha-\alpha'} \, \partial_\theta^j \, \sigma_\sharp \, (x,\theta,\xi,k) \,
\partial_x^{\alpha'} U \, (x,\xi,k) \, V(x,\theta,\xi,k) \, {\rm d}x \, {\rm d}\theta \, {\rm d}\xi \, ,
\end{multline}
where the $\star$ coefficients denote some harmless numerical constants that only depend on $\alpha,\alpha',j$,
and where we have used the notation
\begin{equation*}
V(x,\theta,\xi,k) := \sum_{\ell \in \Z} \left( \int_{\R^d} {\rm e}^{i \, x \cdot \eta} \,
\dfrac{\varphi(\xi+\eta)}{1+2 \, i \, \pi \, (k+\ell)} \, \widehat{c_\ell(v)}(\eta) \, {\rm d}\eta \right)
\, {\rm e}^{2 \, i \, \pi \, \ell \, \theta} \, .
\end{equation*}

The result of Lemma \ref{lem1} follows by applying Cauchy-Schwarz inequality to each integral in
\eqref{lem1-equation2} (here the integral also includes the sum with respect to the index $k \in \Z$).
Each derivative $\partial_x^{\alpha-\alpha'} \, \partial_\theta^j \, \sigma_\sharp$ that appears in the
right-hand side of \eqref{lem1-equation2} can be estimated in $L^\infty$-norm by a harmless constant times
the quantity $\Ng \sigma \Nd$ defined in \eqref{continuiteL2-1}. We thus get (here and from now on, $C$
denotes a positive numerical constant that may vary from line to line)
\begin{equation*}
|I|^2 \le C \, \Ng \sigma \Nd^2 \, \left( \sum_{\alpha \in \{ 0,1\}^d} \sum_{k\in \Z}
\int_{\R^d \times \R^d} \big| \partial_x^\alpha U \, (x,\xi,k) \big|^2 \, {\rm d}x \, {\rm d}\xi \right)
\, \sum_{k\in \Z} \int_{\R^d \times \T \times \R^d} |V \, (x,\theta,\xi,k)|^2 \, {\rm d}x \, {\rm d}\theta \,
{\rm d}\xi \, .
\end{equation*}
Each term on the right-hand side is estimated by using the Parseval-Bessel equality and Plancherel's Theorem
(see \cite{hwang} for the case of $\R^d$, here the incorporation of the additional periodic variable is almost
straightforward):
\begin{align*}
&\sum_{k\in \Z} \int_{\R^d \times \R^d} \big| \partial_x^\alpha U \, (x,\xi,k) \big|^2 \, {\rm d}x \,
{\rm d}\xi \le C \, \sum_{k\in \Z} \| c_k(u) \|_0^2 \le C \, \| u \|_0^2 \, ,\\
&\sum_{k\in \Z} \int_{\R^d \times \T \times \R^d} |V \, (x,\theta,\xi,k)|^2 \, {\rm d}x \, {\rm d}\theta \,
{\rm d}\xi \le C \, \| v \|_0^2 \, .
\end{align*}
The proof of Lemma \ref{lem1} is thus complete.
\end{proof}

To complete the proof of Theorem \ref{thm1}, it is sufficient to prove the following approximation result:

\begin{lemma}
\label{lem2}
Let $\sigma : \R^d \times \T \times \R^d \times \Z \rightarrow \C^{N \times N}$ satisfy the assumptions of
Theorem \ref{thm1}. Then there exists a sequence $(\sigma_p)_{p \in \N}$ in ${\mathcal C}^\infty_0 (\R^d
\times \T \times \R^d \times \Z; \C^{N \times N})$ such that:
\begin{itemize}
 \item[{\rm (i)}] $\sup_{p \in \N} \Ng \sigma_p \Nd \le C \, \Ng \sigma \Nd$ for some numerical constant $C$
      that does not depend on $\sigma$,
 \item[{\rm (ii)}] for all $u,v \in {\mathcal S}(\R^d \times \T)$, there holds
\begin{equation*}
\lim_{p \rightarrow +\infty}
\int_{\R^d \times \T} \op(\sigma_p) \, u (x,\theta) \, v(x,\theta) \, {\rm d}x \, {\rm d}\theta
=\int_{\R^d \times \T} \op(\sigma) \, u (x,\theta) \, v(x,\theta) \, {\rm d}x \, {\rm d}\theta \, .
\end{equation*}
\end{itemize}
\end{lemma}

The proof of Lemma \ref{lem2} follows by the classical truncation-convolution argument. We leave the details
to the reader. The convergence property {\rm (ii)} follows from the dominated convergence Theorem.

Combining Lemma \ref{lem1} and Lemma \ref{lem2}, we obtain the main estimate \eqref{lem1-equation0} not only for
smooth symbols with compact support but also for the more general class of symbols that satisfy the assumptions
of Theorem \ref{thm1}. In particular, the Riesz Theorem shows that $\op (\sigma) \, u$ coincides almost-everywhere
with an element of $L^2 (\R^d \times \T)$, and the conclusion of Theorem \ref{thm1} follows.
\end{proof}

It is useful to observe that in the proof of Lemma \ref{lem1} above, we do not need the symbol $\sigma$ to
have compact support with respect to the space variable $x$. As a matter of fact, compact support with respect
to the dual variables $(\xi,k)$ is sufficient to justify all the calculations. This observation will be used
in the proof of Lemma \ref{lem3} below.

Of course, the classical version of the Calder\'on-Vaillancourt Theorem in $\R^d$ now appears as a particular
case of Theorem \ref{thm1} (apply Theorem \ref{thm1} with a symbol $\sigma$ containing only the $0$-harmonic
and that is independent of $\theta$ and similar test functions $u$), see \cite[page 18]{coifmanmeyer} and
\cite{hwang}. In the proof of Theorem \ref{thm1}, no finite difference with respect to $k$ appears because
there is no need to gain integrability for the function $U$ with respect to the variable $\theta$ (because
the torus has finite measure). An even more direct explaination is that for a bounded sequence, the iterative
finite differences are also bounded so the assumption would be redundant.

Applying formally Fubini's Theorem to the definition of $\op(\sigma) \, u$, we have
\begin{equation}
\label{oscintegral0}
\op (\sigma) \, u \, (x,\theta) = \dfrac{1}{(2\, \pi)^d} \, \sum_{k \in \Z} \int_{\R^d \times \R^d \times \T}
{\rm e}^{i \, (x-y) \cdot \xi} \, {\rm e}^{2 \, i \, \pi \, k \, (\theta-\omega)} \, \sigma (x,\theta,\xi,k) \,
u(y,\omega) \, {\rm d}\xi \, {\rm d}y \, {\rm d}\omega \, .
\end{equation}
The latter formula is rigorous e.g. for smooth symbols with compact support in $(\xi,k)$. In order to prepare
the results of symbolic calculus, our next goal is to obtain $L^2$ continuity results for oscillatory integral
operators as in \eqref{oscintegral0} with more general amplitudes $\sigma$; namely we should allow $\sigma$ to
depend on $(x,\theta)$ but also on the additional variables $(y,\omega)$, see e.g. \cite[page 144]{williams3}.
Our second main continuity result is the following:

\begin{theorem}
\label{thm2}
Let $\sigma : \R^d \times \T \times \R^d \times \T \times \R^d \times \Z \rightarrow \C^{N \times N}$
be a continuous function that satisfies the property: for all $\alpha,\beta \in \{ 0,1\}^d$, for all
$j,l \in \{ 0,1\}$ and for all $\nu \in \{ 0,1,2\}^d$, the derivative (in the sense of distributions)
$\partial_x^\alpha \, \partial_\theta^j \, \partial_y^\beta \, \partial_\omega^l \, \partial_\xi^\nu \,
\sigma$ belongs to $L^\infty (\R^d \times \T \times \R^d \times \T \times \R^d \times \Z)$. Let $\chi_1
\in {\mathcal C}^\infty_0 (\R)$ and $\chi_2 \in {\mathcal C}^\infty_0 (\R^d)$ satisfy $\chi_1 (0) =
\chi_2 (0) = 1$.

Then for all $u \in {\mathcal S}(\R^d \times \T)$, the sequence of functions $(T_\delta)_{\delta>0}$
defined on $\R^d \times \T$ by
\begin{multline}
\label{oscintegral}
T_\delta \, (x,\theta) := \dfrac{1}{(2\, \pi)^d} \, \sum_{k \in \Z} \chi_1 (\delta \, k) \,
\int_{\R^d \times \R^d \times \T} {\rm e}^{i \, (x-y) \cdot \xi} \,
{\rm e}^{2 \, i \, \pi \, k \, (\theta-\omega)} \\
\chi_2 (\delta \, \xi) \, \sigma (x,\theta,y,\omega,\xi,k) \, u(y,\omega) \,
{\rm d}\xi \, {\rm d}y \, {\rm d}\omega \, ,
\end{multline}
converges in ${\mathcal S}' (\R^d \times \T)$, as $\delta$ tends to $0$, towards a distribution
$\optilde (\sigma) \, u \in L^2 (\R^d \times \T)$. This limit is independent of the truncation functions
$\chi_1,\chi_2$. Moreover, there exists a numerical constant $C$, that only depends on $d$ and $N$, such
that there holds
\begin{multline}
\label{continuiteL2-2}
\left\| \optilde (\sigma) \, u \right\|_0 \le C \, \Ng \sigma \Nd_{\rm Amp} \, \| u \|_0 \, ,\, \\
\text{\rm with } \Ng \sigma \Nd_{\rm Amp} := \sup_{\alpha,\beta \in \{ 0,1\}^d} \,
\sup_{j,l \in \{ 0,1\}} \, \sup_{\nu \in \{ 0,1,2\}^d} \,
\left\| \partial_x^\alpha \, \partial_\theta^j \, \partial_y^\beta \, \partial_\omega^l \, \partial_\xi^\nu
\, \sigma \right\|_{L^\infty (\R^d \times \T \times \R^d \times \T \times \R^d \times \Z)} \, .
\end{multline}
\end{theorem}

The proof of Theorem \ref{thm2} splits in several steps. The first point is to show that the conclusion holds
for smooth symbols with compact support in $(\xi,k)$. In this case, the convergence of the oscillatory integral
as $\delta$ tends to $0$ follows from the dominated convergence Theorem, and the proof of the continuity estimate
\eqref{continuiteL2-2} relies on some arguments that are similar to those used in the proof of Lemma \ref{lem1}.
This first part of the proof of Theorem \ref{thm2} is achieved in Lemma \ref{lem3} below. The end of the proof
of Theorem \ref{thm2} consists in justifying the convergence of the oscillatory integral for an arbitrary amplitude
and in showing that \eqref{continuiteL2-2} still holds. This part of the proof relies on a regularization process
as for Lemma \ref{lem2}.

The process used in Theorem \ref{thm2} that consists in introducing cut-off functions in the frequency variables
and in passing to the limit will be systematically used in what follows in order to define oscillatory integral
operators and to show some properties on such operators. To highlight the difference between standard
pseudodifferential operators and oscillatory integral operators (for which the integrals do not converge in
a classical sense), we always use the notation $\optilde$ for oscillatory integral operators. In that case,
the representation by a convergent integral only takes place when the amplitude is integrable with respect
to the frequency variables (for instance, when it has compact support with respect to these variables).

\begin{proof}[Proof of Theorem \ref{thm2}]
We begin with the following generalization of Lemma \ref{lem1}.

\begin{lemma}
\label{lem3}
Let $\sigma \in {\mathcal C}^\infty_b (\R^d \times \T \times \R^d \times \T \times \R^d \times \Z; \C^{N \times N})$
have compact support with respect to $(\xi,k)$, that is, there exists an integer $K_0$ and a positive number $R_0$
such that $\sigma (x,\theta,y,\omega,\xi,k) = 0$ as long as $|k| \ge K_0$ or $|\xi| \ge R_0$.

Then all the conclusions of Theorem \ref{thm2} hold and the oscillatory integral $\optilde (\sigma) \, u$ coincides
with the function
\begin{equation*}
(x,\theta) \in \R^d \times \T \longmapsto \dfrac{1}{(2\, \pi)^d} \, \sum_{k \in \Z}
\int_{\R^d \times \R^d \times \T} {\rm e}^{i \, (x-y) \cdot \xi} \, {\rm e}^{2 \, i \, \pi \, k \, (\theta-\omega)}
\, \sigma (x,\theta,y,\omega,\xi,k) \, u(y,\omega) \, {\rm d}\xi \, {\rm d}y \, {\rm d}\omega \, .
\end{equation*}
\end{lemma}

\begin{proof}[Proof of Lemma \ref{lem3}]
Our strategy follows closely the proof of Lemma \ref{lem1}. In particular, we keep the same notation for
the function $\varphi$ on $\R^d$, and we make the proof in the case $N=1$ for simplicity.

First of all, since $\sigma$ has compact support in $(\xi,k)$ and is bounded, the sequence $(T_\delta)_{\delta>0}$
defined by \eqref{oscintegral} is bounded in $L^\infty (\R^d \times \T)$. Moreover, the dominated convergence
Theorem shows that $T_\delta$ converges pointwise, as $\delta$ tends to $0$, towards
\begin{equation*}
T(x,\theta) := \dfrac{1}{(2\, \pi)^d} \, \sum_{k \in \Z} \int_{\R^d \times \R^d \times \T}
{\rm e}^{i \, (x-y) \cdot \xi} \, {\rm e}^{2 \, i \, \pi \, k \, (\theta-\omega)} \,
\sigma (x,\theta,y,\omega,\xi,k) \, u(y,\omega) \, {\rm d}\xi \, {\rm d}y \, {\rm d}\omega \, .
\end{equation*}
There is no ambiguity in the definition of the latter integral since the function to be integrated has compact
support in $\xi$ and fast decay at infinity in $y$ (the sum with respect to $k$ only involves finitely many terms).
Applying again the dominated convergence Theorem, $(T_\delta)_{\delta>0}$ converges towards $T$ not only pointwise
but also in ${\mathcal S}' (\R^d \times \T)$. It thus only remains to estimate the function $T$ in $L^2$ in order
to complete the proof of Lemma \ref{lem3}. We emphasize that the proof below does not assume compact support of
$\sigma$ in $x$ or $y$, which will be useful in Section \ref{sect4}.

For $v \in {\mathcal S} (\R^d \times \T;\C)$, let us define the integral
\begin{align*}
I &:= \int_{\R^d \times \T} T (x,\theta) \, v (x,\theta) \, {\rm d}x \, {\rm d}\theta \\
&=\dfrac{1}{(2\, \pi)^d} \, \sum_{k \in \Z} \int_{\R^d \times \R^d \times \T \times \R^d \times \T}
{\rm e}^{i \, (x-y) \cdot \xi} \, {\rm e}^{2 \, i \, \pi \, k \, (\theta-\omega)} \,
\sigma (x,\theta,y,\omega,\xi,k) \, u(y,\omega) \, v(x,\theta) \,
{\rm d}\xi \, {\rm d}x \, {\rm d}\theta \, {\rm d}y \, {\rm d}\omega \, ,
\end{align*}
where we have applied Fubini's Theorem. We first expand $v$ as a Fourier series in $\theta$:
\begin{equation*}
v(x,\theta) =\sum_{\ell \in \Z} c_\ell(v) (x) \, {\rm e}^{2 \, i \, \pi \, \ell \, \theta} \, ,
\end{equation*}
apply Fubini's Theorem, and integrate by parts with respect to $\theta$ using the relation
\begin{equation*}
{\rm e}^{2 \, i \, \pi \, (k+\ell) \, \theta} =\dfrac{1}{1+2 \, i \, \pi \, (k+\ell)} \,
(1+\partial_\theta) \, {\rm e}^{2 \, i \, \pi \, (k+\ell) \, \theta} \, .
\end{equation*}
We apply a similar manipulation for $u$, and we obtain the relation
\begin{equation}
\label{lem3-equation1}
I = \sum_{k \in \Z} \int_{\R^d \times \R^d \times \T \times \R^d \times \T}
{\rm e}^{i \, (x-y) \cdot \xi} \, {\rm e}^{2 \, i \, \pi \, k \, (\theta-\omega)} \,
\sigma_\natural (x,\theta,y,\omega,\xi,k) \, \widetilde{u}(y,\omega,k) \, \widetilde{v}(x,\theta,k) \,
{\rm d}\xi \, {\rm d}x \, {\rm d}\theta \, {\rm d}y \, {\rm d}\omega \, ,
\end{equation}
where we have introduced the notation
\begin{align}
\sigma_\natural &:= (1-\partial_\theta) \, (1-\partial_\omega) \, \sigma \, ,\notag\\
\widetilde{u}(y,\omega,k) &:= \sum_{\ell \in \Z} \dfrac{c_\ell(u)(y)}{1+2 \, i \, \pi \, (\ell-k)} \,
{\rm e}^{2 \, i \, \pi \, \ell \, \omega} \, ,\quad
\widetilde{v}(x,\theta,k) := \sum_{\ell \in \Z} \dfrac{c_\ell(v)(x)}{1+2 \, i \, \pi \, (\ell+k)} \,
{\rm e}^{2 \, i \, \pi \, \ell \, \theta} \, .\label{lem3-equation2}
\end{align}
The latter manipulations are justified by the fact that both sequences $(c_\ell(u))_{\ell \in \Z}$ and
$(c_\ell(v))_{\ell \in \Z}$ have fast decay in ${\mathcal S} (\R^d)$.

Let us now transform the expression of $I$ in \eqref{lem3-equation1} by integrating by parts with respect to
$\xi$. More precisely, we use the relation
\begin{equation*}
{\rm e}^{i \, (x-y) \cdot \xi} =\varphi(x-y)^2 \, \prod_{j=1}^d (1+\partial_{\xi_j})^2 \,
{\rm e}^{i \, (x-y) \cdot \xi} \, ,
\end{equation*}
integrate by parts with respect to $\xi$ in \eqref{lem3-equation1} and obtain
\begin{multline}
\label{lem3-equation3}
I = \sum_{k \in \Z} \int_{\R^d \times \R^d \times \T \times \R^d \times \T}
{\rm e}^{i \, (x-y) \cdot \xi} \, {\rm e}^{2 \, i \, \pi \, k \, (\theta-\omega)} \\
\times \sigma_\flat (x,\theta,y,\omega,\xi,k) \, \varphi(x-y) \, \widetilde{u}(y,\omega,k) \,
\varphi(x-y) \, \widetilde{v}(x,\theta,k) \,
{\rm d}\xi \, {\rm d}x \, {\rm d}\theta \, {\rm d}y \, {\rm d}\omega \, ,
\end{multline}
where we have used the notation
\begin{equation*}
\sigma_\flat := \prod_{j=1}^d (1-\partial_{\xi_j})^2 \, \sigma_\natural
= \prod_{j=1}^d (1-\partial_{\xi_j})^2 \, (1-\partial_\theta) \, (1-\partial_\omega) \sigma \, .
\end{equation*}
A crucial observation for what follows is that the new term $\varphi (x-y)^2$ in \eqref{lem3-equation3} yields
integrability with respect to either $x$ or $y$.

Now we follow the argument already used in the proof of Lemma \ref{lem1}. We use Fourier's inversion formula,
and write
\begin{equation*}
\widetilde{v}(x,\theta,k) = \dfrac{1}{(2\, \pi)^d} \, \int_{\R^d}
{\rm e}^{i\, x \cdot \eta} \, \widehat{\widetilde{v}}(\eta,\theta,k) \, {\rm d}\eta \, ,
\end{equation*}
where for each $k \in \Z$, the partial Fourier transform $\widehat{\widetilde{v}}(\cdot,\cdot,k)$ with respect
to $x$ belongs to the Schwartz space ${\mathcal S} (\R^d \times \T)$. Then we apply Fubini's Theorem in
\eqref{lem3-equation3}, and integrate by parts with respect to $x$. As observed above, applying Fubini's Theorem
has been made possible thanks to the new factor $\varphi (x-y)^2$ which makes the integral in $x$ converge. We
make the symmetric operation with $\widetilde{u}$ instead of $\widetilde{v}$ and integrate by parts with respect
to $y$. Eventually, we obtain a formula of the form
\begin{multline}
\label{lem3-equation4}
I = \sum_{\alpha,\beta \in \{ 0,1\}^d} \sum_{\alpha' +\alpha'' \le \alpha}
\sum_{\beta' +\beta'' \le \beta} \star \,
\sum_{k \in \Z} \int_{\R^d \times \R^d \times \T \times \R^d \times \T}
{\rm e}^{i \, (x-y) \cdot \xi} \, {\rm e}^{2 \, i \, \pi \, k \, (\theta-\omega)} \\
\times \partial_x^{\alpha-\alpha'-\alpha''} \, \partial_y^{\beta-\beta'-\beta''} \, \sigma_\flat
(x,\theta,y,\omega,\xi,k) \\
\times \big( \partial^{\alpha'+\beta'} \varphi (x-y) \big) \, U(y,\omega,\xi,k) \,
\big( \partial^{\alpha''+\beta''} \varphi(x-y) \big) \, V(x,\theta,\xi,k) \,
{\rm d}\xi \, {\rm d}x \, {\rm d}\theta \, {\rm d}y \, {\rm d}\omega \, ,
\end{multline}
where the $\star$ coefficients only depend on $d,\alpha,\alpha',\alpha'',\beta,\beta',\beta''$, and where
we have used the notation
\begin{align*}
U(y,\omega,\xi,k) &:= \dfrac{1}{(2\, \pi)^d} \, \int_{\R^d} {\rm e}^{i\, y \cdot \eta} \,
\varphi(\eta-\xi) \, \widehat{\widetilde{u}}(\eta,\omega,k) \, {\rm d}\eta \, ,\\
V(x,\theta,\xi,k) &:= \dfrac{1}{(2\, \pi)^d} \, \int_{\R^d} {\rm e}^{i\, y \cdot \eta} \,
\varphi(\eta+\xi) \, \widehat{\widetilde{v}}(\eta,\theta,k) \, {\rm d}\eta \, .
\end{align*}
Let us now observe that each derivative $\partial_x^{\alpha-\alpha'-\alpha''} \,
\partial_y^{\beta-\beta'-\beta''} \, \sigma_\flat$ that appears in \eqref{lem3-equation4} can be bounded in
$L^\infty$-norm by $C\, \Ng \sigma \Nd_{\rm Amp}$, where the quantity $\Ng \sigma \Nd_{\rm Amp}$ is defined
in \eqref{continuiteL2-2}. Eventually, we apply the Cauchy-Schwarz inequality on $(\R^d \times \T)^2 \times
\R^d \times \Z$ in \eqref{lem3-equation4}, and we thus need to estimate integrals of the form
\begin{equation*}
\sum_{k \in \Z} \int_{\R^d \times \R^d \times \T \times \R^d \times \T}
\big| \partial^{\alpha'+\beta'} \varphi (x-y) \big|^2 \, |U(y,\omega,\xi,k)|^2 \,
{\rm d}\xi \, {\rm d}x \, {\rm d}\theta \, {\rm d}y \, {\rm d}\omega \, ,
\end{equation*}
and symmetric expressions in $V$. The latter integral is computed by first integrating with respect to
$(x,\theta)$. Then we apply Plancherel's Theorem for transforming the integral in $y$ into an integral
in $\eta$. Applying Fubini's Theorem, we can get rid of the integral in $\xi$ (see the above definition
of $U$ in terms of $\widehat{\widetilde{u}}$) and we are left with estimating a quantity of the form
\begin{equation*}
\sum_{k \in \Z} \int_{\R^d \times \T} \big| \widetilde{u} (y,\omega,k) \big|^2 \, {\rm d}y \, {\rm d}\omega \, ,
\end{equation*}
where $\widetilde{u}$ is defined by \eqref{lem3-equation2}. The latter quantity is estimated by using
Parseval-Bessel's equality and Fubini's Theorem again. Eventually, we obtain
\begin{equation*}
\sum_{k \in \Z} \int_{\R^d \times \R^d \times \T \times \R^d \times \T}
\big| \partial^{\alpha'+\beta'} \varphi (x-y) \big|^2 \, \big| U(y,\omega,\xi,k) \big|^2 \,
{\rm d}\xi \, {\rm d}x \, {\rm d}\theta \, {\rm d}y \, {\rm d}\omega \le C \, \| u \|_0^2 \, ,
\end{equation*}
and a similar estimate holds for $V$ in terms of $\| v \|_0$. We have thus proved that there exists a
numerical constant $C$ such that there holds
\begin{equation*}
|I| \le C \, \Ng \sigma \Nd_{\rm Amp} \, \| u \|_0 \, \| v \|_0 \, .
\end{equation*}
In particular, this yields the bound \eqref{continuiteL2-2} when the amplitude $\sigma$ is smooth with
compact support.
\end{proof}

Actually, the proof of Lemma \ref{lem3} even shows the following stronger result which is encoded in the formula
\eqref{lem3-equation4}.

\begin{corollary}
\label{coro1}
Let $\sigma \in {\mathcal C}^\infty_b (\R^d \times \T \times \R^d \times \T \times \R^d \times \Z; \C^{N \times N})$
have compact support with respect to $(\xi,k)$. Let $\{ Z_1,\dots,Z_M \}$ denote the set of all derivatives
$\partial_x^\alpha \, \partial_\theta^j \, \partial_y^\beta \, \partial_\omega^l \, \partial_\xi^\nu$ that appear
in the definition \eqref{continuiteL2-2} of the norm $\Ng \cdot \Nd_{\rm Amp}$.

Then there exist some continuous bilinear mappings
\begin{equation*}
{\mathcal L}_1 \, , \dots \, , {\mathcal L}_M \, : \, {\mathcal S} (\R^d \times \T) \times {\mathcal S} (\R^d \times \T)
\longrightarrow L^1 (\R^d \times \T \times \R^d \times \T \times \R^d \times \Z) \, ,
\end{equation*}
which are independent of $\sigma$, that satisfy a continuity estimate of the form
\begin{equation*}
\forall \, m=1,\dots,M \, ,\quad
\| {\mathcal L}_m (u,v) \|_{L^1(\R^d \times \T \times \R^d \times \T \times \R^d \times \Z)} \le C \,
\| u \|_0 \, \| v \|_0 \, ,
\end{equation*}
and such that for all $u,v \in {\mathcal S} (\R^d \times \T)$, there holds
\begin{equation}
\label{continuiteL2-3}
\int_{\R^d \times \T} \optilde (\sigma) \, u (x,\theta) \, v (v,\theta) \, {\rm d}x \, {\rm d}\theta =
\sum_{m = 1}^M \, \sum_{k \in \Z} \int_{\R^d \times \T \times \R^d \times \T \times \R^d}
(Z_m \, \sigma) \, {\mathcal L}_m (u,v) \, {\rm d}x \, {\rm d}\theta \, {\rm d}y \, {\rm d}\omega \, {\rm d}\xi \, ,
\end{equation}
where the expression of the function $\optilde (\sigma) \, u$ is given in Lemma \ref{lem3}.
\end{corollary}

For a general amplitude $\sigma$ satisfying the assumptions of Theorem \ref{thm2}, we need to define the limit,
as $\delta$ tends to $0$, of the truncated oscillatory integrals \eqref{oscintegral}. The goal is to show that
formula \eqref{continuiteL2-3}, which holds for smooth amplitudes with compact support in $(\xi,k)$, also holds
for the more general class of amplitudes satisfying the assumptions of Theorem \ref{thm2}.

Let therefore $\sigma$ satisfy the assumptions of Theorem \ref{thm2}, and let us define the truncated amplitude
$\sigma_\delta$, $\delta > 0$, by
\begin{equation*}
\sigma_\delta (x,\theta,y,\omega,\xi,k) := \chi_1 (\delta \, k) \, \chi_2 (\delta \, \xi) \,
\sigma (x,\theta,y,\omega,\xi,k) \, .
\end{equation*}
The truncated amplitude $\sigma_\delta$ has as many derivatives as $\sigma$ in $L^\infty$. Moreover, there exists
a constant $C_\chi$ that only depends on $\chi := (\chi_1,\chi_2)$ such that
\begin{equation}
\label{thm2-equation1}
\forall \, \delta \in \, ]0,1] \, ,\quad \Ng \sigma_\delta \Nd_{\rm Amp} \le C_\chi \, \Ng \sigma \Nd_{\rm Amp} \, .
\end{equation}
Let us now consider a nonnegative function $\rho \in {\mathcal C}^\infty_0 (\R^d)$ with integral $1$. We then
define the regularizing kernels
\begin{equation*}
\forall \, n \in \N \, ,\quad \rho_n (x) :=(n+1)^d \, \rho ((n+1)\, x) \, .
\end{equation*}
We also consider the F\'ejer kernel
\begin{equation*}
F_n (\theta) := \dfrac{1}{n+1} \,
\left( \dfrac{\sin ((n+1) \, \pi \, \theta)}{\sin (\pi \, \theta)} \right)^2 \, ,\quad
F_n (0) := n+1 \, ,
\end{equation*}
that belongs to ${\mathcal C}^\infty (\T)$ and whose integral over $\T$ equals $1$. Then we define the
regularized amplitude
\begin{multline}
\label{defsigmadeltan}
\sigma_{\delta,n} (x,\theta,y,\omega,\xi,k) := \int_{\R^d \times \T \times \R^d \times \T \times \R^d}
\rho_n(x-x') \, F_n(\theta-\theta') \, \rho_n(y-y') \, F_n(\omega-\omega') \, \rho_n(\xi-\xi') \\
\sigma_\delta (x',\theta',y',\omega',\xi',k) \,
{\rm d}x' \, {\rm d}\theta' \, {\rm d}y' \, {\rm d}\omega' \, {\rm d}\xi' \, .
\end{multline}
It follows from the classical Theorems of calculus that for all $\delta>0$ and for all integer $n \in \N$,
$\sigma_{\delta,n}$ belongs to ${\mathcal C}^\infty_b (\R^d \times \T \times \R^d \times \T \times \R^d
\times \Z; \C^{N \times N})$ and has compact support in $(\xi,k)$. Differentiating under the integral, we
also have the bound
\begin{equation}
\label{thm2-equation2}
\forall \, n \in \N \, ,\quad \Ng \sigma_{\delta,n} \Nd_{\rm Amp} \le \Ng \sigma_\delta \Nd_{\rm Amp} \, .
\end{equation}
Moreover, since $\sigma_\delta$ is continuous, the sequence $(\sigma_{\delta,n})_{n \in \N}$ converges
pointwise towards $\sigma_\delta$.

For all $u,v \in {\mathcal S}(\R^d \times \T)$, let us define the integral
\begin{equation*}
I_\delta := \int_{\R^d \times \T} T_\delta(x,\theta) \, v(x,\theta) \, {\rm d}x \, {\rm d}\theta \, ,
\end{equation*}
where the function $T_\delta$ is defined by \eqref{oscintegral}. Applying Fubini's Theorem, we have
\begin{equation}
\label{thm2-equation3}
I_\delta = \dfrac{1}{(2\, \pi)^d} \, \sum_{k \in \Z} \int_{\R^d \times \T \times \R^d \times \T \times \R^d}
{\rm e}^{i \, (x-y) \cdot \xi} \, {\rm e}^{2 \, i \, \pi \, k \, (\theta-\omega)} \,
\sigma_\delta (x,\theta,y,\omega,\xi,k) \, u(y,\omega) \, v(x,\theta) \,
{\rm d}x \, {\rm d}\theta \, {\rm d}y \, {\rm d}\omega \, {\rm d}\xi \, .
\end{equation}
We also define the quantity $I_{\delta,n}$ that is the analogue of \eqref{thm2-equation3} with the amplitude
$\sigma_{\delta,n}$ instead of $\sigma_\delta$.

The sequence $(\sigma_{\delta,n})_{n \in \N}$ is bounded in $L^\infty (\R^d \times \T \times \R^d \times \T \times
\R^d \times \Z)$ and it is supported in a fixed compact set with respect to $(\xi,k)$. We can thus apply the dominated
convergence Theorem and obtain that $(I_{\delta,n})_{n \in \N}$ converges towards $I_\delta$ as $n$ tends to $+\infty$.
Moreover, we can apply Lemma \ref{lem3} to the amplitude $\sigma_{\delta,n}$ and derive the bound
\begin{equation*}
|I_{\delta,n}| =\left| \int_{\R^d \times \T} \optilde(\sigma_{\delta,n}) \, u (x,\theta) \, v(x,\theta) \,
{\rm d}x \, {\rm d}\theta \right| \le C \, \Ng \sigma_{\delta,n} \Nd_{\rm Amp} \, \| u \|_0 \, \| v \|_0
\le C_\chi \, \Ng \sigma \Nd_{\rm Amp} \, \| u \|_0 \, \| v \|_0 \, ,
\end{equation*}
where we have used \eqref{thm2-equation2} and \eqref{thm2-equation1}. Passing to the limit as $n$ tends to
$+\infty$, we obtain the uniform bound
\begin{equation}
\label{thm2-equation4}
|I_\delta| \le C_\chi \, \Ng \sigma \Nd_{\rm Amp} \, \| u \|_0 \, \| v \|_0 \, .
\end{equation}
If we can prove that $(I_\delta)_{\delta>0}$ has a limit as $\delta$ tends to $0$, and that this limit
is independent of the truncation function $\chi$, then we shall have shown that the sequence of functions
$(T_\delta)_{\delta>0}$ converges in ${\mathcal S}' (\R^d \times \T)$ towards some limit $\optilde (\sigma)
\, u$. Moreover, the estimate \eqref{thm2-equation4} will show that this distribution coincides with a
function in $L^2 (\R^d \times \T)$ satisfying \eqref{continuiteL2-2}. (If the limit of $(I_\delta)_{\delta>0}$
is independent of $\chi$, then the constant in \eqref{continuiteL2-2} is given by passing to the limit in
\eqref{thm2-equation4} with one particular choice of $\chi$.) It therefore only remains to prove that
$(I_\delta)_{\delta>0}$ has a limit and that this limit is independent of $\chi$.

Since the amplitude $\sigma_{\delta,n}$ is smooth with compact support in $(\xi,k)$, we can apply Corollary
\ref{coro1}. We obtain that $I_{\delta,n}$ can be written under the form
\begin{equation}
\label{thm2-equation5}
I_{\delta,n} = \sum_{m = 1}^M \, \sum_{k \in \Z} \int_{\R^d \times \T \times \R^d \times \T \times \R^d}
(Z_m \, \sigma_{\delta,n}) \, {\mathcal L}_m (u,v) \, {\rm d}x \, {\rm d}\theta \, {\rm d}y \, {\rm d}\omega \,
{\rm d}\xi \, .
\end{equation}
We wish to pass to the limit in \eqref{thm2-equation5}. We first observe that the derivative
$Z_m \, \sigma_{\delta,n}$ is obtained by differentiating under the integral sign in \eqref{defsigmadeltan},
that is
\begin{multline*}
Z_m \, \sigma_{\delta,n} (x,\theta,y,\omega,\xi,k) = \int_{\R^d \times \T \times \R^d \times \T \times \R^d}
\rho_n(x-x') \, F_n(\theta-\theta') \, \rho_n(y-y') \, F_n(\omega-\omega') \, \rho_n(\xi-\xi') \\
Z_m \, \sigma_\delta (x',\theta',y',\omega',\xi',k) \,
{\rm d}x' \, {\rm d}\theta' \, {\rm d}y' \, {\rm d}\omega' \, {\rm d}\xi' \, .
\end{multline*}
Consequently, the right-hand side of \eqref{thm2-equation5} is a finite sum of terms that all have the form
\begin{equation*}
\int_\Upsilon (\varrho_n * h)(\upsilon) \, f(\upsilon) \, {\rm d}\upsilon \, ,\quad
\Upsilon := \R^d \times \T \times \R^d \times \T \times \R^d \, ,
\end{equation*}
with $h \in L^\infty (\Upsilon)$, $f \in L^1 (\Upsilon)$, and $\varrho_n$ is the corresponding regularizing
kernel. (Recall that the sum with respect to $k$ in \eqref{thm2-equation5} involves finitely many terms,
where the number of terms only depends on $\delta$ and not on $n$.) Applying Fubini's Theorem, we can
make the convolution kernel $\varrho_n$ act on $f$ rather than on $h$. This only replaces $\varrho_n$
by $\check{\varrho}_n$ with
\begin{equation*}
\check{\varrho}_n (x,\theta,y,\omega,\xi) := \varrho_n (-x,-\theta,-y,-\omega,-\xi) \, .
\end{equation*}
Then we use the convergence of $\check{\varrho}_n * f$ towards $f$ in $L^1$ (this is a classical result
of convolution that is unfortunately false in $L^\infty$ and this is the reason why we need to switch
the regularization kernel from one function to the other). Hence we can pass to the limit as $n$ tends
to $+\infty$ in \eqref{thm2-equation5}, and obtain
\begin{equation}
\label{thm2-equation6}
\lim_{n \rightarrow +\infty} I_{\delta,n} = I_\delta
= \sum_{m = 1}^M \, \sum_{k \in \Z} \int_{\R^d \times \T \times \R^d \times \T \times \R^d}
(Z_m \, \sigma_\delta) \, {\mathcal L}_m (u,v) \, {\rm d}x \, {\rm d}\theta \, {\rm d}y \, {\rm d}\omega \,
{\rm d}\xi \, .
\end{equation}
In other words, we have extended formula \eqref{continuiteL2-3} to the truncated amplitude $\sigma_\delta$.

It is now straightforward to pass to the limit as $\delta$ tends to $0$. Indeed each derivative
$Z_m \, \sigma_\delta$ can be decomposed under the form
\begin{equation*}
Z_m \, \sigma = \chi_1 (\delta \, k) \, \chi_2 (\delta \, \xi) \, Z_m \, \sigma
+\sum_{m' = 1}^M \eps_{m,m'} (\delta) \, \chi_1 (\delta \, k) \, \chi_{2,m,m'} (\delta \, \xi) \,
Z_{m'} \, \sigma \, ,
\end{equation*}
where $\chi_{2,m,m'} \in {\mathcal C}^\infty_0 (\R^d)$ and $\eps_{m,m'} (\delta)$ tends to $0$ as $\delta$
tends to $0$. We can therefore apply the dominated convergence Theorem in \eqref{thm2-equation6}, and obtain
the expression
\begin{equation*}
\lim_{\delta \rightarrow 0} I_\delta = \sum_{m = 1}^M \, \sum_{k \in \Z}
\int_{\R^d \times \T \times \R^d \times \T \times \R^d} (Z_m \, \sigma) \,
{\mathcal L}_m (u,v) \, {\rm d}x \, {\rm d}\theta \, {\rm d}y \, {\rm d}\omega \, {\rm d}\xi \, ,
\end{equation*}
from which it is clear that the limit is independent of $\chi$. The proof of Theorem \ref{thm2} is complete.
\end{proof}

The proof of Theorem \ref{thm2} even shows that the formula \eqref{continuiteL2-3} still holds under the more
general assumptions of Theorem \ref{thm2}, and that it actually defines the function $\optilde (\sigma) \, u \in
L^2 (\R^d \times \T)$ in a unique way:

\begin{corollary}
\label{coro2}
Let $\sigma : \R^d \times \T \times \R^d \times \T \times \R^d \times \Z \rightarrow \C^{N \times N}$ be
a continuous function satisfying the differentiability assumptions of Theorem \ref{thm2}. Let the bilinear
operators ${\mathcal L}_m$, $m=1,\dots,M$ be defined in Corollary \ref{coro1}. Then for all $u \in {\mathcal S}
(\R^d \times \T;\C^N)$ and for all $v \in {\mathcal S} (\R^d \times \T;\C)$, the function $\optilde (\sigma)
\, u \in L^2 (\R^d \times \T)$ satisfies \eqref{continuiteL2-3}.
\end{corollary}

\begin{remark}
\label{rem2}
Let us assume now that in Theorem \ref{thm2}, the truncation functions $\chi_1,\chi_2$ do not necessarily satisfy
$\chi_1(0) = \chi_2(0) = 1$. Then the corresponding sequence of functions $(T_\delta)_{\delta>0}$ converges in
${\mathcal S}' (\R^d \times \T)$ towards $\chi_1 (0) \, \chi_2 (0) \, \optilde (\sigma) \, u$.
\end{remark}

Let us observe that for an amplitude $\sigma$ that only depends on $(x,\theta,\xi,k)$ and not on $(y,\omega)$,
then the oscillatory integral $\optilde (\sigma) \, u$ coincides with $\op (\sigma) \, u$. This can be checked
directly by applying Fubini's Theorem and the dominated convergence Theorem. In that case, the convergence of
the sequence $(T_\delta)_{\delta>0}$ in ${\mathcal S}'(\R^d \times \T)$ is much easier to obtain. Since we
shall use this argument in what follows, we state the result in a precise way.

\begin{proposition}
\label{prop1}
Let $\sigma : \R^d \times \T \times \R^d \times \T \times \R^d \times \Z \rightarrow \C^{N \times N}$ be
a continuous function that satisfies the differentiability assumptions of Theorem \ref{thm2} and that is
independent of its third and fourth variables: $\sigma (x,\theta,y,\omega,\xi,k) = \sigma_\sharp
(x,\theta,\xi,k)$. Then for all $u \in {\mathcal S} (\R^d \times \T)$, $\optilde (\sigma) \, u$ coincides
with the function $\op (\sigma_\sharp) \, u$ defined in Theorem \ref{thm1}.
\end{proposition}

For simplicity, a function defined on $\R^d \times \T \times \R^d \times \T \times \R^d \times \Z$ that is
independent of its third and fourth variables is equally considered as a function on $\R^d \times \T \times
\R^d \times \Z$, that is we use from now on the same notation for $\sigma$ and $\sigma_\sharp$ in Proposition
\ref{prop1}. We hope that this does not create any confusion. The following result is a more precise comparison
between oscillatory integral operators and pseudodifferential operators. It contains Proposition \ref{prop1}
as a special trivial case. It is also the starting point for the pseudodifferential calculus developed in the
following section.

\begin{proposition}
\label{prop2}
Let $\widetilde{\sigma} \in {\mathcal C}^\infty_b (\R^d \times \T \times \R^d \times \T \times \R^d \times \Z ;
\C^{N \times N})$ be an amplitude, and let the symbol $\sigma \in {\mathcal C}^\infty_b (\R^d \times \T \times
\R^d \times \Z ; \C^{N \times N})$ be defined by
\begin{equation*}
\sigma (x,\theta,\xi,k) := \widetilde{\sigma} (x,\theta,x,\theta,\xi,k) \, .
\end{equation*}
Then the operator $\optilde (\widetilde{\sigma}) - \op (\sigma)$ coincides with $\optilde (r)$, where the
amplitude $r \in {\mathcal C}^\infty_b (\R^d \times \T \times \R^d \times \T \times \R^d \times \Z ; \C^{N \times N})$
is decomposed as
\begin{equation*}
r (x,\theta,y,\omega,\xi,k) = r_1 (x,\theta,y,\omega,\xi,k)
+ R_2 (x,\theta,y,\omega,\xi,k+1) -R_2 (x,\theta,y,\omega,\xi,k) \, ,
\end{equation*}
with
\begin{align*}
r_1 (x,\theta,y,\omega,\xi,k) &:= \dfrac{1}{i} \, \sum_{j=1}^d \int_0^1 \partial_{y_j} \, \partial_{\xi_j}
\widetilde{\sigma} \big( x,\theta,(1-t) \, x+t \, y,\omega,\xi,k \big) \, {\rm d}t \, ,\\
R_2 (x,\theta,y,\omega,\xi,k) &:= \begin{cases}
\dfrac{\widetilde{\sigma} (x,\theta,x,\omega,\xi,k) -\widetilde{\sigma} (x,\theta,x,\theta,\xi,k)}
{1-{\rm e}^{-2\, i \, \pi \, (\omega-\theta)}} \, ,& \text{\rm if $\omega \neq \theta$,} \\
\dfrac{1}{2 \, i \, \pi} \, \partial_\omega \widetilde{\sigma} (x,\theta,x,\theta,\xi,k)
\, ,& \text{\rm if $\omega = \theta$.}
\end{cases}
\end{align*}
\end{proposition}

We observe that the amplitude $R_2$ does not depend on $y$ but it depends on $\omega$, so it does not enter
the framework of Proposition \ref{prop1}.

\begin{proof}[Proof of Proposition \ref{prop2}]
Let us first assume that the amplitude $\widetilde{\sigma}$ has also compact support in $(\xi,k)$. In that
case, the symbol $\sigma$ has compact support in $(\xi,k)$, and we can apply Proposition \ref{prop1} and
Lemma \ref{lem3}:
\begin{multline*}
\optilde (\widetilde{\sigma}) \, u (x,\theta) - \op (\sigma) \, u (x,\theta)
= \optilde (\widetilde{\sigma}-\sigma) \, u (x,\theta) \\
= \dfrac{1}{(2\, \pi)^d} \, \sum_{k \in \Z} \int_{\R^d \times \T \times \R^d}
{\rm e}^{i \, (x-y) \cdot \xi} \, {\rm e}^{2 \, i \, \pi \, k \, (\theta-\omega)}
\, \big( \widetilde{\sigma} (x,\theta,y,\omega,\xi,k) -\widetilde{\sigma} (x,\theta,x,\omega,\xi,k) \big) \,
u(y,\omega) \, {\rm d}y \, {\rm d}\omega \, {\rm d}\xi \\
+ \dfrac{1}{(2\, \pi)^d} \, \sum_{k \in \Z} \int_{\R^d \times \T \times \R^d}
{\rm e}^{i \, (x-y) \cdot \xi} \, {\rm e}^{2 \, i \, \pi \, k \, (\theta-\omega)}
\, \big( \widetilde{\sigma} (x,\theta,x,\omega,\xi,k) -\widetilde{\sigma} (x,\theta,x,\theta,\xi,k) \big) \,
u(y,\omega) \, {\rm d}y \, {\rm d}\omega \, {\rm d}\xi \, .
\end{multline*}
Let us start with the first term on the right-hand side. Applying Taylor's formula, we get
\begin{equation*}
\widetilde{\sigma} (x,\theta,y,\omega,\xi,k) -\widetilde{\sigma} (x,\theta,x,\omega,\xi,k)
= -\dfrac{1}{i} \, \sum_{j=1}^d i \, (x_j-y_j) \, \int_0^1 \partial_{y_j} \widetilde{\sigma}
\big( x,\theta,(1-t) \, x+t \, y,\omega,\xi,k \big) \, {\rm d}t \, ,
\end{equation*}
then we integrate by parts with respect to $\xi$ and we already obtain
\begin{multline*}
\dfrac{1}{(2\, \pi)^d} \, \sum_{k \in \Z} \int_{\R^d \times \T \times \R^d}
{\rm e}^{i \, (x-y) \cdot \xi} \, {\rm e}^{2 \, i \, \pi \, k \, (\theta-\omega)}
\, \big( \widetilde{\sigma} (x,\theta,y,\omega,\xi,k) -\widetilde{\sigma} (x,\theta,x,\omega,\xi,k) \big) \,
u(y,\omega) \, {\rm d}y \, {\rm d}\omega \, {\rm d}\xi \\
= \optilde (r_1) \, u (x,\theta) \, .
\end{multline*}
All manipulations are made possible by the compact support assumption with respect to $(\xi,k)$ and the
fact that $u$ belongs to ${\mathcal S} (\R^d \times \T)$.

Let us now study the second term in the decomposition of $\optilde (\widetilde{\sigma}-\sigma) \, u$. By
standard results of calculus, the function $R_2$ defined in Proposition \ref{prop2} is $1$-periodic with
respect to $\theta$ and $\omega$, and is smooth (namely, ${\mathcal C}^\infty_b$) with respect to all its
arguments. (The reason why we divide by $1 -{\rm e}^{-2\, i \, \pi \, (\omega-\theta)}$ and not by $\omega
-\theta$ in the definition of $R_2$ is to keep the periodicity with respect to both $\theta$ and $\omega$.
However, this is of little consequence, and $R_2$ basically counts as one $\omega$-derivative of the
amplitude $\widetilde{\sigma}$.) We apply Abel's transformation and obtain
\begin{multline*}
\dfrac{1}{(2\, \pi)^d} \, \sum_{k \in \Z} \int_{\R^d \times \T \times \R^d}
{\rm e}^{i \, (x-y) \cdot \xi} \, {\rm e}^{2 \, i \, \pi \, k \, (\theta-\omega)}
\, \big( \widetilde{\sigma} (x,\theta,x,\omega,\xi,k) -\widetilde{\sigma} (x,\theta,x,\theta,\xi,k) \big) \,
u(y,\omega) \, {\rm d}y \, {\rm d}\omega \, {\rm d}\xi \\
= \dfrac{1}{(2\, \pi)^d} \, \sum_{k \in \Z} \int_{\R^d \times \T \times \R^d}
{\rm e}^{i \, (x-y) \cdot \xi} \,
\big( {\rm e}^{2 \, i \, \pi \, k \, (\theta-\omega)} - {\rm e}^{2 \, i \, \pi \, (k+1) \, (\theta-\omega)} \big)
\, R_2 (x,\theta,\omega,\xi,k) \, u(y,\omega) \, {\rm d}y \, {\rm d}\omega \, {\rm d}\xi \\
= \dfrac{1}{(2\, \pi)^d} \, \sum_{k \in \Z} \int_{\R^d \times \T \times \R^d}
{\rm e}^{i \, (x-y) \cdot \xi} \, {\rm e}^{2 \, i \, \pi \, k \, (\theta-\omega)}
\big( R_2 (x,\theta,\omega,\xi,k+1) -R_2 (x,\theta,\omega,\xi,k) \big)
u(y,\omega) \, {\rm d}y \, {\rm d}\omega \, {\rm d}\xi \, .
\end{multline*}
We have thus proved the result announced in Proposition \ref{prop2} under the additional assumption that
the amplitude $\widetilde{\sigma}$ has compact support in $(\xi,k)$.

When the amplitude $\widetilde{\sigma}$ does not necessarily have compact support in $(\xi,k)$, we approximate
$\widetilde{\sigma}$ by a sequence $\widetilde{\sigma}_\delta$, $\delta>0$, as in Theorem \ref{thm2}. We leave
as an exercise to the reader the verification that for the corresponding sequence of amplitudes
$(r_\delta)_{\delta \in \, ]0,1]}$, there holds
\begin{equation*}
\forall \, u \in {\mathcal S} (\R^d \times \T) \, ,\quad
\lim_{\delta \rightarrow 0} \optilde (r_\delta) \, u = \optilde (r) \, u \, ,
\end{equation*}
where the limit is understood in the sense of ${\mathcal S}' (\R^d \times \T)$ (use Remark \ref{rem2}). This
completes the proof of Proposition \ref{prop2}.
\end{proof}

We have only proved Proposition \ref{prop2} for very smooth amplitudes. In the following Section, we shall
extend this decomposition to amplitudes with finite regularity by the standard smoothing procedure. At this
stage, we feel free to shorten some of the arguments in the proof when they are completely similar to what
we have already explained.

\section{Singular pseudodifferential calculus I. Definition of operators and action on Sobolev spaces}
\label{sect4}

\subsection{Singular symbols and singular pseudodifferential operators}

Following \cite{williams3}, we now introduce the singular symbols and their associated operators. The classes
of symbols are defined by first considering the following sets.

\begin{definition}
\label{def1}
Let $q \ge 1$, and let ${\mathcal O} \subset \R^q$ be an open set that contains the origin. Let $m \in \R$. Then
we let ${\bf S}^m ({\mathcal O})$ denote the class of all functions $\sigma : {\mathcal O} \times \R^d \times
[1,+\infty[ \rightarrow \C^{N \times N}$ such that
\begin{itemize}
 \item[{\rm (i)}] for all $\gamma \ge 1$, $\sigma (\cdot,\cdot,\gamma)$ is ${\mathcal C}^\infty$ on ${\mathcal O}
                  \times \R^d$,
 \item[{\rm (ii)}] for all compact subset $K$ of ${\mathcal O}$, for all $\alpha \in \N^q$ and for all $\nu \in
                   \N^d$, there exists a constant $C_{\alpha,\nu,K}$ satisfying
\begin{equation*}
\sup_{v \in K} \, \sup_{\xi \in \R^d} \, \sup_{\gamma \ge 1} \, \, (\gamma^2+|\xi|^2)^{-(m-|\nu|)/2} \,
\big| \partial_v^\alpha \, \partial_\xi^\nu \, \sigma \, (v,\xi,\gamma) \big| \le C_{\alpha,\nu,K} \, .
\end{equation*}
\end{itemize}
\end{definition}

\noindent Let us now define the singular symbols.

\begin{definition}[Singular symbols]
\label{def2}
Let $m \in \R$, and let $n \in \N$. Then we let $S^m_n$ denote the set of families of functions
$(a_{\eps,\gamma})_{\eps \in ]0,1],\gamma \ge 1}$ that are constructed as follows:
\begin{equation}
\label{singularsymbol}
\forall \, (x,\theta,\xi,k) \in \R^d \times \T \times \R^d \times \Z \, ,\quad
a_{\eps,\gamma} (x,\theta,\xi,k) = \sigma \left(
\eps \, V(x,\theta),\xi+\dfrac{2\, \pi \, k \, \beta}{\eps},\gamma \right) \, ,
\end{equation}
where $\sigma \in {\bf S}^m({\mathcal O})$, $V$ belongs to the space ${\mathcal C}^n_b (\R^d \times \T)$ and
where furthermore $V$ takes its values in a convex compact subset $K$ of ${\mathcal O}$ that contains the
origin (for instance $K$ can be a closed ball centered round the origin).
\end{definition}

In Definition \ref{def2}, we ask the function $V$ to take its values in a convex compact subset $K$ of
${\mathcal O}$ so that for all $\eps \in \, ]0,1]$, the function $\eps \, V$ takes its values in the same
convex compact set $K$. This property is used in several places below to derive uniform $L^\infty$ bounds
with respect to the small parameter $\eps$.

For simplicity, we shall not mention that $S^m_n$ depends on the open set ${\mathcal O}$. (It will be
convenient from time to time to let ${\mathcal O}$ denote various possible open sets.) With a slight abuse
in the terminology, we shall refer to the elements of $S^m_n$ as symbols rather than as families of symbols.
We hope that this does not create any confusion.

To each symbol $a = (a_{\eps,\gamma})_{\eps \in ]0,1],\gamma \ge 1} \in S^m_n$ given by the formula
\eqref{singularsymbol}, we associate a singular pseudodifferential operator $\opeg (a)$, with $\eps \in \, ]0,1]$
and $\gamma \ge 1$, whose action on a function $u \in {\mathcal S} (\R^d \times \T ; \C^N)$ is defined by
\begin{equation}
\label{singularpseudo}
\opeg (a) \, u \, (x,\theta) := \dfrac{1}{(2\, \pi)^d} \, \sum_{k \in \Z} \int_{\R^d}
{\rm e}^{i\, x \cdot \xi} \, {\rm e}^{2\, i\, \pi \, k \, \theta} \,
\sigma \left( \eps \, V(x,\theta),\xi+\dfrac{2 \, \pi \, k \, \beta}{\eps},\gamma \right) \, \widehat{c_k(u)} (\xi)
\, {\rm d}\xi \, .
\end{equation}
Let us briefly note that for the Fourier multiplier $\sigma (v,\xi,\gamma) =i\, \xi_1$, the corresponding
singular operator is $\partial_{x_1} +(\beta_1/\eps) \, \partial_\theta$. We now wish to describe the action
of singular pseudodifferential operators on Sobolev spaces. As can be expected from this simple example,
the natural framework is provided by the spaces $H^{s,\eps}$ defined in Section \ref{sect2}. The following
result is a direct consequence of Theorem \ref{thm1}.

\begin{proposition}
\label{prop3}
Let $n \ge d+1$, and let $a \in S^m_n$ with $m \le 0$. Then $\opeg (a)$ in \eqref{singularpseudo} defines
a bounded operator on $L^2 (\R^d \times \T)$: there exists a constant $C>0$, that only depends on $\sigma$
and $V$ in the representation \eqref{singularsymbol}, such that for all $\eps \in \, ]0,1]$ and for all
$\gamma \ge 1$, there holds
\begin{equation*}
\forall \, u \in {\mathcal S} (\R^d \times \T) \, ,\quad \left\| \opeg (a) \, u \right\|_0
\le \dfrac{C}{\gamma^{|m|}} \, \| u \|_0 \, .
\end{equation*}
\end{proposition}

Let us observe that if we compare Proposition \ref{prop3} with \cite[Proposition 1.1]{williams3}, we obtain
the same result with slightly less regularity on $V$, and above all without the compact support assumption
on the function $V$. The constant $C$ in Proposition \ref{prop3} depends uniformly on the compact set in
which $V$ takes its values and on the norm of $V$ in ${\mathcal C}^{d+1}_b$. Even when we do not state it
so clearly, all constants in the results below will depend uniformly on a finite number of derivatives of
the symbols (or amplitudes).

\begin{proof}[Proof of Proposition \ref{prop3}]
We wish to apply Theorem \ref{thm1}, so the only thing to check is that the symbol $a_{\eps,\gamma}$ defined
by \eqref{singularsymbol} satisfies a bound of the form
\begin{equation*}
\forall \, \eps \in \, ]0,1] \, ,\quad \forall \, \gamma \ge 1 \, ,\quad
\Ng a_{\eps,\gamma} \Nd \le \dfrac{C_{\sigma,V}}{\gamma^{|m|}} \, .
\end{equation*}
For instance, the proof of the $L^\infty$ bound follows from Definitions \ref{def1} and \ref{def2}. Let us
recall that for all $\eps \in \, ]0,1]$, $\eps \, V$ takes its values in a fixed convex compact subset $K
\subset {\mathcal O}$ (because $K$ has been assumed to contain the origin, see Definition \ref{def2}), so
we have
\begin{equation*}
\left| \sigma \left( \eps \, V(x,\theta),\xi+\dfrac{2\, \pi \, k \, \beta}{\eps},\gamma \right) \right|
\le C_{0,0,K} \, \left( \gamma^2 +\left| \xi+\dfrac{2\, \pi \, k \, \beta}{\eps} \right|^2 \right)^{m/2}
\le \dfrac{C}{\gamma^{|m|}} \, .
\end{equation*}
The $L^\infty$ bounds for the derivatives of $a_{\eps,\gamma}$ follow by applying the Fa\`a di Bruno formula
for the composition of functions. We omit the details.
\end{proof}

\begin{remark}
\label{rem3}
The result of Proposition \ref{prop3} does not rely on the scaling of the substitution in the representation
\eqref{singularsymbol}. More precisely, the same result would hold with the substitution $V(x,\theta)$ instead
of $\eps \, V(x,\theta)$. The only important point in the proof is the fact that the function substituted in
the $v$-variable takes its values in a compact subset of ${\mathcal O}$ that is independent of $\eps$,
and that sufficiently many of its derivatives belong to $L^\infty$. This fact will be used several times in
what follows.
\end{remark}

\noindent There is no great difficulty in extending Proposition \ref{prop3} to symbols of positive degree.

\begin{proposition}
\label{prop4}
Let $n \ge d+1$, and let $a \in S^m_n$ with $m>0$. Then $\opeg (a)$ defines a bounded operator from
$H^{m,\eps} (\R^d \times \T)$ to $L^2 (\R^d \times \T)$ with a norm that is independent of $\eps,\gamma$: there
exists a constant $C>0$, that only depends on $\sigma$ and $V$ in the representation \eqref{singularsymbol},
such that for all $\eps \in \, ]0,1]$ and for all $\gamma \ge 1$, there holds
\begin{equation*}
\forall \, u \in {\mathcal S} (\R^d \times \T) \, ,\quad \left\| \opeg (a) \, u \right\|_0
\le C \, \| u \|_{H^{m,\eps},\gamma} \, .
\end{equation*}
\end{proposition}

\begin{proof}[Proof of Proposition \ref{prop4}]
It is sufficient to write the symbol $a_{\eps,\gamma}$ as
\begin{equation*}
a_{\eps,\gamma} (x,\theta,\xi,k) \,
\left( \gamma^2 +\left| \xi+\dfrac{2\, \pi \, k \, \beta}{\eps} \right|^2 \right)^{-m/2} \,
\left( \gamma^2 +\left| \xi+\dfrac{2\, \pi \, k \, \beta}{\eps} \right|^2 \right)^{m/2} \, ,
\end{equation*}
to observe that the symbol
\begin{equation*}
(x,\theta,\xi,k) \longmapsto
a_{\eps,\gamma} (x,\theta,\xi,k) \,
\left( \gamma^2 +\left| \xi+\dfrac{2\, \pi \, k \, \beta}{\eps} \right|^2 \right)^{-m/2}
\end{equation*}
belongs to $S^0_n$, and eventually to observe that the Fourier multiplier with symbol
\begin{equation*}
\left( \gamma^2 +\left| \xi+\dfrac{2\, \pi \, k \, \beta}{\eps} \right|^2 \right)^{m/2}
\end{equation*}
is an isometry from $H^{m,\eps} (\R^d \times \T)$ - equipped with the norm $\| \cdot \|_{H^{m,\eps},\gamma}$ -
to $L^2 (\R^d \times \T)$.
\end{proof}

The result of Proposition \ref{prop3} can be made more precise when the degree $m$ of the symbol is negative.
We shall not deal with the general case $m<0$ since in what follows, the case $m=-1$ will be our main concern.
Our result is the following.

\begin{proposition}
\label{prop5}
Let $n \ge d+2$, and let $a \in S^{-1}_n$. Then $\opeg (a)$ defines a bounded operator from $L^2 (\R^d \times \T)$
to $H^{1,\eps} (\R^d \times \T)$ with a norm that is independent of $\eps,\gamma$: there exists a constant $C>0$,
that only depends on $\sigma$ and $V$ in the representation \eqref{singularsymbol}, such that for all $\eps \in
\, ]0,1]$ and for all $\gamma \ge 1$, there holds
\begin{equation*}
\forall \, u \in {\mathcal S} (\R^d \times \T) \, ,\quad \left\| \opeg (a) \, u \right\|_{H^{1,\eps},\gamma}
\le C \, \| u \|_0 \, .
\end{equation*}
\end{proposition}

Let us observe that the regularizing effect of Proposition \ref{prop5} requires an additional space derivative
on the symbol compared to the $L^2$ bound of Propositions \ref{prop3} and \ref{prop4}. This is the first occurence
in this article of the general principle that ``symbolic calculus (and not only boundedness of operators) requires
spatial regularity''. Here, we study the action of the composition
\begin{equation*}
\left( \partial_{x_j} +\dfrac{\beta_j}{\eps} \, \partial_\theta \right) \, \opeg (a) \, .
\end{equation*}

\begin{proof}[Proof of Proposition \ref{prop5}]
We first observe that Proposition \ref{prop3} already gives the estimate
\begin{equation*}
\forall \, u \in {\mathcal S} (\R^d \times \T) \, ,\quad \left\| \opeg (a) \, u \right\|_0
\le \dfrac{C}{\gamma} \, \| u \|_0 \, .
\end{equation*}
Using the definition \eqref{defnormeepsgamma} of the norm $\| \cdot \|_{H^{1,\eps},\gamma}$, we see that it only
remains to prove some bounds of the form
\begin{equation}
\label{estimprop5}
\forall \, j=1,\dots,d \, ,\quad \forall \, u \in {\mathcal S} (\R^d \times \T) \, ,\quad
\left\| \left( \partial_{x_j} +\dfrac{\beta_j}{\eps} \, \partial_\theta \right) \opeg (a) \, u \right\|_0
\le C \, \| u \|_0 \, ,
\end{equation}
with a constant $C$ that is independent of $\eps,\gamma,u$. We prove such a bound in the case $j=1$
(this is only to simplify the notation).

We can differentiate under the integral sign in the definition of $\opeg (a) \, u$, see
\eqref{singularpseudo}, obtaining
\begin{equation*}
\left( \partial_{x_1} +\dfrac{\beta_1}{\eps} \, \partial_\theta \right) \opeg (a) \, u \, (x,\theta)
= (T_1 + T_2 + T_3) (x,\theta) \, ,
\end{equation*}
where we use the notation
\begin{align*}
T_1 (x,\theta) &:= \dfrac{1}{(2\, \pi)^d} \, \sum_{k \in \Z} \int_{\R^d} {\rm e}^{i\, x \cdot \xi} \,
{\rm e}^{2\, i\, \pi \, k \, \theta} \, i\, \left( \xi_1+\dfrac{2 \, \pi \, k \, \beta_1}{\eps} \right) \,
\sigma \left( \eps \, V(x,\theta),\xi+\dfrac{2 \, \pi \, k \, \beta}{\eps},\gamma \right) \, \widehat{c_k(u)} (\xi)
\, {\rm d}\xi \, ,\\
T_2 (x,\theta) &:= \dfrac{1}{(2\, \pi)^d} \, \sum_{k \in \Z} \int_{\R^d}
{\rm e}^{i\, x \cdot \xi} \, {\rm e}^{2\, i\, \pi \, k \, \theta} \,
\left[ \partial_v \sigma \left( \eps \, V(x,\theta),\xi+\dfrac{2 \, \pi \, k \, \beta}{\eps},\gamma \right)
\cdot \eps \, \partial_{x_1} \, V(x,\theta) \right] \, \widehat{c_k(u)} (\xi) \, {\rm d}\xi \, ,\\
T_3 (x,\theta) &:= \dfrac{\beta_1}{(2\, \pi)^d} \, \sum_{k \in \Z} \int_{\R^d}
{\rm e}^{i\, x \cdot \xi} \, {\rm e}^{2\, i\, \pi \, k \, \theta} \,
\left[ \partial_v \sigma \left( \eps \, V(x,\theta),\xi+\dfrac{2 \, \pi \, k \, \beta}{\eps},\gamma \right)
\cdot \partial_\theta \, V(x,\theta) \right] \, \widehat{c_k(u)} (\xi) \, {\rm d}\xi \, .
\end{align*}
The terms $T_1$ and $T_2$ fall into the framework of Proposition \ref{prop3}. Indeed, the function
\begin{equation*}
\sigma_\flat (v,\xi,\gamma) := i\, \xi_1 \, \sigma (v,\xi,\gamma) \, ,
\end{equation*}
belongs to ${\bf S}^0$, since $\sigma$ belongs to ${\bf S}^{-1}$. Consequently, the term $T_1$ reads
$\opeg (a_\flat) \, u$ where the singular symbol $a_\flat$ belongs to $S^0_n$, $n \ge d+2$. In the same
spirit, the term $T_2$ reads $\opeg (a_\sharp) \, u$ where the singular symbol $a_\sharp$ belongs to
$S^{-1}_{n-1}$, $n-1 \ge d+1$ (use the substitution $(\eps \, V,\eps \, W)$ with $W:=\partial_{x_1} V$
in the symbol $\partial_v \sigma (v,\xi,\gamma) \cdot w$). We can thus apply Proposition \ref{prop3} to
estimate $T_1$ and $T_2$ in $L^2 (\R^d \times \T)$.

The remaining term $T_3$ does not fall directly into the framework of Proposition \ref{prop3} since there is
an $\eps$ missing in front of $\partial_\theta \, V$, so we do not exactly have a singular pseudodifferential
operator as defined in \eqref{singularpseudo}. However, we can still apply Theorem \ref{thm1} (see Remark
\ref{rem3}) to the symbol
\begin{equation*}
(x,\theta,\xi,k) \longmapsto \partial_v \sigma
\left( \eps \, V(x,\theta),\xi+\dfrac{2 \, \pi \, k \, \beta}{\eps},\gamma \right)
\cdot \partial_\theta \, V(x,\theta) \, .
\end{equation*}
Since $V$ belongs to ${\mathcal C}^n_b (\R^d \times \T)$ with $n \ge d+2$, the latter symbol is bounded and
it has exactly as many derivatives in $L^\infty$ as required in order to apply Theorem \ref{thm1}, and the
$L^\infty$ bounds on the symbol are independent of $\eps \in \, ]0,1]$ and $\gamma \ge 1$. We can therefore
apply Theorem \ref{thm1} in order to estimate $T_3$ in $L^2 (\R^d \times \T)$. The estimates of $T_1,T_2$
and $T_3$ yield \eqref{estimprop5}, so the proof of Proposition \ref{prop5} is complete.
\end{proof}

\begin{remark}
\label{rem4}
It would be tempting to extrapolate from Propositions \ref{prop3} and \ref{prop5} that symbols in $S_{-m}^n$,
$m \in \N$ and $n$ sufficiently large, define pseudodifferential operators that act from $L^2$ to $H^{m,\eps}$.
This is true indeed, but unfortunately the operator norm seems to blow up with $\eps$ as soon as $m$ is larger
than $2$ (as soon as $m$ is larger than $2$, one faces a derivative $(\partial_\theta/\eps)^2$ and the factor
$\eps^{-2}$ is too large when acting on the function $\eps\, V$). We thus need to pay special attention and
check carefully each result one by one in order to prove uniform bounds.
\end{remark}

The proof of Proposition \ref{prop5} can be adapted without any difficulty to show that singular
pseudodifferential operators with symbols of degree $0$ act boundedly on $H^{1,\eps}$ and not only on
$L^2$. We feel free to omit the proof of this result which will be useful later on.

\begin{lemma}
\label{lem4}
Let $n \ge d+2$, and let $a \in S^0_n$. Then $\opeg (a)$ acts boundedly on $H^{1,\eps} (\R^d \times \T)$ with
a norm that is independent of $\eps,\gamma$: there exists a constant $C>0$, that only depends on $\sigma$ and
$V$ in the representation \eqref{singularsymbol}, such that for all $\eps \in \, ]0,1]$ and for all $\gamma
\ge 1$, there holds
\begin{equation*}
\forall \, u \in {\mathcal S} (\R^d \times \T) \, ,\quad \left\| \opeg (a) \, u \right\|_{H^{1,\eps},\gamma}
\le C \, \| u \|_{H^{1,\eps},\gamma} \, .
\end{equation*}
\end{lemma}

\subsection{Singular amplitudes and singular oscillatory integral operators}

The result of Proposition \ref{prop3} can be generalized to singular amplitudes by using Theorem \ref{thm2}
instead of Theorem \ref{thm1}. More precisely, let us first define the classes of singular amplitudes.

\begin{definition}[Singular amplitudes]
\label{def3}
Let $m \in \R$, and let $n \in \N$. Then we let $A^m_n$ denote the set of families of functions
$(\widetilde{a}_{\eps,\gamma})_{\eps \in ]0,1],\gamma \ge 1}$ that are constructed as follows:
\begin{multline}
\label{singularamplitude}
\forall \, (x,\theta,y,\omega,\xi,k) \in \R^d \times \T \times \R^d \times \T \times \R^d \times \Z \, ,\\
\widetilde{a}_{\eps,\gamma} (x,\theta,y,\omega,\xi,k) := \sigma \left(
\eps \, V(x,\theta),\eps \, W(y,\omega),\xi+\dfrac{2\, \pi \, k \, \beta}{\eps},\gamma \right) \, ,
\end{multline}
where $\sigma \in {\bf S}^m({\mathcal O}_1 \times {\mathcal O}_2)$, $V$ and $W$ belong to the space
${\mathcal C}^n_b (\R^d \times \T)$, and where furthermore $V$, resp. $W$, takes its values in a convex
compact subset $K_1$, resp. $K_2$, of ${\mathcal O}_1$, resp. ${\mathcal O}_2$, that contains the origin.
\end{definition}

To each amplitude $\widetilde{a} =(\widetilde{a}_{\eps,\gamma})_{\eps \in \, ]0,1], \gamma \ge 1} \in A^m_n$
given by the formula \eqref{singularamplitude}, we wish to associate a singular oscillatory integral operator
$\opteg (\widetilde{a})$, that would be defined (formally at first) by
\begin{equation*}
\forall \, \eps \in \, ]0,1] \, ,\quad \forall \, \gamma \ge 1 \, ,\quad
\opteg (\widetilde{a}) := \optilde (\widetilde{a}_{\eps,\gamma}) \, ,
\end{equation*}
and the oscillatory integral operator $\optilde$ is introduced in Theorem \ref{thm2}. The problem is that,
at this point of the analysis, the operator $\optilde$ has only been defined for bounded amplitudes that
have sufficiently many derivatives in $L^\infty$, see Theorem \ref{thm2}. We can therefore only define
$\opteg (\widetilde{a})$ for nonpositive degrees $m$. The following result generalizes
\cite[Proposition 2.2]{williams3}. The proof follows exactly that of Proposition \ref{prop3} above,
except that we use Theorem \ref{thm2} instead of Theorem \ref{thm1}.

\begin{proposition}
\label{prop6}
Let $n \ge d+1$, and let $\widetilde{a} \in A^m_n$ with $m \le 0$. Then for all $\eps \in \, ]0,1]$ and for
all $\gamma \ge 1$, the amplitude $\widetilde{a}_{\eps,\gamma}$ satisfies the assumptions of Theorem \ref{thm2}.
Moreover $\opteg (\widetilde{a})$ defines a bounded operator on $L^2 (\R^d \times \T)$: there exists a constant
$C>0$, that only depends on $\sigma$, $V$ and $W$ in the representation \eqref{singularamplitude}, such that
for all $\eps \in \, ]0,1]$ and for all $\gamma \ge 1$, there holds
\begin{equation*}
\forall \, u \in {\mathcal S} (\R^d \times \T) \, ,\quad \left\| \opteg (\widetilde{a}) \, u \right\|_0
\le \dfrac{C}{\gamma^{|m|}} \, \| u \|_0 \, .
\end{equation*}
\end{proposition}

The derivatives $\partial_x^\alpha \, \partial_\theta^j \, \partial_y^\nu \, \partial_\omega^l \, \partial_\xi^\mu
\widetilde{a}_{\eps,\gamma}$ are computed in the classical sense and all of them are continous bounded functions
on $\R^d \times \T \times \R^d \times \T \times \R^d \times \Z$. These derivatives are obtained by applying the
Fa\`a di Bruno formula.

Remark \ref{rem3} still applies, meaning that the result of Proposition \ref{prop6} would still hold if we
had made the substitution $(v,w) \rightarrow (V(x,\theta),W(y,\omega))$ instead of $(v,w) \rightarrow (\eps
\, V(x,\theta),\eps \, W(y,\omega))$. Here, the small parameter $\eps$ is not crucial in order to derive the
uniform $L^\infty$ bound on the symbol.

In the same way as we proved a regularization effect for singular pseudodifferential operators with symbols
of negative order, we are going to prove a regularization effect for singular oscillatory integrals operators
when the amplitude has negative order and is sufficiently smooth.

\begin{proposition}
\label{prop7}
Let $n \ge d+2$, and let $\widetilde{a} \in A^{-1}_n$. Then the oscillatory integral operator
$\opteg (\widetilde{a})$ is bounded from $L^2 (\R^d \times \T)$ to $H^{1,\eps} (\R^d \times \T)$. More
precisely, there exists a constant $C>0$, that only depends on $\sigma$, $V$ and $W$ in the representation
\eqref{singularamplitude}, such that for all $\eps \in \, ]0,1]$ and for all $\gamma \ge 1$, there holds
\begin{equation*}
\forall \, u \in {\mathcal S} (\R^d \times \T) \, ,\quad
\left\| \opteg (\widetilde{a}) \, u \right\|_{H^{1,\eps},\gamma} \le C \, \| u \|_0 \, .
\end{equation*}
Moreover, the derivatives of $\opteg (\widetilde{a}) \, u$ are computed by differentiating formally under
the integral sign.
\end{proposition}

\begin{proof}[Proof of Proposition \ref{prop7}]
In order to prove Proposition \ref{prop7}, we need to go back to the definition of the oscillatory integral
operator $\optilde$ in Theorem \ref{thm2}. Let $n \ge d+2$, $\widetilde{a} \in A^{-1}_n$, and let $u \in
{\mathcal S} (\R^d \times \T)$. Let now $\chi_1 \in {\mathcal C}^\infty_0 (\R)$ and $\chi_2 \in
{\mathcal C}^\infty_0 (\R^d)$ satisfy $\chi_1(0) = \chi_2 (0) = 1$. According to Theorem \ref{thm2}, we know
that the function $\opteg (\widetilde{a}) \, u \in L^2 (\R^d \times \T)$ is the limit in ${\mathcal S}'
(\R^d \times \T)$, as $\delta$ tends to $0$, of the sequence of functions
\begin{multline*}
T_\delta(x,\theta) := \dfrac{1}{(2\, \pi)^d} \, \sum_{k \in \Z} \chi_1 (\delta \, k) \,
\int_{\R^d \times \R^d \times \T} {\rm e}^{i \, (x-y) \cdot \xi} \, {\rm e}^{2 \, i \, \pi \, k \, (\theta-\omega)} \\
\chi_2 (\delta \, \xi) \, \sigma \left(
\eps \, V(x,\theta),\eps \, W(y,\omega),\xi+\dfrac{2\, \pi \, k \, \beta}{\eps},\gamma \right)
\, u(y,\omega) \, {\rm d}\xi \, {\rm d}y \, {\rm d}\omega \, .
\end{multline*}
Moreover, Proposition \ref{prop6} already gives the estimate
\begin{equation}
\label{estim1prop7}
\left\| \opteg (\widetilde{a}) \, u \right\|_0 \le \dfrac{C}{\gamma} \, \| u \|_0 \, .
\end{equation}

Each function $T_\delta$ is bounded. Moreover, we have shown in the proof of Theorem \ref{thm2} that
the sequence $(T_\delta)_{\delta \in \, ]0,1]}$ is bounded in $L^2 (\R^d \times \T)$ and converges in
${\mathcal S}' (\R^d \times \T)$. This boundedness property follows from the relation \eqref{thm2-equation6},
Corollary \ref{coro1} and \eqref{thm2-equation1}.

Let $j \in \{ 1, \dots, d \}$. We are going to prove that there exists a constant $C$, that is independent
of $\delta,\eps,\gamma,u$, such that
\begin{equation}
\label{estim2prop7}
\left\| \left( \partial_{x_j} +\dfrac{\beta_j}{\eps} \, \partial_\theta \right) T_\delta \right\|_0
\le C \, \| u \|_0 \, .
\end{equation}
Combining \eqref{estim2prop7} with \eqref{estim1prop7}, we shall obtain the result of Proposition \ref{prop7}.
Indeed, the uniform bound \eqref{estim2prop7} is sufficient to show that the limit of $T_\delta$ in ${\mathcal S}'
(\R^d \times \T)$ belongs to $H^{1,\eps} (\R^d \times \T)$. (Here, we use the classical weak convergence
argument and the uniqueness of the limit in the sense of distributions.) We thus focus on the derivation
of the bound \eqref{estim2prop7} for $j=1$.

Each function $T_\delta$ has ${\mathcal C}^1$ regularity, and can be differentiated under the integral sign
by applying standard rules of calculus. We obtain
\begin{equation*}
\left( \partial_{x_1} +\dfrac{\beta_1}{\eps} \, \partial_\theta \right) T_\delta
= T_{1,\delta} +T_{2,\delta} +T_{3,\delta} \, ,
\end{equation*}
where, similarly to the proof of Proposition \ref{prop5}, we use the notation
\begin{align*}
T_{1,\delta} (x,\theta) &:= \dfrac{1}{(2\, \pi)^d} \, \sum_{k \in \Z} \chi_1 (\delta \, k) \,
\int_{\R^d \times \R^d \times \T} {\rm e}^{i \, (x-y) \cdot \xi} \,
{\rm e}^{2 \, i \, \pi \, k \, (\theta-\omega)} \, \chi_2 (\delta \, \xi) \\
& \qquad \qquad i \, \left( \xi_1+\dfrac{2\, \pi \, k \, \beta_1}{\eps} \right) \, \sigma \left(
\eps \, V(x,\theta),\eps \, W(y,\omega),\xi+\dfrac{2\, \pi \, k \, \beta}{\eps},\gamma \right)
\, u(y,\omega) \, {\rm d}\xi \, {\rm d}y \, {\rm d}\omega \, ,\\
T_{2,\delta} (x,\theta) &:= \dfrac{1}{(2\, \pi)^d} \, \sum_{k \in \Z} \chi_1 (\delta \, k) \,
\int_{\R^d \times \R^d \times \T} {\rm e}^{i \, (x-y) \cdot \xi} \,
{\rm e}^{2 \, i \, \pi \, k \, (\theta-\omega)} \, \chi_2 (\delta \, \xi) \\
& \qquad \qquad \left[ \partial_v \, \sigma \left(
\eps \, V(x,\theta),\eps \, W(y,\omega),\xi+\dfrac{2\, \pi \, k \, \beta}{\eps},\gamma \right) \cdot
\eps \, \partial_{x_1} \, V (x,\theta) \right] \, u(y,\omega) \, {\rm d}\xi \, {\rm d}y \, {\rm d}\omega \, ,\\
T_{3,\delta} (x,\theta) &:= \dfrac{\beta_1}{(2\, \pi)^d} \, \sum_{k \in \Z} \chi_1 (\delta \, k) \,
\int_{\R^d \times \R^d \times \T} {\rm e}^{i \, (x-y) \cdot \xi} \,
{\rm e}^{2 \, i \, \pi \, k \, (\theta-\omega)} \, \chi_2 (\delta \, \xi) \\
& \qquad \qquad \left[ \partial_v \, \sigma \left(
\eps \, V(x,\theta),\eps \, W(y,\omega),\xi+\dfrac{2\, \pi \, k \, \beta}{\eps},\gamma \right) \cdot
\partial_\theta \, V (x,\theta) \right] \, u(y,\omega) \, {\rm d}\xi \, {\rm d}y \, {\rm d}\omega \, .
\end{align*}

The singular amplitude appearing in the first term $T_{1,\delta}$ belongs to $A^0_n$, so we can apply
the same argument as in Proposition \ref{prop6} to estimate this term in $L^2 (\R^d \times \T)$. In the
same way, the singular amplitude appearing in the second term $T_{2,\delta}$ belongs to $A^{-1}_{n-1}$,
so we can still apply the argument of Proposition \ref{prop6}. The amplitude appearing in the third term
$T_{3,\delta}$ does not fall into a representation of the form \eqref{singularamplitude} because there is
an $\eps$ missing in front of $\partial_\theta V$. However, we can still apply Theorem \ref{thm2} because
this amplitude has sufficiently many derivatives in $L^\infty$ and the $L^\infty$ bounds are independent
of $\eps,\gamma$ (same argument as in Remark \ref{rem3}). The result of Proposition \ref{prop7} follows.
In particular, we have justified that the derivative
\begin{equation*}
\left( \partial_{x_1} +\dfrac{\beta_1}{\eps} \, \partial_\theta \right) \, \opteg (\widetilde{a}) \, u
\end{equation*}
is computed by differentiating formally under the integral sign (meaning that the singular amplitudes that
appear after differentiation yield well-defined oscillatory integral operators).
\end{proof}

Extending the definition of $\opteg (\widetilde{a})$ to the case $m>0$ does not seem so clear at first sight.
The trick of Proposition \ref{prop4} does not apply anymore, and we need another argument that we detail now.
Due to the application that we have in mind (see the companion article \cite{cgw3}), we restrict to the case
of amplitudes of degree $1$, meaning that the growth at infinity is $O(|\xi|+|k|)$. We do not claim that our
criterion in Lemma \ref{lem5} below is sharp. As a matter of fact, there is some hope that refined methods
may yield a similar result with less regularity on the amplitude, but this is not our main concern here. We
simply note that using sufficiently many derivatives to integrate by parts enables us to justify the convergence
of the truncation process without the compact support assumption of \cite{williams3}.

\begin{lemma}
\label{lem5}
Let $\widetilde{a} \in A_n^1$, $n \ge 3\, (d+1)$. Let $\chi_1 \in {\mathcal C}^\infty_0 (\R)$ and $\chi_2 \in
{\mathcal C}^\infty_0 (\R^d)$ satisfy $\chi_1 (0) =\chi_2 (0) =1$. Then for all $u \in {\mathcal S} (\R^d
\times \T)$, the sequence of functions $(T_\delta)_{\delta>0}$ defined by \eqref{oscintegral} with the
amplitude $\widetilde{a}_{\eps,\gamma}$ converges in ${\mathcal S}' (\R^d \times \T)$, and the limit is
independent of the truncation functions $\chi_1,\chi_2$.
\end{lemma}

As in the case $\widetilde{a} \in A_n^0$, we let $\opteg (\widetilde{a})$ denote the oscillatory integral
operator associated with $\widetilde{a} \in A_n^1$. At this stage, this operator maps ${\mathcal S}$ into
${\mathcal S}'$.

\begin{proof}[Proof of Lemma \ref{lem5}]
As in the proof of Theorem \ref{thm2}, our goal is to show that for all $u,v \in {\mathcal S}(\R^d \times \T)$,
the integral $I_\delta$ defined by
\begin{equation*}
I_\delta := \int_{\R^d \times \T} T_\delta(x,\theta) \, v(x,\theta) \, {\rm d}x \, {\rm d}\theta \, ,
\end{equation*}
with $T_\delta$ defined by \eqref{oscintegral} (just replace the general amplitude $\sigma$ in \eqref{oscintegral}
by $\widetilde{a}_{\eps,\gamma}$), has a limit as $\delta$ tends to $0$, and that the limit is independent
of the truncation functions $\chi_1,\chi_2$. Applying Fubini's Theorem, we have (let us ignore from now on
the powers of $2\, \pi$ that do not play any role):
\begin{equation*}
I_\delta = \sum_{k \in \Z} \chi_1(\delta \, k) \, \int_{\R^d \times \T \times \R^d}
{\rm e}^{i \, x \cdot \xi} \, {\rm e}^{2 \, i \, \pi \, k \, \theta} \, \chi_2(\delta \, \xi) \, v(x,\theta)
\, U(x,\theta,\xi,k) \, {\rm d}x \, {\rm d}\theta \, {\rm d}\xi \, ,
\end{equation*}
with
\begin{equation*}
U(x,\theta,\xi,k) := \int_{\R^d \times \T}
{\rm e}^{-i \, y \cdot \xi} \, {\rm e}^{-2 \, i \, \pi \, k \, \omega} \,
\widetilde{a}_{\eps,\gamma} (x,\theta,y,\omega,\xi,k) \, u(y,\omega) \, {\rm d}y \, {\rm d}\omega \, .
\end{equation*}
We claim that it is sufficient to prove an estimate of the form
\begin{equation*}
|U(x,\theta,\xi,k)| \le C(\eps,\gamma,\widetilde{a},u) \, \dfrac{1}{1+k^2} \,
\prod_{j=1}^d \dfrac{1}{1+\xi_j^2} \, ,
\end{equation*}
and the convergence of $I_\delta$ will follow from the dominated convergence Theorem (the constants may depend
in a very bad way on $\eps$ but this is no concern for us since we are only interested in the convergence of the
integral for every fixed value of $\eps$). The $L^\infty$ bound for $U$ is obtained by multiplying by the factor
\begin{equation*}
(1-2\, i \, \pi \, k)^3 \, \prod_{j=1}^d (1-i \, \xi_j)^3 \, ,
\end{equation*}
and by integrating by parts. Observing that $\widetilde{a} \in A_n^1$ with $n \ge 3\, (d+1)$, we claim that
the amplitude $\widetilde{a}_{\eps,\gamma}$ satisfies the following bounds for each fixed value of the parameters
$\eps,\gamma$:
\begin{equation*}
\big| \partial_y^\beta \, \partial_\omega^\ell \, \widetilde{a}_{\eps,\gamma} (x,\theta,y,\omega,\xi,k) \big|
\le C \, \big( 1+|\xi|^2 +k^2 \big)^{1/2} \, ,\quad |\beta| +\ell \le 3 \, (d+1) \, ,
\end{equation*}
and we thus get
\begin{equation*}
\left| (1-2\, i \, \pi \, k)^3 \, \prod_{j=1}^d (1-i \, \xi_j)^3 \, U(x,\theta,\xi,k) \right| \le
C \, \big( 1+|\xi|^2 +k^2 \big)^{1/2} \, ,
\end{equation*}
which gives the result.
\end{proof}

In the following paragraph, we shall see how the oscillatory integral operator defined in Lemma \ref{lem5}
for amplitudes in $A_n^1$ acts on singular Sobolev spaces.

\subsection{Comparison between singular oscillatory integrals operators and singular pseudodifferential operators}

Theorem \ref{thm3} below extends the result of \cite[Proposition 2.3]{williams3} to our framework in the case
of bounded symbols, and is the main ingredient in Section \ref{sect5} to prove the symbolic calculus results.

\begin{theorem}
\label{thm3}
Let $\widetilde{a} \in A_n^0$, $n \ge 2\, (d+1)$, be given by \eqref{singularamplitude}, and let $a \in S_n^0$
be defined by
\begin{equation*}
\forall \, (x,\theta,\xi,k) \in \R^d \times \T \times \R^d \times \Z \, ,\quad
a_{\eps,\gamma} (x,\theta,\xi,k) := \sigma \left( \eps \, V(x,\theta),\eps \, W(x,\theta),
\xi+\dfrac{2\, \pi \, k \, \beta}{\eps},\gamma \right) \, .
\end{equation*}
Then there exists a constant $C \ge 0$ such that for all $\eps \in \, ]0,1]$ and for all $\gamma \ge 1$, there holds
\begin{equation}
\label{estimthm31}
\forall \, u \in {\mathcal S} (\R^d \times \T) \, ,\quad
\left\| \opteg (\widetilde{a}) \, u -\opeg (a) \, u \right\|_0 \le \dfrac{C}{\gamma} \, \| u \|_0 \, .
\end{equation}

If $n \ge 2\, d +3$, then for another constant $C$, there holds
\begin{equation}
\label{estimthm32}
\forall \, u \in {\mathcal S} (\R^d \times \T) \, ,\quad
\left\| \opteg (\widetilde{a}) \, u -\opeg (a) \, u \right\|_{H^{1,\eps},\gamma} \le C \, \| u \|_0 \, ,
\end{equation}
uniformly in $\eps$ and $\gamma$.
\end{theorem}

\begin{proof}[Proof of Theorem \ref{thm3}]
The proof relies mainly on Proposition \ref{prop2}, which gives the expression of the difference $\opteg (\widetilde{a})
\, u -\opeg (a) \, u$. As a matter of fact, Proposition \ref{prop2} holds for very smooth amplitudes but using
the standard regularization procedure, the result of Proposition \ref{prop2} can be extended to amplitudes for
which the remainder $r$ defined in Proposition \ref{prop2} satisfies the assumptions of Theorem \ref{thm2}. In
what follows, we are going to verify that under the assumptions of Theorem \ref{thm3}, the remainder $r$ can be
estimated in the norm $\Ng \cdot \Nd_{\rm Amp}$ and we shall feel free to apply Proposition \ref{prop2} in this
finite regularity framework.

Let us recall that the remainder $r$ can be split as $r =r_1 +r_2$ with $r_1$ also defined in Proposition
\ref{prop2} and $r_2$ is a finite difference in $k$ (the amplitude $r_2$ does not depend on $y$). Here we
consider the amplitude $\widetilde{a}_{\eps,\gamma}$. We are first going to estimate the amplitude $r_1$,
and then $r_2$. Eventually, we shall prove the regularization estimate \eqref{estimthm32}.

$\bullet$ The amplitude $r_1$ reads
\begin{equation}
\label{expressionr1}
r_1 = \dfrac{1}{i} \, \sum_{j=1}^d \int_0^1 {\rm d}_w \sigma_j \left( \eps \, V(x,\theta),\eps \,
W((1-t) \, x +t\, y,\omega),\xi+\dfrac{2\, \pi \, k \, \beta}{\eps},\gamma \right) \cdot \eps \,
\partial_{y_j} W((1-t) \, x +t\, y,\omega) \, {\rm d}t \, ,
\end{equation}
with $\sigma_j := \partial_{\xi_j} \, \sigma \in {\bf S}^{-1}$. To prove that $\optilde (r_1)$ is bounded on
$L^2$, we wish to apply Theorem \ref{thm2} and we thus try to control $\Ng r_1 \Nd_{\rm Amp}$. For instance,
the $L^\infty$ norm of $r_1$ is estimated by using the decay of $\sigma_j$ with respect to the frequency
variables and we obtain
\begin{equation*}
|r_1(x,\theta,y,\omega,\xi,k)| \le \dfrac{C \, \eps}{\gamma} \, .
\end{equation*}
When estimating derivatives, the worst case occurs when the derivative with respect to $\omega$, the $d$
derivatives with respect to $x$ and the $d$ derivatives with respect to $y$ all act on the term $\partial_{y_j}
W((1-t) \, x +t\, y,\omega)$. This requires having a bound for the $2\, d+2$ first derivatives of $W$ in
$L^\infty$. Derivatives with respect to $\xi$ are harmless since they only add decay with respect to the
frequency variables. Under the assumption of Theorem \ref{thm3}, we thus get a bound of the form
\begin{equation*}
\Ng r_1 \Nd_{\rm Amp} \le \dfrac{C \, \eps}{\gamma} \, ,
\end{equation*}
which is even better than what we aimed at in \eqref{estimthm31}.

$\bullet$ The estimate of the term $r_2$ is more delicate and requires some attention. We first use the
trick that appears repeatedly in \cite{williams3}, namely we write
\begin{equation*}
\widetilde{a}_{\eps,\gamma} =\sigma
\left( \eps \, V(x,\theta),0,\xi+\dfrac{2\, \pi \, k \, \beta}{\eps},\gamma \right)
+\sigma_\sharp \left( \eps \, V(x,\theta),\eps \, W(y,\omega),\xi+\dfrac{2\, \pi \, k \, \beta}{\eps},\gamma
\right) \cdot \eps \, W(y,\omega) \, ,
\end{equation*}
where $\sigma_\sharp$ still belongs to ${\bf S}^0$. The first term on the right-hand side does not contribute
to the difference $\opteg (\widetilde{a}) \, u -\opeg (a) \, u$, see Proposition \ref{prop1}. We can therefore
focus on the second term for which we have an extra $\eps$. To avoid introducing some new notation, we still
use $\widetilde{a}_{\eps,\gamma}$ to denote the second term on the right-hand side. Then we have $r_2
=R(\cdot,k+1) -R(\cdot,k)$ with
\begin{equation*}
R (x,\theta,\omega,\xi,k) := \begin{cases}
\dfrac{\widetilde{a}_{\eps,\gamma} (x,\theta,x,\omega,\xi,k) -\widetilde{a}_{\eps,\gamma} (x,\theta,x,\theta,\xi,k)}
{1-{\rm e}^{-2\, i \, \pi \, (\omega-\theta)}} \, ,& \text{\rm if $\omega \neq \theta$,} \\
\dfrac{1}{2 \, i \, \pi} \, \partial_\omega \widetilde{a}_{\eps,\gamma} (x,\theta,x,\theta,\xi,k)
\, ,& \text{\rm if $\omega = \theta$.}
\end{cases}
\end{equation*}
Considering $k$ as a real variable (and not only an element of $\Z$), there holds
\begin{align*}
|r_2 (x,\theta,\omega,\xi,k)| &\le \sup_{\kappa \in [0,1]} \big| \partial_k R (x,\theta,\omega,\xi,k+\kappa) \big| \\
&\le C \, \sup_{\theta,\omega,\kappa \in [0,1]} \big| \partial_\omega \, \partial_k \,
\widetilde{a}_{\eps,\gamma} (x,\theta,x,\omega,\xi,k+\kappa) \big| \, .
\end{align*}
The $k$-derivative of $\widetilde{a}_{\eps,\gamma}$ introduces a $1/\eps$ factor times a frequency derivative
of the symbol $\sigma_\sharp$. The $1/\eps$ factor is compensated by the $\eps$ factor in the term $\eps \, W$
and the frequency derivative of the symbol yields a decay of the form
\begin{equation*}
\left( \gamma^2 +\left| \xi +\dfrac{2\, (k+\kappa) \, \pi \, \beta}{\eps} \right|^2 \right)^{-1/2} \, .
\end{equation*}
We thus obtain a bound
\begin{equation*}
|r_2 (x,\theta,\omega,\xi,k)| \le \dfrac{C}{\gamma} \, .
\end{equation*}
The estimate of the derivatives of $r_2$ follow the same strategy. Here there is no $y$-derivative to control,
and the worst case occurs when we take $d$ derivatives in $x$, one derivative in $\theta$, and one derivative
in $\omega$. This requires having $d+3$ derivatives of the functions $V,W$ in $L^\infty$. Since $d+3 \le
2 \, (d+1)$, we thus derive a bound of the form
\begin{equation*}
\Ng r_2 \Nd_{\rm Amp} \le \dfrac{C}{\gamma} \, .
\end{equation*}
Combining with our estimate of $r_1$ and applying Theorem \ref{thm2}, we already get \eqref{estimthm31}.

Before going on and proving \eqref{estimthm32}, we make an important remark. In our estimate of $r_2$, we have
taken into account the finite difference with respect to $k$ in order to make a frequency derivative appear, to
the price of a $1/\eps$ but gaining a $1/\gamma$. We could have also estimated each term of $r_2$, meaning the
terms $R(\cdot,k+1)$ and $R(\cdot,k)$, separately. If we had adopted such strategy, we would not have gained a
$1/\gamma$ but there would have been no trouble with the $1/\eps$ term. More precisely, the amplitude $r_2$
satisfies a bound of the form
\begin{equation*}
\Ng r_2 \Nd_{\rm Amp} \le C \, \eps \, .
\end{equation*}

$\bullet$ Our goal is now to prove \eqref{estimthm32}. Following Proposition \ref{prop7}, the derivative
\begin{equation*}
\left( \partial_{x_1} +\dfrac{\beta_1}{\eps} \, \partial_\theta \right) \optilde (r_1) \, u
\end{equation*}
is computed by differentiating under the integral sign provided that the amplitude has sufficiently many
derivatives in $L^\infty$, and similarly for $r_2$. We show how to estimate such derivatives under the
assumption $n \ge 2\, d+3$. Let us start with the terms involving $r_1$, which are actually easier. There
holds
\begin{equation*}
\left( \partial_{x_1} +\dfrac{\beta_1}{\eps} \, \partial_\theta \right) \optilde (r_1) \, u
=\optilde \left( \Big( i\, \xi_1 +\dfrac{2\, i \, k \, \pi \, \beta_1}{\eps} \Big) \, r_1 \right) \, u
+\optilde (\partial_{x_1} r_1) \, u +\optilde \left( \dfrac{\beta_1}{\eps} \, \partial_\theta r_1 \right) \, u \, .
\end{equation*}
We recall that the amplitude $r_1$ is given by \eqref{expressionr1}. To control $\partial_{x_1} r_1$ in
the norm $\Ng \cdot \Nd_{\rm Amp}$, one just needs an extra space derivative than in the previous step. The
same argument holds for $\partial_\theta r_1$. Consequently, under the assumption $n \ge 2\, d+3$, we get
\begin{equation*}
\left\| \optilde (\partial_{x_1} r_1) \, u
+\optilde \left( \dfrac{\beta_1}{\eps} \, \partial_\theta r_1 \right) \, u \right\|_0 \le
\dfrac{C}{\gamma} \, \| u \|_0 \, .
\end{equation*}
In order to estimate the amplitude
\begin{equation*}
\left( i\, \xi_1 +\dfrac{2\, i \, k \, \pi \, \beta_1}{\eps} \right) \, r_1 \, ,
\end{equation*}
we use the decomposition \eqref{expressionr1}, where we recall that the $\sigma_j$'s belong to ${\bf S}^{-1}$.
Compared to the previous step, this amounts to working with the symbols $i\, \xi_1 \, \sigma_j$, which belong
to ${\bf S}^0$, and we thus get uniform $L^\infty$ bounds in $O(\eps)$. Eventually, we have obtained the bound
\begin{equation*}
\left\| \left( \partial_{x_1} +\dfrac{\beta_1}{\eps} \, \partial_\theta \right) \optilde (r_1) \, u \right\|_0
\le C \, \left( \dfrac{1}{\gamma} +\eps \right) \, \| u \|_0 \, .
\end{equation*}

The remaining task is to control the analogous expression with the amplitude $r_2$ instead of $r_1$. To control
the terms that involve $\partial_{x_1} r_2$ or $(\beta_1/\eps) \, \partial_\theta r_2$, we use the above remark.
More precisely, we estimate each term with $R(\cdot,k+1)$ and $R(\cdot,k)$ separately, keeping the $\eps$ factor
to cancel the singular term $\beta_1/\eps$. This requires only one more derivative on the functions $V,W$ since
we take one more $x_1$ or $\theta$ derivative of the amplitude. The most tricky term corresponds to
\begin{equation*}
\left( i\, \xi_1 +\dfrac{2\, i \, k \, \pi \, \beta_1}{\eps} \right) \, r_2 \, .
\end{equation*}
For this final term, we use the decomposition
\begin{multline*}
\left( i\, \xi_1 +\dfrac{2\, i \, k \, \pi \, \beta_1}{\eps} \right) \, r_2 =
\left( i\, \xi_1 +\dfrac{2\, i \, (k+1) \, \pi \, \beta_1}{\eps} \right) \, R(\cdot,k+1)
-\left( i\, \xi_1 +\dfrac{2\, i \, k \, \pi \, \beta_1}{\eps} \right) \, R(\cdot,k) \\
-\dfrac{2\, i \, \pi \, \beta_1}{\eps} \, R(\cdot,k+1) \, .
\end{multline*}
The last term $R(\cdot,k+1)/\eps$ has already been estimated at the previous step and satisfies an $O(1)$
bound in the norm $\Ng \cdot \Nd_{\rm Amp}$. What remains is a finite difference in $k$ which corresponds
to the symbol $i\, \xi_1 \, \sigma$ instead of $\sigma$ (and then making the substitution with the singular
frequency $\xi +2\, k \, \pi \, \beta/\eps$). We apply the same strategy as in the previous step to make a
frequency derivative appear, to the price of a $1/\eps$. Since the $\xi$-derivatives of $i\, \xi_1 \, \sigma$
belong to ${\bf S}^0$ and thus satisfy uniform $L^\infty$ bounds, we end up with the estimate
\begin{equation*}
\left\| \left( \partial_{x_1} +\dfrac{\beta_1}{\eps} \, \partial_\theta \right) \optilde (r_2) \, u \right\|_0
\le C \, \| u \|_0 \, ,
\end{equation*}
which completes the proof of \eqref{estimthm32}.
\end{proof}

The following result extends Theorem \ref{thm3} to the case of amplitudes with degree $1$. In particular,
it will clarify the action of singular oscillatory integral operators on Sobolev spaces.

\begin{theorem}
\label{thm4}
Let $\widetilde{a} \in A_n^1$, $n \ge 3\, d +4$, be given by \eqref{singularamplitude}, and let $a \in S_n^1$
be defined by
\begin{equation*}
\forall \, (x,\theta,\xi,k) \in \R^d \times \T \times \R^d \times \Z \, ,\quad
a_{\eps,\gamma} (x,\theta,\xi,k) := \sigma \left( \eps \, V(x,\theta),\eps \, W(x,\theta),
\xi+\dfrac{2\, \pi \, k \, \beta}{\eps},\gamma \right) \, .
\end{equation*}
Then the operator $\opteg (\widetilde{a}) -\opeg (a)$ is bounded on $L^2$, namely there exists a constant
$C \ge 0$ such that for all $\eps \in \, ]0,1]$ and for all $\gamma \ge 1$, there holds
\begin{equation}
\label{estimthm4}
\forall \, u \in {\mathcal S} (\R^d \times \T) \, ,\quad
\left\| \opteg (\widetilde{a}) \, u -\opeg (a) \, u \right\|_0 \le C \, \| u \|_0 \, .
\end{equation}
In particular, $\opteg (\widetilde{a})$ maps $H^{1,\eps}$ into $L^2$ and there exists a constant $C \ge 0$
such that for all $\eps \in \, ]0,1]$ and for all $\gamma \ge 1$, there holds
\begin{equation*}
\forall \, u \in {\mathcal S} (\R^d \times \T) \, ,\quad
\left\| \opteg (\widetilde{a}) \, u \right\|_0 \le C \, \| u \|_{H^{1,\eps},\gamma} \, .
\end{equation*}
\end{theorem}

\begin{proof}[Proof of Theorem \ref{thm4}]
Let $u \in {\mathcal S}(\R^d \times \T)$. Then we know that $\opteg (\widetilde{a}) \, u$ is the limit in
${\mathcal S}'(\R^d \times \T)$, as $\delta$ tends to $0$, of the sequence $(T_\delta)$, with (ignore from
now on the powers of $2\, \pi$):
\begin{equation*}
T_\delta \, (x,\theta) := \sum_{k \in \Z} \chi_1 (\delta \, k) \, \int_{\R^d \times \R^d \times \T}
{\rm e}^{i \, (x-y) \cdot \xi} \, {\rm e}^{2 \, i \, \pi \, k \, (\theta-\omega)} \,
\chi_2 (\delta \, \xi) \, \widetilde{a}_{\eps,\gamma} (x,\theta,y,\omega,\xi,k) \, u(y,\omega) \,
{\rm d}\xi \, {\rm d}y \, {\rm d}\omega \, .
\end{equation*}
Using the result of Proposition \ref{prop2} (with a finite regularity, which can be justified by the standard
regularization procedure), we decompose as usual
\begin{equation*}
T_\delta =T_{1,\delta} +\optilde (r_{1,\delta}) \, u +\optilde (r_{2,\delta}) \, u \, ,
\end{equation*}
with
\begin{equation*}
T_\delta \, (x,\theta) := \sum_{k \in \Z} \chi_1 (\delta \, k) \, \int_{\R^d \times \R^d \times \T}
{\rm e}^{i \, (x-y) \cdot \xi} \, {\rm e}^{2 \, i \, \pi \, k \, (\theta-\omega)} \,
\chi_2 (\delta \, \xi) \, \widetilde{a}_{\eps,\gamma} (x,\theta,x,\theta,\xi,k) \, u(y,\omega) \,
{\rm d}\xi \, {\rm d}y \, {\rm d}\omega \, ,
\end{equation*}
and $r_{1,\delta}$, $r_{2,\delta}$ are as in Proposition \ref{prop2} but are obtained by considering the
truncated amplitude $\chi_1 (\delta \, k) \, \chi_2 (\delta \, \xi) \, \widetilde{a}_{\eps,\gamma}$.

It is easy to show that the sequence $T_{1,\delta}$ converges in ${\mathcal S}'$ (and even in a stronger sense)
towards $\opeg (a) \, u$, because one can first integrate in $(y,\omega)$ and use the decay of the Fourier
transform of $u$. We are now going to compute the limit as $\delta$ tends to $0$ of
\begin{equation*}
\optilde (r_{1,\delta}) \, u +\optilde (r_{2,\delta}) \, u \, .
\end{equation*}

Using the general formula of Proposition \ref{prop2}, we have
\begin{align*}
r_{1,\delta} &=\chi_1 (\delta \, k) \, \chi_2 (\delta \, \xi) \, r_1 +\dfrac{\delta}{i} \, \sum_{j=1}^d
\chi_1 (\delta \, k) \, \partial_{\xi_j} \chi_2 (\delta \, \xi) \, \int_0^1
\partial_{y_j} \widetilde{a}_{\eps,\gamma} (x,\theta,(1-t)\, x+t\, y,\omega,\xi,k) \, {\rm d}t \, ,\\
r_1 &:=\dfrac{1}{i} \, \sum_{j=1}^d \int_0^1 \partial_{y_j} \partial_{\xi_j} \widetilde{a}_{\eps,\gamma}
(x,\theta,(1-t)\, x+t\, y,\omega,\xi,k) \, {\rm d}t \, .
\end{align*}
Since $\partial_{\xi_j} \widetilde{a}_{\eps,\gamma}$ is a bounded amplitude, we can apply Theorem \ref{thm2}
for the convergence of the term $\optilde (\chi_1(\delta \, k) \, \chi_2 (\delta \, \xi) \, r_1) \, u$. More
precisely, we have $\partial_{y_j} \partial_{\xi_j} \widetilde{a} \in A_{n-1}^0$, $n-1 \ge 3\, d +3$, and we
therefore know that the limit of this term is $\optilde (r_1) \, u$. Moreover, the operator $\optilde (r_1)$
acts boundedly on $L^2$, uniformly in $\eps,\gamma$.

We now deal with the remaining term in $r_{1,\delta}$. It is sufficient to prove that the singular oscillatory
integral associated with the amplitude
\begin{equation*}
\chi_1 (\delta \, k) \, \partial_{\xi_j} \chi_2 (\delta \, \xi) \, \int_0^1
\partial_{y_j} \widetilde{a}_{\eps,\gamma} (x,\theta,(1-t)\, x+t\, y,\omega,\xi,k) \, {\rm d}t \, ,
\end{equation*}
has a limit in ${\mathcal S}'$ as $\delta$ tends to $0$. Since $\partial_{y_j} \widetilde{a}$ belongs to
$A_{n-1}^1$, $n-1 \ge 3\, d+3$, we can apply Lemma \ref{lem5} to this term. Together with the extra $\delta$
factor, we have shown that the limit of $\optilde (r_{1,\delta}) \, u$ in ${\mathcal S}'$ is $\optilde (r_1)
\, u$, and that this term is controlled in $L^2$ uniformly with respect to $\eps,\gamma$.

The analogous term with $r_{2,\delta}$ is dealt with in a similar way. The finite difference with respect
to $k$ plays the role of the $\xi$ derivative and we can prove uniform bounds of the amplitude in the norm
$\Ng \cdot \Nd_{\rm Amp}$ by using the same arguments as in the proof of Theorem \ref{thm3}. We feel free
to skip the details. The action of $\opteg (\widetilde{a})$ on $H^{1,\eps}$ is obtained by combining
\eqref{estimthm4} with the result of Proposition \ref{prop4} for $\opeg (a)$.
\end{proof}

\section{Singular pseudodifferential calculus II. Adjoints and products}
\label{sect5}

\subsection{Adjoints of singular pseudodifferential operators}

Our results on adjoints are very easy consequences of all the preliminary results in Section \ref{sect4}.
Let us start with the case of bounded symbols.

\begin{proposition}
\label{prop8}
Let $a \in S_n^0$, $n \ge 2\, (d+1)$, and let $a^*$ denote the conjugate transpose of the symbol $a$.
Then $\opeg (a)$ and $\opeg (a^*)$ act boundedly on $L^2$ and there exists a constant $C \ge 0$ such
that for all $\eps \in \, ]0,1]$ and for all $\gamma \ge 1$, there holds
\begin{equation*}
\forall \, u \in {\mathcal S} (\R^d \times \T) \, ,\quad
\left\| \opeg (a)^* \, u -\opeg (a^*) \, u \right\|_0 \le \dfrac{C}{\gamma} \, \| u \|_0 \, .
\end{equation*}

If $n \ge 2\, d +3$, then for another constant $C$, there holds
\begin{equation*}
\forall \, u \in {\mathcal S} (\R^d \times \T) \, ,\quad
\left\| \opeg (a)^* \, u -\opeg (a^*) \, u \right\|_{H^{1,\eps},\gamma} \le C \, \| u \|_0 \, ,
\end{equation*}
uniformly in $\eps$ and $\gamma$.
\end{proposition}

\begin{proof}[Proof of Proposition \ref{prop8}]
As in \cite[proposition 2.4]{williams3}, it is sufficient to observe that if $a_{\eps,\gamma}$ is defined by
\eqref{singularsymbol}, the adjoint operator $\opeg (a)^*$ coincides with the singular oscillatory integral
operator $\opteg (\widetilde{b})$ associated with the amplitude
\begin{equation*}
\widetilde{b}_{\eps,\gamma} (x,\theta,y,\omega,\xi,k) := a_{\eps,\gamma} (y,\omega,\xi,k)^*
= \sigma \left( \eps \, V(y,\omega),\xi+\dfrac{2\, \pi \, k \, \beta}{\eps},\gamma \right)^* \, .
\end{equation*}
Then we apply Theorem \ref{thm3} and the conclusion follows.
\end{proof}

Proposition \ref{prop8} can be extended to symbols of degree $1$ up to an additional regularity in the space
variables (this high regularity is mainly required to give a precise meaning to oscillatory integral operators).

\begin{proposition}
\label{prop9}
Let $a \in S_n^1$, $n \ge 3\, d +4$, and let $a^*$ denote the conjugate transpose of the symbol $a$.
Then $\opeg (a)$ and $\opeg (a^*)$ map $H^{1,\eps}$ into $L^2$ and there exists a family of operators
$R^{\eps,\gamma}$ that satisfies
\begin{itemize}
 \item there exists a constant $C \ge 0$ such that for all $\eps \in \, ]0,1]$ and for all
      $\gamma \ge 1$, there holds
\begin{equation*}
\forall \, u \in {\mathcal S} (\R^d \times \T) \, ,\quad
\left\| R^{\eps,\gamma} \, u \right\|_0 \le C \, \| u \|_0 \, ,
\end{equation*}

 \item the following duality property holds
\begin{equation*}
\forall \, u,v \in {\mathcal S} (\R^d \times \T) \, ,\quad
\langle \opeg (a) \, u,v \rangle_{L^2} -\langle u,\opeg (a^*) \, v \rangle_{L^2} =\langle
R^{\eps,\gamma} \, u,v \rangle_{L^2} \, .
\end{equation*}
In particular, the adjoint $\opeg (a)^*$ for the $L^2$ scalar product maps $H^{1,\eps}$ into $L^2$.
\end{itemize}
\end{proposition}

\begin{proof}[Proof of Proposition \ref{prop9}]
The proof is quite similar to that of Proposition \ref{prop8}. First of all, the regularity assumption
$n \ge 3\, (d+1)$ allows to pass to the limit in the standard truncation process and to show that the
adjoint (with respect to the $L^2$ scalar product) of the operator $\opeg (a)$ coincides with the
oscillatory integral operator $\opteg (\widetilde{b})$ associated with the amplitude
\begin{equation*}
\widetilde{b}_{\eps,\gamma} (x,\theta,y,\omega,\xi,k) := a_{\eps,\gamma} (y,\omega,\xi,k)^*
= \sigma \left( \eps \, V(y,\omega),\xi+\dfrac{2\, \pi \, k \, \beta}{\eps},\gamma \right)^* \, .
\end{equation*}
Then we apply Theorem \ref{thm4} and the conclusion follows.
\end{proof}

Let us observe that we only have proved a symbolic calculus "at the first order", meaning that we have not
proved that the adjoint operator $\opeg (a)^*$ admits an asymptotic expansion with more and more smoothing
operators. Even in the case of ${\mathcal C}^\infty$ regularity for the substituted function $V$, it is not so clear
that the second order expansion holds with a uniformly bounded remainder in the scale of spaces $H^{k,\eps}$.
This bad behavior is more or less the same as in Remark \ref{rem4} (consider for instance the case of differential
operators of order $2$).

\subsection{Products of singular pseudodifferential operators}

We still follow \cite{williams3} and begin with a special case of products.

\begin{proposition}
\label{prop10}
Let $a,b \in S_n^0$, $n \ge 2\, (d+1)$. Then there exists a constant $C \ge 0$ such that for all
$\eps \in \, ]0,1]$ and for all $\gamma \ge 1$, there holds
\begin{equation*}
\forall \, u \in {\mathcal S} (\R^d \times \T) \, ,\quad
\left\| \opeg (a) \, \opeg (b)^* \, u -\opeg (a \, b^*) \, u \right\|_0 \le \dfrac{C}{\gamma} \, \| u \|_0 \, .
\end{equation*}
If $n \ge 2\, d +3$, then for another constant $C$, there holds
\begin{equation*}
\forall \, u \in {\mathcal S} (\R^d \times \T) \, ,\quad
\left\| \opeg (a) \, \opeg (b)^* \, u -\opeg (a \, b^*) \, u \right\|_{H^{1,\eps},\gamma} \le C \, \| u \|_0 \, ,
\end{equation*}
uniformly in $\eps$ and $\gamma$.

Let $a \in S_n^1,b \in S_n^0$ or $a \in S_n^0,b \in S_n^1$, $n \ge 3\, d +4$. Then there exists a constant
$C \ge 0$ such that for all $\eps \in \, ]0,1]$ and for all $\gamma \ge 1$, there holds
\begin{equation*}
\forall \, u \in {\mathcal S} (\R^d \times \T) \, ,\quad
\left\| \opeg (a) \, \opeg (b)^* \, u -\opeg (a \, b^*) \, u \right\|_0 \le C \, \| u \|_0 \, .
\end{equation*}
\end{proposition}

\begin{proof}[Proof of Proposition \ref{prop10}]
In each of the three possible cases, the main point is to observe that the operator $\opeg (a) \, \opeg (b)^*$
coincides with the oscillatory integral operator $\opteg (\widetilde{c})$ associated with the amplitude
\begin{equation*}
\widetilde{c}_{\eps,\gamma} (x,\theta,y,\omega,\xi,k) := a_{\eps,\gamma} (x,\theta,\xi,k) \,
b_{\eps,\gamma} (y,\omega,\xi,k)^* \, .
\end{equation*}
The result is well-known for amplitudes with a sufficient decay with respect to the frequencies, and it holds
in a more general framework provided that all oscillatory integrals can be defined (which is the case under the
regularity assumptions stated in Proposition \ref{prop10}). The conclusion then follows from either Theorem
\ref{thm3} or \ref{thm4}.
\end{proof}

A main improvement with respect to \cite{williams3} is that we can now deal with all kinds of products by the
classical $**$ argument. This improvement has been made possible because we have already shown a smoothing
property for some remainders in the calculus (compare with \cite[Propositions 2.6, 2.7]{williams3}).

\begin{proposition}
\label{prop11}
Let $a,b \in S_n^0$, $n \ge 2\, (d+1)$. Then there exists a constant $C \ge 0$ such that for all
$\eps \in \, ]0,1]$ and for all $\gamma \ge 1$, there holds
\begin{equation*}
\forall \, u \in {\mathcal S} (\R^d \times \T) \, ,\quad
\left\| \opeg (a) \, \opeg (b) \, u -\opeg (a \, b) \, u \right\|_0 \le \dfrac{C}{\gamma} \, \| u \|_0 \, .
\end{equation*}
If $n \ge 2\, d +3$, then for another constant $C$, there holds
\begin{equation*}
\forall \, u \in {\mathcal S} (\R^d \times \T) \, ,\quad
\left\| \opeg (a) \, \opeg (b) \, u -\opeg (a \, b) \, u \right\|_{H^{1,\eps},\gamma} \le C \, \| u \|_0 \, ,
\end{equation*}
uniformly in $\eps$ and $\gamma$.

Let $a \in S_n^1,b \in S_n^0$ or $a \in S_n^0,b \in S_n^1$, $n \ge 3\, d +4$. Then there exists a constant
$C \ge 0$ such that for all $\eps \in \, ]0,1]$ and for all $\gamma \ge 1$, there holds
\begin{equation*}
\forall \, u \in {\mathcal S} (\R^d \times \T) \, ,\quad
\left\| \opeg (a) \, \opeg (b) \, u -\opeg (a \, b) \, u \right\|_0 \le C \, \| u \|_0 \, .
\end{equation*}
\end{proposition}

\begin{proof}[Proof of Proposition \ref{prop11}]
Let us deal for instance with the case $a,b \in S_n^0$, $n \ge 2\, d+ 3$. Then we have
\begin{equation*}
\opeg (a) \, \opeg (b) =\opeg (a) \, \opeg (b^{**}) =\opeg (a) \, \Big( \opeg (b^*)^* +R_{-1}^{\eps,\gamma} \Big)
\, ,
\end{equation*}
where we have applied Theorem \ref{thm3} to the symbol $b^*$ and denoted $R_{-1}^{\eps,\gamma}$ the smoothing
remainder (mapping $L^2$ into $H^{1,\eps}$). Thanks to Lemma \ref{lem4}, we know that $\opeg (a)$ acts continuously
on $H^{1,\eps}$, uniformly with respect to $\eps,\gamma$, so the product $\opeg (a) \, R_{-1}^{\eps,\gamma}$ can be
rewritten as a remainder of the form $R_{-1}^{\eps,\gamma}$. The product $\opeg (a) \, \opeg (b^*)^*$ is dealt with
by applying Proposition \ref{prop10}. We end up with
\begin{equation*}
\opeg (a) \, \opeg (b) =\opeg (a \, b) +R_{-1}^{\eps,\gamma} \, .
\end{equation*}

The only other interesting case is $a \in S_n^1,b \in S_n^0$, $n \ge 3\, d+ 4$. Then we write again
\begin{equation*}
\opeg (a) \, \opeg (b) =\opeg (a) \, \Big( \opeg (b^*)^* +R_{-1}^{\eps,\gamma} \Big) \, ,
\end{equation*}
and we observe that the product $\opeg (a) \, R_{-1}^{\eps,\gamma}$ acts boundedly on $L^2$, uniformly with respect
to $\eps,\gamma$ (use Theorem \ref{thm4}). The product $\opeg (a) \, \opeg (b^*)^*$ is dealt with by applying again
Proposition \ref{prop10}. We leave all remaining cases to the interested reader.
\end{proof}

\noindent A surprising fact is that the $**$ argument also applies for products of operators with degree $-1$ and $1$.
We feel free to skip the proof that is entirely similar to that of Proposition \ref{prop11}.

\begin{proposition}
\label{prop12}
Let $a \in S_n^{-1},b \in S_n^1$, $n \ge 3\, d +4$. Then $\opeg (a) \, \opeg (b)$ defines a bounded operator on
$H^{1,\eps}$ and there exists a constant $C \ge 0$ such that for all $\eps \in \, ]0,1]$ and for all $\gamma \ge 1$,
there holds
\begin{equation*}
\forall \, u \in {\mathcal S} (\R^d \times \T) \, ,\quad
\left\| \opeg (a) \, \opeg (b) \, u -\opeg (a \, b) \, u \right\|_{H^{1,\eps},\gamma} \le C \, \| u \|_0 \, .
\end{equation*}
\end{proposition}

The analogue of Proposition \ref{prop12} seems unfortunately untrue when the product is taken the other way
round, meaning when the operator of order $+1$ acts on the left. This can be seen for instance by choosing for
the left operator the singular derivative $\partial_{x_1} +(\beta_1/\eps) \, \partial_\theta$. We are then reduced
to showing a bound in $H^{1,\eps}$ for the terms $T_2,T_3$ appearing in the proof of Proposition \ref{prop5}.
Such a bound is available for $T_2$ but not for the last term $T_3$. This fact gives rise to a special treatment
of $+1,-1$ products in the companion article \cite{cgw3}.

\subsection{G{\aa}rding's inequality}

The exact same arguments as in \cite[page 155]{williams3} apply to prove the G{\aa}rding inequality without
any compact support assumption on the symbols. We just need slightly more regularity on the symbols in
order to apply Propositions \ref{prop8} and \ref{prop11} above.

\begin{theorem}
\label{thm5}
Let $\sigma \in {\bf S}^0$ satisfy $\text{\rm Re} \, \sigma (v,\xi,\gamma) \ge C_K>0$ for all $v$ in a compact
subset $K$ of ${\mathcal O}$. Let now $a \in S_0^n$, $n \ge 2\, d+2$ be given by \eqref{singularsymbol}, where
$V$ is valued in a convex compact subset $K$. Then for all $\delta >0$, there exists $\gamma_0$ which depends
uniformly on $V$, the constant $C_K$ and $\delta$, such that for all $\gamma \ge \gamma_0$ and all $u \in
{\mathcal S}(\R^d \times \T)$, there holds
\begin{equation*}
\text{\rm Re } \langle \opeg (a) \, u ;u \rangle_{L^2} \ge (C-\delta) \, \| u \|_0^2 \, .
\end{equation*}
\end{theorem}

\subsection{Extended singular pseudodifferential calculus}

Following \cite[page 153]{williams3}, we can extend all the above results on boundedness/adjoints/products
to the larger class $e{\bf S}^m$ of functions $\sigma : {\mathcal O} \times \R^d \times \R^{2d} \times
[1,+\infty[ \rightarrow \C^{N \times N}$ such that
\begin{itemize}
 \item[{\rm (i)}] for all $\gamma \ge 1$, $\sigma (\cdot,\cdot,\cdot,\gamma)$ is ${\mathcal C}^\infty$ on
                   ${\mathcal O} \times \R^d \times \R^{2d}$,
 \item[{\rm (ii)}] for all compact subset $K$ of ${\mathcal O}$, for all $\alpha \in \N^q$ and for all $\nu \in
                   \N^{3d}$, there exists a constant $C_{\alpha,\nu,K}$ satisfying
\begin{equation*}
\sup_{v \in K} \, \sup_{(\xi,\zeta) \in \R^d \times \R^{2d}, |\xi| \le |\zeta|} \, \sup_{\gamma \ge 1} \, \,
(\gamma^2+|\xi|^2)^{-(m-|\nu|)/2} \, \big| \partial_v^\alpha \, \partial_{\xi,\zeta}^\nu \, \sigma \, (v,\xi,\zeta,\gamma)
\big| \le C_{\alpha,\nu,K} \, .
\end{equation*}
\end{itemize}
For such symbols, we use the substitution $v \rightarrow \eps \, V(x,\theta)$, $\xi \rightarrow \xi
+2\, \pi \, k \, \beta /\eps$, $\zeta \rightarrow (\xi,2\, \pi \, k \, \beta /\eps)$, which gives rise to extended
singular pseudodifferential operators of the form
\begin{equation*}
e\opeg (a) \, u \, (x,\theta) := \dfrac{1}{(2\, \pi)^d} \, \sum_{k \in \Z} \int_{\R^d}
{\rm e}^{i\, x \cdot \xi} \, {\rm e}^{2\, i\, \pi \, k \, \theta} \,
\sigma \left( \eps \, V(x,\theta),\xi+\dfrac{2 \, \pi \, k \, \beta}{\eps},\xi,\dfrac{2 \, \pi \, k \, \beta}{\eps},\gamma \right)
\, \widehat{c_k(u)} (\xi) \, {\rm d}\xi \, .
\end{equation*}

We can also define extended singular amplitudes and compare the extended oscillatory integral operators
with the above extended pseudodifferential operator. All results in Sections \ref{sect4} and \ref{sect5} are
proved in the same way for this extended class because we have always relied on general boundedness
result such as Theorem \ref{thm1} or Theorem \ref{thm2} and these results can handle symbols or amplitudes
in the extended class.

The main interest of defining this extended class is to be able to consider pseudodifferential cut-offs of the form
\begin{equation*}
\chi \left( \eps \, V(x,\theta), D_x, \dfrac{\beta \, D_\theta}{\eps}, \gamma \right) \, ,
\end{equation*}
where $\chi$ is supported in the region $|\zeta_1| \ll |\zeta_2|$ (here $\zeta_1$ is the placeholder for $\xi$
and $\zeta_2$ is the placeholder for $2 \, \pi \, k \, \beta/\eps$). Such cut-offs are useful to microlocalize near
the specific frequency $\beta$. We again refer to \cite{cgw3} for further applications of these techniques.
\newpage

\begin{center}
{\sc Part B: singular pseudodifferential calculus for pulses}
\end{center}
\bigskip

In this second part, we briefly explain why all the results of the first part also give continuity and symbolic
calculus results in the case where the additional space variable and associated singular frequency lie in $\R$.
There are some slight technical differences in the proofs and we pay specific attention to those terms that
have the worst behavior with respect to the singular parameter $\eps$. As far as the notation is concerned,
we feel free to use the same notation to denote new classes of symbols, amplitudes and so on, in order to
highlight the similarities between Part A and Part B. We hope that this will not create any confusion.

The variable in $\R^{d+1}$ is denoted $(x,\theta)$, $x \in \R^d$, $\theta \in \R$, and the associated frequency
is denoted $(\xi,k)$. In this new context, the singular Sobolev spaces are defined as follows. We still consider
a vector $\beta \in \R^d \setminus \{ 0\}$. Then for $s \in \R$ and $\eps \in \, ]0,1]$, the anisotropic Sobolev
space $H^{s,\eps} (\R^{d+1})$ is defined by
\begin{multline*}
H^{s,\eps}(\R^{d+1}) := \Big\{ u \in {\mathcal S}'(\R^{d+1}) \, / \, \widehat{u} \in L^2_{\rm loc}(\R^{d+1}) \\
\text{\rm and} \quad \int_{\R^{d+1}} \left( 1+\left| \xi+\dfrac{k \, \beta}{\eps} \right|^2 \right)^s
\, \big| \widehat{u}(\xi,k) \big|^2 \, {\rm d}\xi \, {\rm d}k <+\infty \Big\} \, .
\end{multline*}
Here $\widehat{u}$ denotes the Fourier transform of $u$ on $\R^{d+1}$. The space $H^{s,\eps}(\R^{d+1})$ is
equipped with the family of norms
\begin{equation*}
\forall \, \gamma \ge 1 \, ,\quad \forall \, u \in H^{s,\eps}(\R^{d+1}) \, ,\quad
\| u \|_{H^{s,\eps},\gamma}^2 := \dfrac{1}{(2\, \pi)^{d+1}} \, \int_{\R^{d+1}}
\left( \gamma^2 +\left| \xi+\dfrac{k \, \beta}{\eps} \right|^2 \right)^s
\, \big| \widehat{u}(\xi,k) \big|^2 \, {\rm d}\xi \, {\rm d}k \, .
\end{equation*}
When $m$ is an integer, the space $H^{m,\eps} (\R^{d+1})$ coincides with the space of functions $u \in L^2
(\R^{d+1})$ such that the derivatives, in the sense of distributions,
\begin{equation*}
\left( \partial_{x_1} +\dfrac{\beta_1}{\eps} \, \partial_\theta \right)^{\alpha_1} \dots
\left( \partial_{x_d} +\dfrac{\beta_d}{\eps} \, \partial_\theta \right)^{\alpha_d} \, u \, ,\quad
\alpha_1+\dots+\alpha_d \le m \, ,
\end{equation*}
belong to $L^2 (\R^{d+1})$. In the definition of the norm $\| \cdot \|_{H^{m,\eps},\gamma}$, one power of
$\gamma$ counts as much as one derivative.

\section{The main $L^2$ continuity results}
\label{sect6}

We adapt the Calder\'on-Vaillancourt Theorems to our framework where symbols will enjoy a suitable decay
property in the additional space variable $\theta$. The decay will allow us to avoid requiring a control of $k$-derivatives
to prove $L^2$-boundedness.

\begin{theorem}
\label{thm6}
Let $\sigma : \R^{d+1} \times \R^{d+1} \rightarrow \C^{N \times N}$ be a continuous function that satisfies the
property: for all $\alpha,\beta \in \{ 0,1\}^d$ and for all $j \in \{ 0,1\}$, $\langle \theta \rangle \, \partial_x^\alpha
\, \partial_\theta^j \, \partial_\xi^\beta \, \sigma$ belongs to $L^\infty (\R^{d+1} \times \R^{d+1})$, where the
derivative is understood in the sense of distributions and $\langle \cdot \rangle$ denotes the Japanese bracket.

For $u \in {\mathcal S}(\R^{d+1};\C^N)$, let us define
\begin{equation*}
\forall \, (x,\theta) \in \R^{d+1} \, ,\quad
\op (\sigma) \, u \, (x,\theta) := \dfrac{1}{(2\, \pi)^{d+1}} \, \int_{\R^{d+1}} {\rm e}^{i \, (\xi \cdot x +k \, \theta)} \,
\sigma (x,\theta,\xi,k) \, \widehat{u}(\xi,k) \, {\rm d}\xi \, {\rm d}k \, .
\end{equation*}
Then $\op (\sigma)$ extends as a continuous operator on $L^2(\R^{d+1};\C^N)$. More precisely, there
exists a numerical constant $C$, that only depends on $d$ and $N$, such that for all $u \in {\mathcal S}
(\R^{d+1};\C^N)$, there holds
\begin{equation*}
\| \op (\sigma) \, u \|_0 \le C \, \Ng \sigma \Nd \, \| u \|_0 \, ,\, \quad
\text{\rm with } \Ng \sigma \Nd := \sup_{\alpha,\beta \in \{ 0,1\}^d} \, \sup_{j \in \{ 0,1\}} \,
\left\| \langle \theta \rangle \, \partial_x^\alpha \, \partial_\theta^j \, \partial_\xi^\beta \, \sigma
\right\|_{L^\infty (\R^{d+1} \times \R^{d+1})} \, .
\end{equation*}
\end{theorem}

There is very little to change compared to the proof of Theorem \ref{thm1}. More precisely, we can reproduce
the proof of Theorem \ref{thm1} by making the following modifications: first of all, we consider $d+1$ space
variables instead of $d$, and forget about the additional periodic variable. Then we perform all integration by
parts except the one with respect to the last frequency variables (which would be aimed at gaining a decaying
factor in $\theta$). In the end, we add a $\langle \theta \rangle$ weight to the derivatives of the symbol $\sigma$
and use the square integrable $\langle \theta \rangle^{-1}$ weight on $\partial_x^{\alpha'} U (x,\xi,k)$. The final
estimates (Cauchy-Schwarz and so on) are unchanged.

The main point to keep in mind is that, in contrast with the standard Calder\'on-Vaillancourt Theorem which requires
controlling one $k$-derivative in $L^\infty$, we can use instead a decay property in $\theta$ for the symbol and still
prove $L^2$ boundedness of the associated pseudodifferential operator. As we have already seen in the case of
wavetrains, avoiding a control of $k$-derivatives is crucial if we wish to prove some bounds that are uniform with
respect to the singular parameter $\eps$.

We now turn to the case of oscillatory integral operators associated with amplitudes defined on $\R^{d+1} \times
\R^{d+1} \times \R^{d+1}$. We feel free to skip the proof of the following Theorem which is again a slight adaptation
(with similar modifications as described just above) of Theorem \ref{thm2}.

\begin{theorem}
\label{thm7}
Let $\sigma : \R^{d+1} \times \R^{d+1} \times \R^{d+1} \rightarrow \C^{N \times N}$ be a continuous function that
satisfies the property: for all $\alpha,\beta \in \{ 0,1\}^{d+1}$, for all $j \in \{ 0,1\}$ and for all $\nu \in \{ 0,1,2\}^d$,
there holds $\langle \omega \rangle \, \partial_{x,\theta}^\alpha \, \partial_{y,\omega}^\beta \, \partial_\xi^\nu \,
\partial_k ^j \sigma$ belongs to $L^\infty (\R^{d+1} \times \R^{d+1} \times \R^{d+1})$. Let $\chi \in {\mathcal C}^\infty_0
(\R^{d+1})$ satisfy $\chi (0) = 1$.

Then for all $u \in {\mathcal S}(\R^{d+1})$, the sequence of functions $(T_\delta)_{\delta>0}$ defined on $\R^{d+1}$ by
\begin{equation*}
T_\delta \, (x,\theta) := \dfrac{1}{(2\, \pi)^{d+1}} \, \int_{\R^{d+1} \times \R^{d+1}}
{\rm e}^{i \, (\xi \cdot (x-y) +k \, (\theta-\omega))} \, \chi (\delta \, \xi,\delta \, k) \, \sigma (x,\theta,y,\omega,\xi,k) \,
u(y,\omega) \, {\rm d}\xi \, {\rm d}k\, {\rm d}y \, {\rm d}\omega \, ,
\end{equation*}
converges in ${\mathcal S}' (\R^{d+1})$, as $\delta$ tends to $0$, towards a distribution $\optilde (\sigma) \, u \in L^2
(\R^{d+1})$. This limit is independent of the truncation function $\chi$. Moreover, there exists a numerical constant $C$,
that only depends on $d$ and $N$, such that there holds
\begin{multline*}
\left\| \optilde (\sigma) \, u \right\|_0 \le C \, \Ng \sigma \Nd_{\rm Amp} \, \| u \|_0 \, ,\, \\
\text{\rm with } \Ng \sigma \Nd_{\rm Amp} := \sup_{\alpha,\beta \in \{ 0,1\}^{d+1}} \,
\sup_{j \in \{ 0,1\}} \, \sup_{\nu \in \{ 0,1,2\}^d} \,
\left\| \langle \omega \rangle \, \partial_{x,\theta}^\alpha \, \partial_{y,\omega}^\beta \, \partial_\xi^\nu \, \partial_k^j
\, \sigma \right\|_{L^\infty (\R^{d+1} \times \R^{d+1} \times \R^{d+1})} \, .
\end{multline*}
\end{theorem}

In Theorem \ref{thm7}, one exchanges a $k$-derivative for a decay in $\omega$. Using a decay in $\theta$ is also
possible but it does not seem to be useful for our purpose. However, it will be useful below to have a version of
Theorem \ref{thm7} that combines both decay in $\omega$ and $\theta$ and thus gets rid of all $k$-derivatives.

\begin{theorem}
\label{thm8}
Let $\sigma : \R^{d+1} \times \R^{d+1} \times \R^{d+1} \rightarrow \C^{N \times N}$ be a continuous function that
satisfies the property: for all $\alpha,\beta \in \{ 0,1\}^{d+1}$ and for all $\nu \in \{ 0,1,2\}^d$, there holds $\langle
\theta \rangle \, \langle \omega \rangle \, \partial_{x,\theta}^\alpha \, \partial_{y,\omega}^\beta \, \partial_\xi^\nu \,
\sigma$ belongs to $L^\infty (\R^{d+1} \times \R^{d+1} \times \R^{d+1})$. Let $\chi \in {\mathcal C}^\infty_0
(\R^{d+1})$ satisfy $\chi (0) = 1$.

Then for all $u \in {\mathcal S}(\R^{d+1})$, the sequence of functions $(T_\delta)_{\delta>0}$ defined on $\R^{d+1}$ by
\begin{equation*}
T_\delta \, (x,\theta) := \dfrac{1}{(2\, \pi)^{d+1}} \, \int_{\R^{d+1} \times \R^{d+1}}
{\rm e}^{i \, (\xi \cdot (x-y) +k \, (\theta-\omega))} \, \chi (\delta \, \xi,\delta \, k) \, \sigma (x,\theta,y,\omega,\xi,k) \,
u(y,\omega) \, {\rm d}\xi \, {\rm d}k\, {\rm d}y \, {\rm d}\omega \, ,
\end{equation*}
converges in ${\mathcal S}' (\R^{d+1})$, as $\delta$ tends to $0$, towards a distribution $\optilde (\sigma) \, u \in L^2
(\R^{d+1})$. This limit is independent of the truncation function $\chi$. Moreover, there exists a numerical constant $C$,
that only depends on $d$ and $N$, such that there holds
\begin{multline*}
\left\| \optilde (\sigma) \, u \right\|_0 \le C \, \Ng \sigma \Nd_{\rm Amp} \, \| u \|_0 \, ,\, \\
\text{\rm with } \Ng \sigma \Nd_{\rm Amp} := \sup_{\alpha,\beta \in \{ 0,1\}^{d+1}} \, \sup_{\nu \in \{ 0,1,2\}^d} \,
\left\| \langle \theta \rangle \, \langle \omega \rangle \, \partial_{x,\theta}^\alpha \, \partial_{y,\omega}^\beta \,
\partial_\xi^\nu \, \sigma \right\|_{L^\infty (\R^{d+1} \times \R^{d+1} \times \R^{d+1})} \, .
\end{multline*}
\end{theorem}

Of course, when the amplitude $\sigma$ in Theorems \ref{thm7} and \ref{thm8} does not depend on its third and
fourth variables, we are reduced to the case of pseudodifferential operators.

\section{Singular pseudodifferential calculus I. Definition of operators and action on Sobolev spaces}
\label{sect7}

Let us first define the singular symbols.

\begin{definition}[Singular symbols]
\label{def4}
Let $m \in \R$, and let $n \in \N$. Then we let $S^m_n$ denote the set of families of functions
$(a_{\eps,\gamma})_{\eps \in ]0,1],\gamma \ge 1}$ that are constructed as follows:
\begin{equation}
\label{singularsymbolp}
\forall \, (x,\theta,\xi,k) \in \R^{d+1} \times \R^{d+1} \, ,\quad a_{\eps,\gamma} (x,\theta,\xi,k) =
\sigma \left( \eps \, V(x,\theta),\xi+\dfrac{k \, \beta}{\eps},\gamma \right) \, ,
\end{equation}
where $\sigma \in {\bf S}^m({\mathcal O})$, $\langle \theta \rangle \, V$ belongs to the space ${\mathcal C}^n_b
(\R^{d+1})$ and where furthermore $V$ takes its values in a convex compact subset $K$ of ${\mathcal O}$ that
contains the origin (for instance $K$ can be a closed ball centered round the origin).
\end{definition}

To each symbol $a = (a_{\eps,\gamma})_{\eps \in ]0,1],\gamma \ge 1} \in S^m_n$ given by the formula
\eqref{singularsymbolp}, we associate a singular pseudodifferential operator $\opeg (a)$, with $\eps \in \, ]0,1]$
and $\gamma \ge 1$, whose action on a function $u \in {\mathcal S} (\R^{d+1} ; \C^N)$ is defined by
\begin{equation}
\label{singularpseudop}
\opeg (a) \, u \, (x,\theta) := \dfrac{1}{(2\, \pi)^{d+1}} \, \int_{\R^{d+1}} {\rm e}^{i\, (\xi \cdot x +k \, \theta)} \,
\sigma \left( \eps \, V(x,\theta),\xi+\dfrac{k \, \beta}{\eps},\gamma \right) \, \widehat{u} (\xi,k)
\, {\rm d}\xi \, {\rm d}k \, .
\end{equation}
Let us briefly note that for the Fourier multiplier $\sigma (v,\xi,\gamma) =i\, \xi_1$, the corresponding
singular operator is $\partial_{x_1} +(\beta_1/\eps) \, \partial_\theta$. The main difference with respect to
\eqref{singularpseudo} is that now the singular frequency $k/\eps$ lies in all $\R$ and not only in a discrete
set. This modification will not be negligeable in some places, especially when we compare the difference
between oscillatory integral operators and pseudodifferential operators. Following Part A, we wish to describe
the action of singular pseudodifferential operators on Sobolev spaces. The following result is a rather direct
consequence of Theorem \ref{thm6}.

\begin{proposition}
\label{prop13}
Let $n \ge d+1$, and let $a \in S^m_n$ with $m \le 0$. Then $\opeg (a)$ in \eqref{singularpseudop} defines
a bounded operator on $L^2 (\R^{d+1})$: there exists a constant $C>0$, that only depends on $\sigma$
and $V$ in the representation \eqref{singularsymbolp}, such that for all $\eps \in \, ]0,1]$ and for all
$\gamma \ge 1$, there holds
\begin{equation*}
\forall \, u \in {\mathcal S} (\R^{d+1}) \, ,\quad \left\| \opeg (a) \, u \right\|_0 \le \dfrac{C}{\gamma^{|m|}} \, \| u \|_0 \, .
\end{equation*}
\end{proposition}

The constant $C$ in Proposition \ref{prop13} depends uniformly on the compact set in which $V$ takes its
values and on the norm of $V$ in ${\mathcal C}^{d+1}_b$.

\begin{proof}[Proof of Proposition \ref{prop13}]
The situation is more or less the same as in Proposition \ref{prop3}. The only thing to observe is that if the
symbol $\sigma$ were independent of $V$, then we could not use Theorem \ref{thm6} because $\langle
\theta \rangle \, a_{\eps,\gamma}$ would not be bounded (except in the trivial case $a \equiv 0$). We thus
use the same trick as in \cite{williams3} and write
\begin{equation*}
\widetilde{a}_{\eps,\gamma} =\sigma \left( 0,\xi+\dfrac{k \, \beta}{\eps},\gamma \right)
+\sigma_\sharp \left( \eps \, V(x,\theta),\xi+\dfrac{k \, \beta}{\eps},\gamma \right) \cdot \eps \, V(x,\theta) \, .
\end{equation*}
Boundedness on $L^2$ for the Fourier multiplier is trivial, and as far as the second part is concerned, we can
apply Theorem \ref{thm3} (we could even apply the classical Calder\'on-Vaillancourt Theorem because taking
only one $k$-derivative is harmless for the second term since it contains an $\eps$ factor). We obtain a bound
of the form
\begin{equation*}
\forall \, \eps \in \, ]0,1] \, ,\quad \forall \, \gamma \ge 1 \, ,\quad
\Ng a_{\eps,\gamma} -\sigma \left( 0,\xi+\dfrac{k \, \beta}{\eps},\gamma \right) \Nd
\le \dfrac{C_{\sigma,V} \, \eps}{\gamma^{|m|}} \, .
\end{equation*}
\end{proof}

Let us again observe that the scaling $\eps \, V$ is not crucial here to obtain a uniform $L^2$ bound, as in
the case of wavetrains. This is because we still do not need to control any $k$-derivative. The analogue of
Proposition \ref{prop4} works in exactly the same way.

\begin{proposition}
\label{prop14}
Let $n \ge d+1$, and let $a \in S^m_n$ with $m>0$. Then $\opeg (a)$ in \eqref{singularpseudop} defines
a bounded operator from $H^{m,\eps}(\R^{d+1})$ to $L^2 (\R^{d+1})$: there exists a constant $C>0$, that
only depends on $\sigma$ and $V$ in the representation \eqref{singularsymbolp}, such that for all $\eps \in
\, ]0,1]$ and for all $\gamma \ge 1$, there holds
\begin{equation*}
\forall \, u \in {\mathcal S} (\R^{d+1}) \, ,\quad \left\| \opeg (a) \, u \right\|_0 \le C \, \| u \|_{H^{m,\eps},\gamma} \, .
\end{equation*}
\end{proposition}

\noindent There is also a smoothing effect in the case $m<0$ that is analogous to the one proved in Proposition
\ref{prop5} for wavetrains.

\begin{proposition}
\label{prop15}
Let $n \ge d+2$, and let $a \in S^{-1}_n$. Then $\opeg (a)$ in \eqref{singularpseudop} defines a bounded
operator from $L^2 (\R^{d+1})$ to $H^{1,\eps}(\R^{d+1})$: there exists a constant $C>0$, that only depends
on $\sigma$ and $V$ in the representation \eqref{singularsymbolp}, such that for all $\eps \in \, ]0,1]$ and for
all $\gamma \ge 1$, there holds
\begin{equation*}
\forall \, u \in {\mathcal S} (\R^{d+1}) \, ,\quad \left\| \opeg (a) \, u \right\|_{H^{1,\eps},\gamma} \le C \, \| u \|_0 \, .
\end{equation*}
\end{proposition}

If one compares with the proof of Proposition \ref{prop5}, the analogue of the term $T_3$ does not have
an additional $\eps$ factor but depends linearly on $\partial_\theta V$. Hence it falls into the framework of
Theorem \ref{thm6}, uniformly in $\eps$ and $\gamma$ (we use the decay in $\theta$ rather than controlling
one $k$-derivative).

We can extend the above results to singular amplitudes which are defined in the following way.

\begin{definition}[Singular amplitudes]
\label{def5}
Let $m \in \R$, and let $n \in \N$. Then we let $A^m_n$ denote the set of families of functions
$(\widetilde{a}_{\eps,\gamma})_{\eps \in ]0,1],\gamma \ge 1}$ that are constructed as follows:
\begin{multline}
\label{singularamplitudep}
\forall \, (x,\theta,y,\omega,\xi,k) \in \R^{d+1} \times \R^{d+1} \times \R^{d+1} \, ,\\
\widetilde{a}_{\eps,\gamma} (x,\theta,y,\omega,\xi,k) := \sigma \left(
\eps \, V(x,\theta),\eps \, W(y,\omega),\xi+\dfrac{k \, \beta}{\eps},\gamma \right) \, ,
\end{multline}
where $\sigma \in {\bf S}^m({\mathcal O}_1 \times {\mathcal O}_2)$, $\langle \theta \rangle \, V$ and $\langle
\omega \rangle \, W$ belong to the space ${\mathcal C}^n_b (\R^{d+1})$, and where furthermore $V$, resp. $W$,
takes its values in a convex compact subset $K_1$, resp. $K_2$, of ${\mathcal O}_1$, resp. ${\mathcal O}_2$, that
contains the origin.
\end{definition}

Our continuity results of Propositions \ref{prop6} and \ref{prop7} extend to the case of pulses. For $\widetilde{a}
\in A_n^m$, our goal is to define, whenever the formula makes sense, the singular oscillatory integral operator
acting on functions $u \in {\mathcal S}(\R^{d+1})$ as follows:
\begin{multline}
\label{singularopampp}
\opteg (\widetilde{a}) \, u (x,\theta) := \dfrac{1}{(2\, \pi)^{d+1}} \, \int_{\R^{d+1} \times \R^{d+1}}
{\rm e}^{i\, (\xi \cdot (x-y) +k \, (\theta-\omega))} \,
\sigma \left( \eps \, V(x,\theta),\eps \, W(y,\omega),\xi+\dfrac{k \, \beta}{\eps},\gamma \right) \\
u (y,\omega) \, {\rm d}\xi \, {\rm d}k \, {\rm d}y \, {\rm d}\omega \, .
\end{multline}
We have the following result for bounded amplitudes. (The integral in \eqref{singularopampp} has to be understood
as the limit in ${\mathcal S}'(\R^{d+1})$ of a truncation process in $(\xi,k)$.)

\begin{proposition}
\label{prop16}
Let $n \ge d+1$, and let $\widetilde{a} \in A^m_n$ with $m \le 0$. Then $\opteg (\widetilde{a})$ in
\eqref{singularopampp} defines a bounded operator on $L^2 (\R^{d+1})$: there exists a constant $C>0$, that
only depends on $\sigma$, $V$ and $W$ in the representation \eqref{singularamplitudep}, such that for all
$\eps \in \, ]0,1]$ and for all $\gamma \ge 1$, there holds
\begin{equation*}
\forall \, u \in {\mathcal S} (\R^{d+1}) \, ,\quad \left\| \opteg (\widetilde{a}) \, u \right\|_0 \le
\dfrac{C}{\gamma^{|m|}} \, \| u \|_0 \, .
\end{equation*}
\end{proposition}

\begin{proof}[Proof of Proposition \ref{prop16}]
We use the decomposition
\begin{equation*}
\widetilde{a}_{\eps,\gamma} =\sigma \left( \eps \, V(x,\theta),0,\xi+\dfrac{k \, \beta}{\eps},\gamma \right)
+\sigma_\sharp \left( \eps \, V(x,\theta),\eps \, W(y,\omega),\xi+\dfrac{k \, \beta}{\eps},\gamma \right) \cdot
\eps \, W(y,\omega) \, .
\end{equation*}
Boundedness on $L^2$ for the first term follows from Proposition \ref{prop13} since we deal here with a
singular pseudodifferential operator. As far as the second term is concerned, we can apply Theorem \ref{thm7}
because we have gained some decay in $\omega$ and taking one $k$-derivative is harmless due to the $\eps$
factor.
\end{proof}

\noindent The smoothing effect for amplitudes of degree $-1$ is more subtle.

\begin{proposition}
\label{prop17}
Let $n \ge d+2$, and let $\widetilde{a} \in A^{-1}_n$. Then $\opteg (\widetilde{a})$ in \eqref{singularopampp}
defines a bounded operator from $L^2 (\R^{d+1})$ into $H^{1,\eps}(\R^{d+1})$: there exists a constant $C>0$,
that only depends on $\sigma$, $V$ and $W$ in the representation \eqref{singularamplitudep}, such that for all
$\eps \in \, ]0,1]$ and for all $\gamma \ge 1$, there holds
\begin{equation*}
\forall \, u \in {\mathcal S} (\R^{d+1}) \, ,\quad \left\| \opteg (\widetilde{a}) \, u \right\|_{H^{1,\eps},\gamma}
\le C \, \| u \|_0 \, .
\end{equation*}
\end{proposition}

\begin{proof}[Proof of Proposition \ref{prop17}]
Proposition \ref{prop16} already gives a control of $\gamma$ times the $L^2$ norm so it only remains to
estimate the derivatives. Applying similar arguments as in the proof of Proposition \ref{prop7}, the derivative
$(\partial_{x_1} +\beta_1/\eps \, \partial_\theta) \, \opteg (\widetilde{a}) \, u$ is computed by differentiating
formally under the integral sign as long as the amplitude obtained by such differentiation defines a bounded
operator on$L^2$. We compute
\begin{equation*}
\left( \partial_{x_1} +\dfrac{\beta_1}{\eps} \, \partial_\theta \right) \opteg (\widetilde{a}) \, u = T_1 +T_2 +T_3 \, ,
\end{equation*}
where we use the notation
\begin{align*}
T_1 (x,\theta) &:= \dfrac{1}{(2\, \pi)^{d+1}} \, \int_{\R^{d+1} \times \R^{d+1}}
{\rm e}^{i \, (\xi \cdot (x-y)+k \, (\theta-\omega))} \\
& \qquad \qquad i \, \left( \xi_1+\dfrac{k \, \beta_1}{\eps} \right) \, \sigma \left(
\eps \, V(x,\theta),\eps \, W(y,\omega),\xi+\dfrac{k \, \beta}{\eps},\gamma \right)
\, u(y,\omega) \, {\rm d}\xi \, {\rm d}k \, {\rm d}y \, {\rm d}\omega \, ,\\
T_2 (x,\theta) &:= \dfrac{1}{(2\, \pi)^{d+1}} \, \int_{\R^{d+1} \times \R^{d+1}}
{\rm e}^{i \, (\xi \cdot (x-y)+k \, (\theta-\omega))} \\
& \qquad \qquad \left[ \partial_v \, \sigma \left(
\eps \, V(x,\theta),\eps \, W(y,\omega),\xi+\dfrac{k \, \beta}{\eps},\gamma \right) \cdot
\eps \, \partial_{x_1} \, V (x,\theta) \right] \, u(y,\omega) \, {\rm d}\xi \, {\rm d}k \, {\rm d}y \, {\rm d}\omega \, ,\\
T_3 (x,\theta) &:= \dfrac{\beta_1}{(2\, \pi)^{d+1}} \, \int_{\R^{d+1} \times \R^{d+1}}
{\rm e}^{i \, (\xi \cdot (x-y)+k \, (\theta-\omega))} \\
& \qquad \qquad \left[ \partial_v \, \sigma \left(
\eps \, V(x,\theta),\eps \, W(y,\omega),\xi+\dfrac{k \, \beta}{\eps},\gamma \right) \cdot
\partial_\theta \, V (x,\theta) \right] \, u(y,\omega) \, {\rm d}\xi \, {\rm d}k \, {\rm d}y \, {\rm d}\omega \, .
\end{align*}
The $L^2$ bound of the terms $T_1,T_2$ follows from Propositions \ref{prop16} since we deal with bounded
amplitudes of sufficient smoothness. The only term to consider with care is $T_3$. We decompose its amplitude
as follows:
\begin{multline*}
\partial_v \, \sigma \left( \eps \, V(x,\theta),\eps \, W(y,\omega),\xi+\dfrac{k \, \beta}{\eps},\gamma \right) \cdot
\partial_\theta \, V (x,\theta)
=\partial_v \, \sigma \left( \eps \, V(x,\theta),0,\xi+\dfrac{k \, \beta}{\eps},\gamma \right) \cdot
\partial_\theta \, V (x,\theta) \\
+\sigma_\flat \left( \eps \, V(x,\theta),\eps \, W(y,\omega),\xi+\dfrac{k \, \beta}{\eps},\gamma \right) \cdot
\left[ \eps \, W(y,\omega),\partial_\theta \, V (x,\theta) \right] \, ,
\end{multline*}
where $\sigma_\flat (v,w,\xi,\gamma)$ acts bilinearly on its arguments. The first term in the decomposition
is a pseudodifferential symbol that has decay in $\theta$ and for which we can use Theorem \ref{thm6} (the
decay in $\theta$ implies that we do not need to control $k$-derivatives). The second term in the decomposition
has decay in both $\theta$ and $\omega$ and it even has an $\eps$ factor. We can thus apply either Theorem
\ref{thm7} or \ref{thm8} and derive an $L^2$ bound that is uniform in the parameters $\eps$, $\gamma$.
\end{proof}

The argument of Lemma \ref{lem5} based on integration by parts still works for amplitudes of degree $1$,
and we have

\begin{lemma}
\label{lem6}
Let $\widetilde{a} \in A_n^1$, $n \ge 3\, (d+1)$. Then $\opteg (\widetilde{a})$ in \eqref{singularopampp}
is well-defined from ${\mathcal S}(\R^{d+1})$ into ${\mathcal S}'(\R^{d+1})$ as the limit of the operators
associated with the amplitude $\chi (\delta \, \xi,\delta \, k) \, \widetilde{a}$, $\chi \in {\mathcal C}^\infty_0
(\R^{d+1})$.
\end{lemma}

Singular oscillatory integral operators and singular pseudodifferential operators are closely linked. The results
below are the direct extensions of Theorems \ref{thm3} and \ref{thm4}. There is a new technical difficulty which
arises because the set of $\theta$-frequencies is no longer discrete and we thus really need to take derivatives
while we had to deal with finite differences in Part A. However, the general ideas in the proof are quite similar.

\begin{theorem}
\label{thm9}
Let $\widetilde{a} \in A_n^0$, $n \ge 2\, (d+1)$, be given by \eqref{singularamplitudep}, and let $a \in S_n^0$
be defined by
\begin{equation*}
\forall \, (x,\theta,\xi,k) \in \R^{d+1} \times \R^{d+1} \, ,\quad a_{\eps,\gamma} (x,\theta,\xi,k) :=
\sigma \left( \eps \, V(x,\theta),\eps \, W(x,\theta),\xi+\dfrac{k \, \beta}{\eps},\gamma \right) \, .
\end{equation*}
Then there exists a constant $C \ge 0$ such that for all $\eps \in \, ]0,1]$ and for all $\gamma \ge 1$, there holds
\begin{equation*}
\forall \, u \in {\mathcal S} (\R^{d+1}) \, ,\quad
\left\| \opteg (\widetilde{a}) \, u -\opeg (a) \, u \right\|_0 \le \dfrac{C}{\gamma} \, \| u \|_0 \, .
\end{equation*}

If $n \ge 3\, d +3$, then for another constant $C$, there holds
\begin{equation*}
\forall \, u \in {\mathcal S} (\R^{d+1}) \, ,\quad
\left\| \opteg (\widetilde{a}) \, u -\opeg (a) \, u \right\|_{H^{1,\eps},\gamma} \le C \, \| u \|_0 \, ,
\end{equation*}
uniformly in $\eps$ and $\gamma$.
\end{theorem}

The reason why we need $3\, d+3$ derivatives on the symbol for the smoothing effect (rather than $2\, d+3$ as
in Theorem \ref{thm3}) will be explained in the proof below.

\begin{proof}[Proof of Theorem \ref{thm9}]
Let us first observe that when the amplitude $\widetilde{a}$ does not depend on $(y,\omega)$, there is
no error in the difference $\opteg (\widetilde{a}) \, u -\opeg (a) \, u$, so we can restrict to the case where
$\widetilde{a}$ has the form
\begin{equation*}
\widetilde{a}_{\eps,\gamma} (x,\theta,y,\omega,\xi,k) := \sigma \left(
\eps \, V(x,\theta),\eps \, W(y,\omega),\xi+\dfrac{k \, \beta}{\eps},\gamma \right) \cdot \eps \, W(y,\omega) \, ,
\end{equation*}
where $\sigma (v,w_1,\xi,\gamma) \cdot w_2$ acts linearly on $w_2$.

$\bullet$ Following the ideas of Proposition \ref{prop2}, we can decompose the difference $\opteg (\widetilde{a})
-\opeg (a)$ as $\opteg (r_1) +\opteg (r_2)$, with $r_1$ as in Proposition \ref{prop2}, and
\begin{equation*}
r_2 := \dfrac{1}{i} \, \int_0^1 \partial_{\omega} \, \partial_k
\widetilde{a}_{\eps,\gamma}  \big( x,\theta,x,(1-t) \, \theta +t \, \omega,\xi,k \big) \, {\rm d}t \, .
\end{equation*}

The amplitude $r_1$ reads
\begin{multline*}
r_1 = \dfrac{1}{i} \, \sum_{j=1}^d \int_0^1 \sigma_j \left( \eps \, V(x,\theta),\eps \,
W((1-t) \, x +t\, y,\omega),\xi+\dfrac{k \, \beta}{\eps},\gamma \right) \cdot \eps \,
\partial_{y_j} W((1-t) \, x +t \, y,\omega) \, {\rm d}t \\
+\dfrac{1}{i} \, \sum_{j=1}^d \int_0^1 {\rm d}_w \sigma_j \left( \eps \, V(x,\theta),\eps \,
W((1-t) \, x +t\, y,\omega),\xi+\dfrac{k \, \beta}{\eps},\gamma \right) \cdot [\eps \, \partial_{y_j} W,
\eps \, W] ((1-t) \, x +t \, y,\omega) \, {\rm d}t \, ,
\end{multline*}
with $\sigma_j := \partial_{\xi_j} \, \sigma \in {\bf S}^{-1}$. To prove that $\optilde (r_1)$ is bounded on
$L^2$, we wish to apply Theorem \ref{thm7} and we thus try to control $\Ng r_1 \Nd_{\rm Amp}$. Since
we can use some decay in $\omega$, it is sufficient to control one derivative in $k$, $2\, d$ derivatives
in $\xi$, $d$ derivatives in $x$, $d$ derivatives in $y$, one derivative in $\theta$ and one derivatives in
$\omega$. The worst case occurs when $d$ derivatives in $x$ and $d$ derivatives in $y$ act on the
term $\partial_{y_j} W (\dots)$ and we thus need $W$ to have $2\, d+2$ derivatives in $L^\infty$ with
the weight $\langle \omega \rangle$. The factor $\eps$ allows for a uniform control of the $k$-derivative,
and we thus get a bound of the form
\begin{equation*}
\Ng r_1 \Nd_{\rm Amp} \le \dfrac{C}{\gamma} \, ,
\end{equation*}
with the quantity $\Ng r_1 \Nd_{\rm Amp}$ defined in Theorem \ref{thm7}.

$\bullet$ Let us now look at the operator $\opteg (r_2)$, which is more complicated. We compute
\begin{multline*}
r_2 = \dfrac{1}{i} \, \sum_{j=1}^d \int_0^1 \beta_j \, \sigma_j \left( \eps \, V(x,\theta),\eps \,
W(x,(1-t) \, \theta +t\, \omega),\xi+\dfrac{k \, \beta}{\eps},\gamma \right) \cdot
\partial_\omega W(x,(1-t) \, \theta +t \, \omega) \, {\rm d}t \\
+\dfrac{1}{i} \, \sum_{j=1}^d \int_0^1 \beta_j \, {\rm d}_w \sigma_j \left( \eps \, V(x,\theta),\eps \,
W(x,(1-t) \, \theta +t \, \omega),\xi+\dfrac{k \, \beta}{\eps},\gamma \right) \cdot [\eps \, \partial_\omega W,W]
(x,(1-t) \, \theta +t \, \omega) \, {\rm d}t \, ,
\end{multline*}
with again $\sigma_j := \partial_{\xi_j} \, \sigma \in {\bf S}^{-1}$. The "leading" part of the amplitude $r_2$ is
\begin{align*}
r_3 &:= \dfrac{1}{i} \, \sum_{j=1}^d \int_0^1 \beta_j \, \sigma_j \left( \eps \, V(x,\theta),
0,\xi+\dfrac{k \, \beta}{\eps},\gamma \right) \cdot \partial_\omega W(x,(1-t) \, \theta +t \, \omega) \, {\rm d}t \\
&= \dfrac{1}{i} \, \sum_{j=1}^d \beta_j \, \sigma_j \left( \eps \, V(x,\theta),0,\xi+\dfrac{k \, \beta}{\eps},\gamma \right)
\cdot \dfrac{W(x,\theta) -W(x,\omega)}{\theta -\omega} \, ,
\end{align*}
for $\theta \neq \omega$. Using separate estimates for $|\theta -\omega| \le 1$ or $|\theta -\omega| \ge 1$, we
obtain
\begin{equation*}
\sup_{\alpha,\beta \in \{ 0,1\}^{d+1}} \, \sup_{\nu \in \{ 0,1,2\}^d} \,
\left\| \langle \theta -\omega \rangle \, (\langle \theta \rangle^{-1} +\langle \omega \rangle^{-1})^{-1} \,
\partial_{x,\theta}^\alpha \, \partial_{y,\omega}^\beta \, \partial_\xi^\nu \, r_3
\right\|_{L^\infty (\R^{d+1} \times \R^{d+1} \times \R^{d+1})} \le \dfrac{C}{\gamma} \, ,
\end{equation*}
uniformly in the parameters $\eps$ and $\gamma$. (Observe that we cannot take any $k$-derivative because
there is no $\eps$ factor in the amplitude $r_3$, but we have decay in two directions of $\R^2$ that are either
$(\theta-\omega,\theta)$ or $(\theta-\omega,\omega)$.) We are not exactly in the framework of Theorem
\ref{thm8} (where the decay takes place in the $(\theta,\omega)$ directions) but we claim that the continuity
result of Theorem \ref{thm8} still holds if one replaces the weight $\langle \theta \rangle \, \langle \omega \rangle$
by the above weight. The only important point is to have a weight in two independent directions of the
$(\theta,\omega)$-plane. We can therefore conclude that the oscillatory integral operator $\optilde (r_3)$ is
bounded on $L^2$ with an operator norm controlled by $1/\gamma$.

$\bullet$ It remains to prove a bound in $L^2$ for the operator $\optilde (r_2 -r_3)$, and $r_2-r_3$ has the form
\begin{multline*}
\dfrac{1}{i} \, \sum_{j=1}^d \int_0^1 \beta_j \, \left( \sigma_j \left( \eps \, V(x,\theta),
\eps \, W(x,(1-t) \, \theta +t \, \omega),\xi+\dfrac{k \, \beta}{\eps},\gamma \right)
-\sigma_j \left( \eps \, V(x,\theta), 0,\xi+\dfrac{k \, \beta}{\eps},\gamma \right) \right) \\
\cdot \partial_\omega W(x,(1-t) \, \theta +t \, \omega) \, {\rm d}t \\
+\dfrac{1}{i} \, \sum_{j=1}^d \int_0^1 \beta_j \, {\rm d}_w \sigma_j \left( \eps \, V(x,\theta),\eps \,
W(x,(1-t) \, \theta +t \, \omega),\xi+\dfrac{k \, \beta}{\eps},\gamma \right) \cdot [\eps \, \partial_\omega W,W]
(x,(1-t) \, \theta +t \, \omega) \, {\rm d}t \, .
\end{multline*}
The $L^\infty$ norm of this quantity is uniformly controlled by
\begin{equation*}
\dfrac{C}{\gamma} \, \eps \, \int_0^1 \langle \theta +t\, (\omega-\theta) \rangle^{-2} \, {\rm d}t
\le \dfrac{C}{\gamma} \, \eps \, \langle \omega-\theta \rangle^{-1} \, .
\end{equation*}
The $1/\gamma$ factor comes from the fact that $\sigma_j$ belongs to ${\bf S}^{-1}$, and the exponent
$-2$ in the integrand comes from the fact that we have two functions that both have decay in their "fast"
variable. Since we have an $\eps$ factor available in the amplitude $r_2-r_3$, we can control one $k$-derivative
in $L^\infty$ just using the weight $\langle \omega-\theta \rangle$. Observe that we really need to control
one $k$-derivative because this term has decay in one single direction of the $(\theta,\omega)$-plane so
we are not able to use the same argument as for the amplitude $r_3$.

At this stage, we have seen that each piece in the decomposition of the oscillatory integral operator
$\opteg (\widetilde{a}) -\opeg (a)$ gives rise to a bounded operator on $L^2$ with operator norm
$O(1/\gamma)$. We thus obtain the first part of Theorem \ref{thm9}.

$\bullet$ In order to prove the smoothing effect, we need to control the first order singular derivatives of the
difference $\opteg (\widetilde{a}) \, u -\opeg (a) \, u$. If we stick to the above decomposition $r_1+r_2$, the
piece $r_3$ will be differentiated with respect to $\theta$ and multiplied by $1/\eps$. There will then be little
chance to obtain a uniform control of this piece since we had no $\eps$ factor there. We therefore use
another decomposition of the amplitude and write
\begin{equation*}
\opteg (\widetilde{a}) \, u -\opeg (a) \, u =\opteg (r_1) \, u +\opteg (r_{2,\sharp}) \, u \, ,
\end{equation*}
where $r_1$ is the same amplitude as above, and $r_{2,\sharp}$ denotes the amplitude
\begin{equation}
\label{amplituder2diese}
r_{2,\sharp} := \sigma \left(
\eps \, V(x,\theta),\eps \, W(x,\omega),\xi+\dfrac{k \, \beta}{\eps},\gamma \right) \cdot \eps \, W(x,\omega)
-\sigma \left(
\eps \, V(x,\theta),\eps \, W(x,\theta),\xi+\dfrac{k \, \beta}{\eps},\gamma \right) \cdot \eps \, W(x,\theta) \, ,
\end{equation}
which each expression on the right-hand side has degree $0$ with respect to the frequencies (observe that
here we have not applied Taylor's formula and integration by parts to get some decay in the frequency variables).

The singular derivatives of the term $\opteg (r_1) \, u$ are computed according to the formula
\begin{equation*}
\left( \partial_{x_1} +\dfrac{\beta_1}{\eps} \right) \, \opteg (r_1) \, u
=\opteg \left( i\, \left( \xi_1+\dfrac{k\, \beta_1}{\eps} \right) \, r_1 \right) \, u +\opteg (\partial_{x_1} r_1) \, u
+\dfrac{\beta_1}{\eps} \, \opteg (\partial_\theta r_1) \, u \, .
\end{equation*}
Estimating each term in the above decomposition follows from arguments that were already used above.
In particular, there is no problem here with the $\eps$ factors since the $\theta$ derivative on $r_1$ yields
an additional $\eps$ factor, and it also yields some decay in the $\theta$-direction. We can therefore prove
a uniform $L^2$ bound for the singular derivatives of $\opteg (r_1) \, u$ as long as the regularity $n$ of the
functions $V,W$ in the amplitude satisfies $n \ge 2\, d+3$ (compare with Theorem \ref{thm3}).

Let us now look at the singular derivative of the term $\opteg (r_{2,\sharp}) \, u$:
\begin{equation*}
\left( \partial_{x_1} +\dfrac{\beta_1}{\eps} \right) \, \opteg (r_{2,\sharp}) \, u
=\opteg \left( i\, \left( \xi_1+\dfrac{k\, \beta_1}{\eps} \right) \, r_{2,\sharp} \right) \, u
+\opteg (\partial_{x_1} r_{2,\sharp}) \, u +\dfrac{\beta_1}{\eps} \, \opteg (\partial_\theta r_{2,\sharp}) \, u \, .
\end{equation*}
There is a subtletly here because the first amplitude on the right-hand side has degree $+1$ with repect
to the frequencies, and this is the reason why we need $n \ge 3\, d+3$ in Theorem \ref{thm9} (in order to
give a meaning to this quantity). For this first term, we use the Taylor formula and integrate by parts to get
\begin{multline*}
\opteg \left( i\, \left( \xi_1+\dfrac{k\, \beta_1}{\eps} \right) \, r_{2,\sharp} \right) \, u =\dfrac{1}{(2\, \pi)^{d+1}}
\, \int_{\R^{d+1} \times \R^{d+1}} {\rm e}^{i\, (\xi \cdot (x-y) +k \, (\theta-\omega))} \\
\left( \int_0^1 \partial_\omega \, \partial_k \, b_{\eps,\gamma} (x,\theta,x,(1-t) \, \theta +t\, \omega,\xi,k) \,{\rm d}t
\right) \, u (y,\omega) \, {\rm d}\xi \, {\rm d}k \, {\rm d}y \, {\rm d}\omega \, ,
\end{multline*}
with
\begin{equation*}
b_{\eps,\gamma} (x,\theta,y,\omega,\xi,k) := i\, \left( \xi_1+\dfrac{k\, \beta_1}{\eps} \right) \, \sigma \left(
\eps \, V(x,\theta),\eps \, W(y,\omega),\xi+\dfrac{k \, \beta}{\eps},\gamma \right) \cdot \eps \, W(y,\omega) \, .
\end{equation*}
Since $b_{\eps,\gamma}$ has degree $+1$, its $k$-derivative has degree $0$. More precisely, we can check
that all terms arising when computing the derivative $\partial_\omega \, \partial_k \, b_{\eps,\gamma}$ yield an
oscillatory integral operator that is bounded on $L^2$ for $n \ge 3\, d+3$.

The terms $\opteg (\partial_{x_1} r_{2,\sharp}) \, u$ and $\beta_1/\eps \, \opteg (\partial_\theta r_{2,\sharp})
\, u$ are estimated by using the expression \eqref{amplituder2diese}. Let us observe that the second term in
the right-hand side of \eqref{amplituder2diese} is independent of $(y,\omega)$ so it gives rise to a genuine
pseudodifferential operator (for which the continuity criterion of Theorem \ref{thm6} is less restrictive than the
analogous result for oscillatory integral operators). Eventually, the interested reader can check that, using
either Theorem \ref{thm6} or Theorem \ref{thm8}, all amplitudes arising when computing $\partial_{x_1}
r_{2,\sharp}$ and $\beta_1/\eps \, \partial_\theta r_{2,\sharp}$ define oscillatory integral operators that are
bounded on $L^2$ and whose operator norm is $O(1)$ uniformly in $\eps,\gamma$. We feel free at this
stage to shorten the details that are very similar to many of the above arguments.
\end{proof}

In the same spirit as Theorem \ref{thm4}, we have the following result in the case of pulses.

\begin{theorem}
\label{thm10}
Let $\widetilde{a} \in A_n^1$, $n \ge 3\, d +4$, be given by \eqref{singularamplitudep}, and let $a \in S_n^1$
be defined by
\begin{equation*}
\forall \, (x,\theta,\xi,k) \in \R^{d+1} \times \R^{d+1} \, ,\quad a_{\eps,\gamma} (x,\theta,\xi,k)
:= \sigma \left( \eps \, V(x,\theta),\eps \, W(x,\theta),\xi+\dfrac{k \, \beta}{\eps},\gamma \right) \, .
\end{equation*}
Then the operator $\opteg (\widetilde{a}) -\opeg (a)$ is bounded on $L^2$, namely there exists a constant
$C \ge 0$ such that for all $\eps \in \, ]0,1]$ and for all $\gamma \ge 1$, there holds
\begin{equation*}
\forall \, u \in {\mathcal S} (\R^{d+1}) \, ,\quad
\left\| \opteg (\widetilde{a}) \, u -\opeg (a) \, u \right\|_0 \le C \, \| u \|_0 \, .
\end{equation*}
In particular, $\opteg (\widetilde{a})$ maps $H^{1,\eps}$ into $L^2$ and there exists a constant $C \ge 0$
such that for all $\eps \in \, ]0,1]$ and for all $\gamma \ge 1$, there holds
\begin{equation*}
\forall \, u \in {\mathcal S} (\R^{d+1}) \, ,\quad
\left\| \opteg (\widetilde{a}) \, u \right\|_0 \le C \, \| u \|_{H^{1,\eps},\gamma} \, .
\end{equation*}
\end{theorem}

\noindent The proof is very similar to that of Theorem \ref{thm4}, with suitable modifications as in Theorem
\ref{thm9} in order to get some decay in the fast variables $\theta$ and/or $\omega$.

\section{Singular pseudodifferential calculus II. Adjoints and products}
\label{sect8}

The same results as in Section \ref{sect5} hold in the context of pulses. We just state the corresponding
results without proof in view of a future application to nonlinear geometric optics problems. The two first
results deal with adjoints of singular pseudodifferential operators while the last two deal with products.

\begin{proposition}
\label{prop18}
Let $a \in S_n^0$, $n \ge 2\, (d+1)$, and let $a^*$ denote the conjugate transpose of the symbol $a$.
Then $\opeg (a)$ and $\opeg (a^*)$ act boundedly on $L^2$ and there exists a constant $C \ge 0$ such
that for all $\eps \in \, ]0,1]$ and for all $\gamma \ge 1$, there holds
\begin{equation*}
\forall \, u \in {\mathcal S} (\R^{d+1}) \, ,\quad
\left\| \opeg (a)^* \, u -\opeg (a^*) \, u \right\|_0 \le \dfrac{C}{\gamma} \, \| u \|_0 \, .
\end{equation*}

If $n \ge 3\, d +3$, then for another constant $C$, there holds
\begin{equation*}
\forall \, u \in {\mathcal S} (\R^{d+1}) \, ,\quad
\left\| \opeg (a)^* \, u -\opeg (a^*) \, u \right\|_{H^{1,\eps},\gamma} \le C \, \| u \|_0 \, ,
\end{equation*}
uniformly in $\eps$ and $\gamma$.
\end{proposition}

\begin{proposition}
\label{prop19}
Let $a \in S_n^1$, $n \ge 3\, d +4$, and let $a^*$ denote the conjugate transpose of the symbol $a$.
Then $\opeg (a)$ and $\opeg (a^*)$ map $H^{1,\eps}$ into $L^2$ and there exists a family of operators
$R^{\eps,\gamma}$ that satisfies
\begin{itemize}
 \item there exists a constant $C \ge 0$ such that for all $\eps \in \, ]0,1]$ and for all
      $\gamma \ge 1$, there holds
\begin{equation*}
\forall \, u \in {\mathcal S} (\R^{d+1}) \, ,\quad
\left\| R^{\eps,\gamma} \, u \right\|_0 \le C \, \| u \|_0 \, ,
\end{equation*}

 \item the following duality property holds
\begin{equation*}
\forall \, u,v \in {\mathcal S} (\R^{d+1}) \, ,\quad
\langle \opeg (a) \, u,v \rangle_{L^2} -\langle u,\opeg (a^*) \, v \rangle_{L^2} =\langle
R^{\eps,\gamma} \, u,v \rangle_{L^2} \, .
\end{equation*}
In particular, the adjoint $\opeg (a)^*$ for the $L^2$ scalar product maps $H^{1,\eps}$ into $L^2$.
\end{itemize}
\end{proposition}

\begin{proposition}
\label{prop20}
Let $a,b \in S_n^0$, $n \ge 2\, (d+1)$. Then there exists a constant $C \ge 0$ such that for all
$\eps \in \, ]0,1]$ and for all $\gamma \ge 1$, there holds
\begin{equation*}
\forall \, u \in {\mathcal S} (\R^{d+1}) \, ,\quad
\left\| \opeg (a) \, \opeg (b) \, u -\opeg (a \, b) \, u \right\|_0 \le \dfrac{C}{\gamma} \, \| u \|_0 \, .
\end{equation*}
If $n \ge 3\, d +3$, then for another constant $C$, there holds
\begin{equation*}
\forall \, u \in {\mathcal S} (\R^{d+1}) \, ,\quad
\left\| \opeg (a) \, \opeg (b) \, u -\opeg (a \, b) \, u \right\|_{H^{1,\eps},\gamma} \le C \, \| u \|_0 \, ,
\end{equation*}
uniformly in $\eps$ and $\gamma$.

Let $a \in S_n^1,b \in S_n^0$ or $a \in S_n^0,b \in S_n^1$, $n \ge 3\, d +4$. Then there exists a constant
$C \ge 0$ such that for all $\eps \in \, ]0,1]$ and for all $\gamma \ge 1$, there holds
\begin{equation*}
\forall \, u \in {\mathcal S} (\R^{d+1}) \, ,\quad
\left\| \opeg (a) \, \opeg (b) \, u -\opeg (a \, b) \, u \right\|_0 \le C \, \| u \|_0 \, .
\end{equation*}
\end{proposition}

\begin{proposition}
\label{prop21}
Let $a \in S_n^{-1},b \in S_n^1$, $n \ge 3\, d +4$. Then $\opeg (a) \, \opeg (b)$ defines a bounded operator on
$H{1,\eps}$ and there exists a constant $C \ge 0$ such that for all $\eps \in \, ]0,1]$ and for all $\gamma \ge 1$,
there holds
\begin{equation*}
\forall \, u \in {\mathcal S} (\R^{d+1}) \, ,\quad
\left\| \opeg (a) \, \opeg (b) \, u -\opeg (a \, b) \, u \right\|_{H^{1,\eps},\gamma} \le C \, \| u \|_0 \, .
\end{equation*}
\end{proposition}

\noindent Our final result is G{\aa}rding's inequality.

\begin{theorem}
\label{thm11}
Let $\sigma \in {\bf S}^0$ satisfy $\text{\rm Re} \, \sigma (v,\xi,\gamma) \ge C_K>0$ for all $v$ in a compact
subset $K$ of ${\mathcal O}$. Let now $a \in S_0^n$, $n \ge 2\, d+2$ be given by \eqref{singularsymbolp}, where
$V$ is valued in a convex compact subset $K$. Then for all $\delta >0$, there exists $\gamma_0$ which depends
uniformly on $V$, the constant $C_K$ and $\delta$, such that for all $\gamma \ge \gamma_0$ and all $u \in
{\mathcal S}(\R^{d+1})$, there holds
\begin{equation*}
\text{\rm Re } \langle \opeg (a) \, u ;u \rangle_{L^2} \ge (C-\delta) \, \| u \|_0^2 \, .
\end{equation*}
\end{theorem}

There is of course an extended version of the singular calculus that allows for pseudodifferential cut-offs just
as in the wavetrains case.

\bibliographystyle{plain}
\bibliography{CGW2}

\begin{thebibliography}{10}

\bibitem{bourdaudmeyer}
G.~Bourdaud, Y.~Meyer.
\newblock In\'egalit\'es {$L^2$} pr\'ecis\'ees pour la classe {$S^0_{0,0}$}.
\newblock {\em Bull. Soc. Math. France}, 116(4):401--412 (1989), 1988.

\bibitem{calderonvaillancourt}
A.~Calder{\'o}n, R.~Vaillancourt.
\newblock On the boundedness of pseudo-differential operators.
\newblock {\em J. Math. Soc. Japan}, 23:374--378, 1971.

\bibitem{coifmanmeyer}
R.~R. Coifman, Y.~Meyer.
\newblock {\em Au del\`a des op\'erateurs pseudo-diff\'erentiels}, volume~57 of
  {\em Ast\'erisque}.
\newblock Soci\'et\'e Math\'ematique de France, 1978.

\bibitem{cordes}
H.~O. Cordes.
\newblock On compactness of commutators of multiplications and convolutions,
  and boundedness of pseudodifferential operators.
\newblock {\em J. Funct. Anal.}, 18:115--131, 1975.

\bibitem{jfc}
J.-F. Coulombel.
\newblock Weakly stable multidimensional shocks.
\newblock {\em Ann. Inst. H. Poincar\'e Anal. Non Lin\'eaire}, 21(4):401--443,
  2004.

\bibitem{cgw3}
J.-F. Coulombel, O.~Gu\`es, and M.~Williams.
\newblock Semilinear geometric optics with boundary amplification.
\newblock {\em Preprint}, 2012.

\bibitem{jfcog}
J.-F. Coulombel, O.~Gu{\`e}s.
\newblock Geometric optics expansions with amplification for hyperbolic
  boundary value problems: linear problems.
\newblock {\em Ann. Inst. Fourier (Grenoble)}, 60(6):2183--2233, 2010.

\bibitem{delortszeftel}
J.-M. Delort, J.~Szeftel.
\newblock Long-time existence for small data nonlinear {K}lein-{G}ordon
  equations on tori and spheres.
\newblock {\em Int. Math. Res. Not.}, 37:1897--1966, 2004.

\bibitem{hwang}
I.~L. Hwang.
\newblock The {$L^2$}-boundedness of pseudodifferential operators.
\newblock {\em Trans. Amer. Math. Soc.}, 302(1):55--76, 1987.

\bibitem{jmr}
J.-L. Joly, G.~M{\'e}tivier, and J.~Rauch.
\newblock Coherent and focusing multidimensional nonlinear geometric optics.
\newblock {\em Ann. Sci. \'Ecole Norm. Sup. (4)}, 28(1):51--113, 1995.

\bibitem{metivier3}
G.~M{\'e}tivier.
\newblock {\em Para-differential calculus and applications to the {C}auchy
  problem for nonlinear systems}, volume~5 of {\em Centro di Ricerca Matematica
  Ennio De Giorgi (CRM) Series}.
\newblock Edizioni della Normale, Pisa, 2008.

\bibitem{williams3}
M.~Williams.
\newblock Singular pseudodifferential operators, symmetrizers, and oscillatory
  multidimensional shocks.
\newblock {\em J. Funct. Anal.}, 191(1):132--209, 2002.

\end{thebibliography}
\end{document}